\documentclass[notitlepage]{article}

\usepackage{graphicx}
\usepackage{blindtext}
\usepackage[colorlinks=true, linkcolor=black, citecolor=black, urlcolor=black]{hyperref}
\usepackage[margin=1in]{geometry}

\usepackage{setspace}
\usepackage{ dsfont }
\usepackage[utf8]{inputenc}
\usepackage{latexsym,amssymb}
\usepackage{hyperref}
\usepackage{authblk}

\usepackage{amsthm}
\usepackage{amsmath}
\usepackage{tikz-cd,bbm,amsmath,amssymb}
\usepackage{graphicx}
\usepackage{pst-all}

\usepackage{bm}
\usepackage{caption}

\usepackage{enumitem}
\usepackage{float}

\usepackage{lipsum}

\newcommand\blfootnote[1]{%
  \begingroup
  \renewcommand\thefootnote{}\footnote{#1}%
  \addtocounter{footnote}{-1}%
  \endgroup
}

\newtheorem{theorem}{Theorem}[section]
\newtheorem{corollary}[theorem]{Corollary}
\newtheorem{lemma}[theorem]{Lemma}
\newtheorem{definition}[theorem]{Definition}

\newtheorem{proposition}[theorem]{Proposition}
\newtheorem{conjecture}[theorem]{Conjecture}

\newtheorem{remark}[theorem]{Remark}

\newtheorem{introtheorem}{Theorem}
\newtheorem{introdefinition}{Definition}

\begin{document}
\title{Arithmetic Geometric Model for the Renormalisation of Bi-Critical Irrationally Indifferent Attractors}

\author{Jocelyn Finbar Russell}


\maketitle
\begin{abstract}
In this paper we build a geometric model for the renormalisation of irrationally indifferent fixed points of holomorphic maps  with two critical points. The model incorporates arithmetic
properties of the rotation number at the fixed point, as well as the ``angle" between the two critical points. Using this model for the renormalisation, we create a topological model for the local dynamics of such maps. We also explain the topology of the maximal invariant
set for the model, and the dynamics of the map on the maximal invariant
set.
\end{abstract}
\nocite{*}

\section{Introduction}
\blfootnote{\textit{Mathematics Subject Classification:} Primary 37F50; Secondary 37F35, 37F10\\
\textit{Keywords:} Irrationally indifferent attractors, post-critical set, Siegel disks, Cremer points, renormalisation}
Let $f$ be a non-Möbius rational function of $\overline{\mathds{C}}$ such that $f(0)=0$ and $f'(0) = e^{2\pi i \alpha}$ with $\alpha \in \mathds{R}\setminus\mathds{Q}$. The fixed point at $0$ is called \textbf{irrationally indifferent} \cite{milnor}. Iterating $f$ near $0$ exhibits tangential movements due to the irrational rotation, and radial movements due to the non-linearity of the map. It has been well established that such maps exhibit non-trivial behaviour which depends on the arithmetic of $\alpha$, for instance in \cite{siegelref, Cremer1928,MR1367353, mcmullensiegel,perezmarco,MR2051390,positivearea,MR2819165,cheraghi2022topology} and more, but to this day a complete understanding of the dynamics of orbits which remain near $0$, particularly in the cases when multiple critical points come into play, has remained out of reach. 
\par
The map, $f$, is called \textit{linearisable} if there is a neighbourhood of $0$ for which $f$ is conformally conjugate to the rotation $z\mapsto e^{2\pi i \alpha}z$. In such cases we denote $\Delta_f\supsetneq\{0\}$, known as the \textbf{Siegel Disk}, to be the maximal connected open set upon which this conjugacy extends. Many examples, however, fail this condition and thus cannot be explained by rotation maps, in which case we say $0$ is a \textbf{Cremer Fixed Point}, and define $\Delta_f := \{0\}$. Indeed, even in the Siegel case there are examples for which orbits that remain in a bounded region containing $0$ but do not enter the Siegel Disk exhibit extremely non-trivial dynamics (for instance inducing a one-sided hairy Jordan curve in \cite{cheraghi2022topology}). 
\par
In particular, it is a consequence of results by Fatou and Mane \cite{fatou,manee} that there exists a recurrent critical point $c \in \mathds{C}$ of $f$ which satisfies
\begin{equation*}
    \omega(c):= \overline{\cup_{i=1}^{\infty}\{f^{\circ i}(c)\}} \supseteq \partial\Delta_f.
\end{equation*}
The set, $\omega(c)\cup\Delta_f$, is the largest compact, connected, and invariant set containing $0$ that exhibits ``rotation-like" dynamics, and furthermore by Lyubich in \cite{lyubich}, $\omega(c)$ will be a measure-theoretic attractor for the orbits of unstable points that remain near $0$. For these reasons, understanding the dynamics on $\omega(c)$ is a crucial first step for extending the basic linearisability model.
\par
It is exactly this gap which is the target of arithmetic-geometric models, first introduced and fleshed out by Cheraghi in \cite{cheraghi2022arithmetic}. These classes of maps, depending purely on fixed combinatorial data, such as the rotation number, provide the first examples of model maps that are conjugate to the dynamics on the larger domain $\omega(c)\cup\Delta_f$. Following Cheraghi in \cite{cheraghi2022topology}, we call $\omega(c)$ the \textbf{irrationally indifferent attractor}.
\par
Cheraghi's work focused on the uni-critical case, but it is a general feature of dynamical systems that the presence of more critical points can significantly complicate dynamical properties. We consider the bi-critical case, which already exhibits significant differences:
\begin{introdefinition}
\label{bicritdef}
    Let $f$ be a rational map of degree $d\geq3$ with an irrationally indifferent fixed point at $0$. We say $f$ has a \textbf{bi-critical irrationally indifferent attractor} at $0$, if there exist at most two \textbf{recurrent} critical points of $f$, $c_1$ and $c_2$, satisfying the following:
    \begin{itemize}
        \item if $c_1=c_2$ then the local degree at the critical point is three, otherwise both are simple critical points,
        \item we have $c_1\in\omega(c_2)$ and $c_2\in\omega(c_1)$,
        \item and $\omega(c_1)=\omega(c_2)\supset\partial\Delta_f $.

    \end{itemize}
    
\end{introdefinition}
Note that a rational function with a bi-critical irrationally indifferent attractor may have many other critical points elsewhere in the dynamical plane, even with non-recurrent critical points whose orbit accumulates on $\partial\Delta_f$.
\par
One of the main collections of tools for studying the dynamics of $f$ is renormalisation methods. Roughly speaking, the renormalisation operator is the first return map to a sector at $0$, viewed through a suitable change of coordinate. Due to the nature of the change of coordinate, such renormalisation can be unwieldy, however understanding these complications has historically proved fruitful. It was first used by Yoccoz in \cite{yoccozfirst}, who employed the sector renormalisation of Douady-Ghys (\cite{MR936853}) to prove the optimality of the arithmetic Brjuno condition for linearising quadratic polynomials. 
\par
The later development of the powerful near-parabolic renormalisation tool for the (uni-critical) Inou-Shishikura class in \cite{inoushish} led to an explosion of new developments. Examples include \cite{positivearea}, in which it led exhibiting the first ever positive measure Julia set for a polynomial, in particular a quadratic with a Cremer fixed point, \cite{MR3418243} in which it led to solving the longstanding Marmi-Moussa-Yoccoz conjecture for rotation numbers of high type, \cite{MR3854897} in which unique ergodicity on the irrationally indifferent attractor for high type irrationally indifferent quadratics was proven, \cite{MR4956555} in which it was proven that some quadratic Siegel disks are Jordan domains, and finally \cite{cheraghi2022arithmetic,cheraghi2022topology}, where a remarkable topological classification of $\omega(c)$ in the uni-critical case was found. 
\par
An essential ingredient for that topological classification was Cheraghi's aforementioned model for the renormalisation of unicritical irrationally indifferent attractors, initially started in \cite{MR3079339,MR3940907}, and fully fleshed out in \cite{cheraghi2022arithmetic}. In this unicritical case, the contribution of the tangential and radial movements near $0$ culminates in a geometric behaviour which may be simplified in an algebraic form using the following function defined on $x\in\mathds{R}$ for $\alpha\in[-1/2,1/2]\setminus\{0\}$: 

\begin{equation}
    \label{dyncurvem}
    \begin{aligned}
        Q_{\alpha}(x) := \frac{1}{1 + d(x,1/\alpha\mathds{Z})}.
    \end{aligned}
\end{equation}

This corresponds to the geometry of orbits before one return map. Then Cheraghi's model consists of a class of maps,
    $$
    \mathds{F}_1 = \{ \mathds{T}_{[\alpha]}: \mathds{M}_{[\alpha]} \rightarrow \mathds{M}_{[\alpha]} \}_{[\alpha]\in(\mathds{R}\setminus\mathds{Q})/\mathds{Z}}.
    $$
    and a renormalisation operator, $\mathcal{R}$, which preserves $\mathds{F}_1$, and, for $\alpha\in(-1/2,1/2)\setminus\mathds{Q}$,
    $$
        \mathcal{R}(\mathds{T}_{[\alpha]}: \mathds{M}_{[\alpha]} \rightarrow \mathds{M}_{[\alpha]})= (\mathds{T}_{[-1/\alpha]}: \mathds{M}_{[-1/\alpha]} \rightarrow \mathds{M}_{[-1/\alpha]}),
        $$
    ($[x]:=x+\mathds{Z}\in\mathds{R}\setminus\mathds{Z}$ for $x\in\mathds{R}$). In this model, $\mathds{M}_{[\alpha]}$ is compact, connected, and star-like with $\{0,+1\}\subset 
\mathds{M}_{[\alpha]}$, and $\mathds{T}_{[\alpha]}$ is a homeomorphism rotating by $2\pi \alpha$ in the tangential direction, while radially there exists a constant $C>0$ such that
for $\alpha\in(-1/2,1/2)\setminus\mathds{Q}$ and $0\leq k\leq 1/|\alpha|$ we have
        $$
        C^{-1}Q_{\alpha}(k)<|\mathds{T}^{\circ k}_{[\alpha]}(+1)|<CQ_{\alpha}(k).
        $$
When a single critical point is associated to the fixed point at 0, it is expected that the map $\pi: f \mapsto \mathds{T}_{[\log f'(0)/(2\pi i)]}$ conjugates a renormalisation scheme for such holomorphic maps to $\mathcal{R}:\mathds{F}_1\rightarrow \mathds{F}_1$. Using near-parabolic renormalisation defined in \cite{inoushish}, Cheraghi proved this for $\alpha$ of \textit{high type} in \cite{cheraghi2022topology}. In particular this induces a conjugacy for the dynamics of $f:\omega(c)\rightarrow\omega(c)$ to $\mathds{T}_{[\alpha]}:\partial\mathds{M}_{[\alpha]}\rightarrow\partial \mathds{M}_{[\alpha]}$.
\par
For bi-critical irrationally indifferent attractors the non-linearity is fundamentally different, and the arithmetic of an additional parameter, $[\iota]\in\mathds{R}/\mathds{Z}$, comes into play. This $\iota$ value is the internal angle between $c_1$ and $c_2$, which is determined by the dynamics of $f$. For instance, when $c_1, c_2\in\partial\Delta_f$, $\iota$ is the angle between internal rays of $\Delta_f$ accumulating at $c_1$ and $c_2$. This geometric dependence on $\iota$ can be observed from the computed dynamics in Figure \ref{dyncomputation}.
\begin{figure}[H]
\captionsetup{justification=centering}
    \centering
\includegraphics[width=0.7\textwidth]{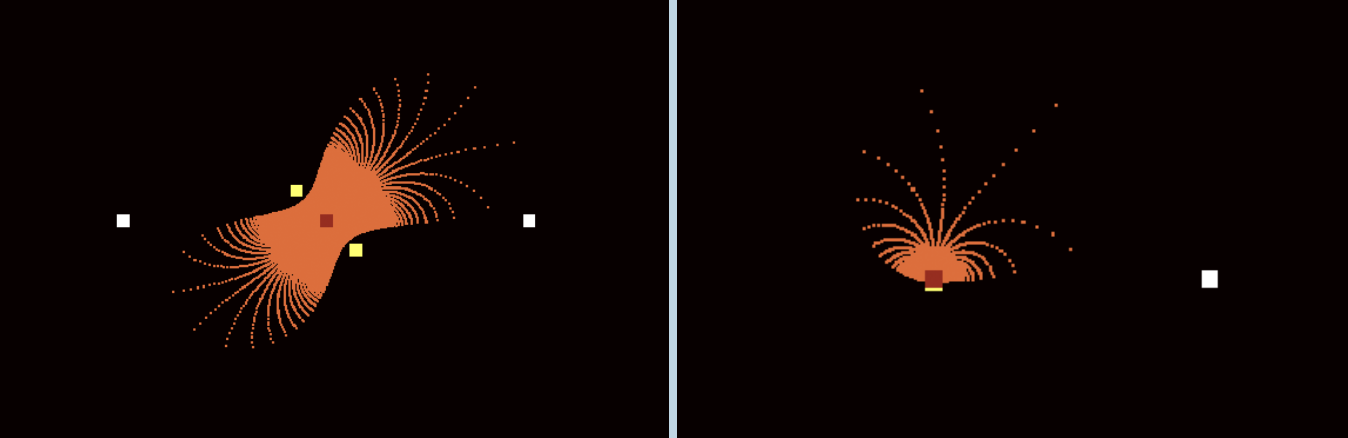}

    \caption{\label{dyncomputation}Computations describing the Geometry of Fatou Coordinates for bi-critical cubics with $\alpha\asymp0.001$ and $\iota=1/2$ on the left, and $\iota=0$ with the same $\alpha$ on the right. The white points are the critical points, the yellow repelling fixed points, red the fixed point at $0$, and orange represents the forward $\lfloor1/|\alpha|\rfloor$ images of a ray coming out of $0$ to one of the nearby repelling fixed points.}
\end{figure}
\par
Here we present a first ever model for the bi-critical case:
\begin{introtheorem}
\label{mainA}
    There exists a class of maps:
    $$
    \mathds{F}_2 = \{ \mathds{T}_{[\alpha],[\iota]}: \mathds{M}_{[\alpha],[\iota]} \rightarrow \mathds{M}_{[\alpha],[\iota]} \}_{[\alpha]\in(\mathds{R}\setminus\mathds{Q})/\mathds{Z}, [\iota]\in\mathds{R}/\mathds{Z}}
    $$
    and a renormalisation operator $\mathcal{R}:\mathds{F}_2 \rightarrow \mathds{F}_2$ satisfying the following properties:
    \begin{enumerate}[label=(\roman*)]
        \item $\mathds{M}_{[\alpha],[\iota]} \subset \mathds{C}$ is a compact star-like set with $\{0,+1,e^{2\pi i \iota}\} \subset \mathds{M}_{[\alpha],[\iota]}$,
        \item $\mathds{T}_{[\alpha],[\iota]}: \mathds{M}_{[\alpha],[\iota]} \rightarrow \mathds{M}_{[\alpha],[\iota]}$ is a homeomorphism which acts as rotation by $2\pi\alpha$ in the tangential direction,
        \item 
        for all $\alpha\in [-1/2,1/2]\setminus\mathds{Q}$, and $\iota\in[-1/2,1/2]$:
        $$
        \mathcal{R}(\mathds{T}_{[\alpha],[\iota]}: \mathds{M}_{[\alpha],[\iota]} \rightarrow \mathds{M}_{[\alpha],[\iota]})= (\mathds{T}_{[-1/\alpha],[-\iota/\alpha]}: \mathds{M}_{[-1/\alpha],[-\iota/\alpha]} \rightarrow \mathds{M}_{[-1/\alpha],[-\iota/\alpha]}),
        $$
        \item recall (\ref{dyncurvem}), there exists a constant $C>0$ such that for all $\alpha\in [-1/2,1/2]\setminus\mathds{Q}$, and $\iota\in[-1/2,1/2]$, and for all $0\leq k\leq 1/|\alpha|$ we have
        $$
        C^{-1}\sqrt{\frac{|\iota|+|\alpha|}{|\alpha|}Q_{\alpha}(k)Q_{\alpha}(k-\iota/\alpha)}<|\mathds{T}^{\circ k}_{[\alpha],[\iota]}(+1)|<C\sqrt{\frac{|\iota|+|\alpha|}{|\alpha|}Q_{\alpha}(k)Q_{\alpha}(k-\iota/\alpha)}.
        $$
        
    \end{enumerate}
    
\end{introtheorem}
The geometry of one return for an orbit depends on two parameters that lie on the torus, $T^2$. We schematically illustrate this dependence in Figure \ref{torus}:
\begin{figure}[H]
\captionsetup{justification=centering}
    \centering
\includegraphics[width=0.7\textwidth]{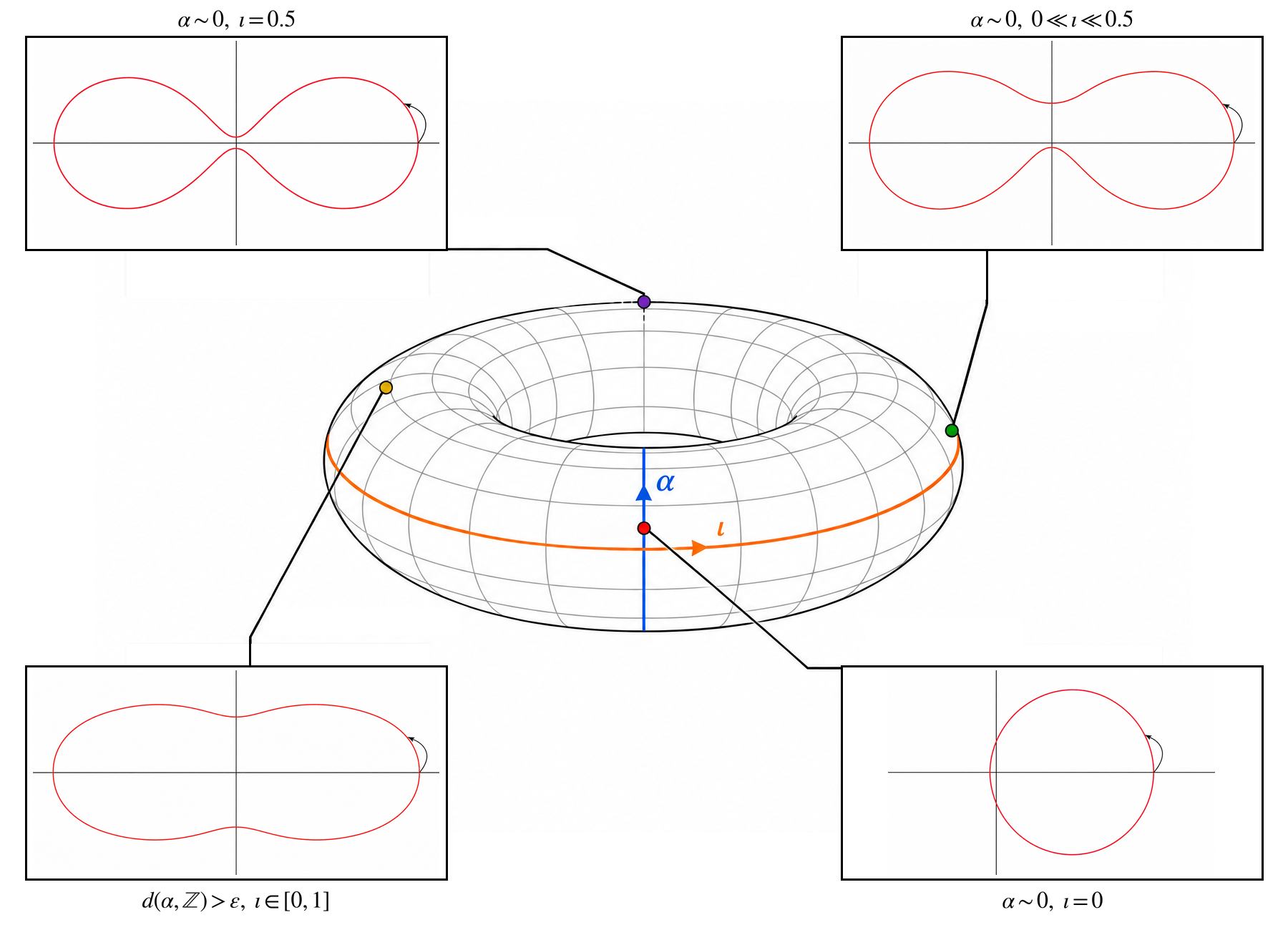}

    \caption{\label{torus} The torus is the parameter space, while each box is a simulation of the iterates of $\mathds{T}_{[\alpha],[\iota]}$ for about $1/d(\alpha,\mathds{Z})$ iterations, with different choices for $\alpha$ and $\iota$.
    }
\end{figure}
 Items $(i-iii)$ are trivially satisfied by the family of rotations in the unit disk with two points marked on its boundary, with the renormalisation corresponding to the return map to a sector of angle $\alpha$. However, property $(iv)$, combined with a renormalisation operator again corresponding to a sector return map, ensures the result is non-trivial, and essentially determines $\mathds{F}_2$ in a unique fashion. The explicit formula in $(iv)$ arises from analysing the behaviour of Fatou coordinates for different parameters. The key method for developing these models is finding change of coordinate functions that reflect this behaviour, which we first introduce in Section \ref{changeofcoordinates}. Following the uni-critical case, we expect $\mathds{A}_{[\alpha],[\iota]}:=\partial\mathds{M}_{[\alpha],[\iota]}$
is a model for $\omega(c)$, while $\mathds{M}_{[\alpha],[\iota]}$ is a model for $\omega(c)\cup \Delta_f$. 

We also explain the dynamics of the model, in particular we prove:
\begin{introtheorem}[Trichotomy]
\label{trichA}
For all $[\alpha]\in (\mathds{R}\setminus\mathds{Q})/\mathds{Z}$ and $[\iota]\in\mathds{R}/\mathds{Z}$:
    \begin{itemize}
        \item If $\alpha$ is a Herman number, $\mathds{A}_{[\alpha],[\iota]}$ is a Jordan curve,
        \item if $\alpha$ is a Brjuno number but not a Herman number, $\mathds{A}_{[\alpha],[\iota]}$ is a one-sided hairy Jordan curve,
        \item if $\alpha$ is not a Brjuno number, $\mathds{A}_{[\alpha],[\iota]}$ is a Cantor bouquet.
    \end{itemize}
\end{introtheorem}
We also classify closed invariant sets of the model. This establishes that broadly the topological possibilities are the same as the uni-critical case in \cite{cheraghi2022arithmetic}. On the other hand there are some drastic geometric and dynamical differences: 
\begin{introtheorem}[Conjugacy Classes]
\label{conjclass}
    For any non-Herman $\alpha\in\mathds{R}\setminus\mathds{Q}$ there is a dense set of $[\iota]\in\mathds{R}/\mathds{Z}$ such that 
$(\mathds{T}_{[\alpha],[\iota]}:\mathds{M}_{[\alpha],[\iota]}\rightarrow\mathds{M}_{[\alpha],[\iota]})$ 
    is not topologically conjugate to the uni-critical model, $
    (\mathds{T}_{[\alpha]}:\mathds{M}_{[\alpha]}\rightarrow \mathds{M}_{[\alpha]})
    $, via a map preserving $+1$.
\end{introtheorem}
From this result it is clear that the $\iota$ parameter introduces a rich degree of complexity for the conjugacy classes. In a forthcoming work, \cite{jossydavoud}, we study centralisers for unicritical irrationally indifferent attractors, which is a vital pre-requisite to further understanding this question of conjugacy.
\par
There is an ``extended gauss map" defined on $([-1/2,1/2]\setminus\mathds{Q})\times[-1/2,1/2]$ by
    $(\alpha,\iota)\mapsto (d(1/\alpha,\mathds{Z}),d(\iota/\alpha,\mathds{Z}))$, which generates a sequence of pairs $(\alpha_n,\iota_n)\in([0,1/2]\setminus\mathds{Q})\times[0,1/2]$ (see specifics in Section \ref{arithmeticconditions}), this is sometimes referred to as an \textit{Ostrowski Enumeration}. In the study of the topology of $\mathds{M}_{[\alpha],[\iota]}$ we come across modified notions of Brjuno and Herman conditions. For instance, a bi-critical Brjuno function arises, defined as (where the product is considered to be $1$ when $n=0$):
    $$
\mathcal{B}(\alpha,\iota) := \frac{1}{2}\sum_{n=0}^{\infty} \left(\prod_{i=0}^{n-1}\alpha_i\right)\log\left(\frac{1}{\alpha_i(\alpha_i+\iota_i)}\right),
$$
which satisfies 
$$
    \frac{1}{2}\mathcal{B}(\alpha)<\mathcal{B}(\alpha,\iota) \leq \mathcal{B}(\alpha),
    $$ 
    where $\mathcal{B}$ is the specific Brjuno function defined by Yoccoz \cite{MR1367353}. It is natural to refer to it as a ``bi-critical Brjuno function" as it exhibits a kind of equivariance with respect to the action of the extended Gauss map, in particular for $\alpha\in[-1/2,1/2]\setminus\mathds{Q}$, $\iota\in[-1/2,1/2]$ it satisfies:
$$
\mathcal{B}(\alpha,\iota)=|\alpha|\mathcal{B}(1/\alpha,\iota/\alpha)+\frac{1}{2}\log(1/|\alpha|)+\frac{1}{2}\log(1/(|\alpha|+|\iota|)).
$$

    This $\mathcal{B}(\alpha,\iota)$ function will determine the linearisability of $\mathds{T}_{[\alpha],[\iota]}$. The analogue of the Herman condition in our setting is more delicate and is described in Sections \ref{arithmeticconditions} and \ref{bicriticalbrjun}. On the other hand, when $\iota=0$, the bi-critical model is actually dynamically conjugate to the uni-critical model, which corroborates with the fact that some of the methods used to develop uni-critical near-parabolic renormalisation in \cite{inoushish} can be extended to higher degrees, as seen in \cite{chéritat2020near,yang2021parabolic}. Other results like \cite{MR4915291}, which rely on different techniques, illustrate how these irrationally indifferent fixed points may be understood in the non-polynomial case as well.
\par

Although there are some results that consider multi-critical renormalisation for irrationally indifferent fixed points, such as Yang Fei in \cite{yang2021parabolic}, Cheritat in \cite{chéritat2020near}, and more recently \cite{chéritat2025invariantclassesparabolicrenormalization} by Cheritat and Roesch, the essential renormalisaiton machinery that led to a cascade of powerful results in the uni-critical case does not yet exist in the bi-critical. This work shows that the methods used by Cheraghi in \cite{cheraghi2022arithmetic} and \cite{cheraghi2022topology} to develop a model map can nonetheless be applied. Once renormalisation can be adequately defined, the model may immediately allow for properties found there to be transferred to holomorphic maps, as was the case for unicritical maps, ultimately bringing our understanding of the multi-critical case to the same level.

Indeed, we conjecture that the model will be linked to the dynamics of holomorphic maps with bi-critical irrationally indifferent attractors in a manner that is made precise in Section \ref{conjectconseq}.

In Section \ref{arithmeticconditions} the arithmetic properties of irrational numbers that we use throughout the paper are discussed. In Section \ref{changeofcoordinates} the key engine for the model - a bi-critical change of coordinate function - is defined and introduced, while in Section \ref{definemodel} this is employed to define the model. In Section \ref{bicriticalbrjun} some of the bi-critical arithmetic conditions that arise in the model are studied, and in Section \ref{trichinbi} these are used to prove a trichotomy for the model. In Section \ref{dynamprop}, dynamical properties of the model are studied, including conjugacy classes, and finally in Section \ref{geometryofbicritjust} the geometry of holomorphic bi-critical maps is studied in order to understand the reasoning behind the model.

\subsection{Commonly Used Notation}
\label{commonlyused}
    Let $D\subset\mathds{C}$ and let $f, g$ be functions on $D$ that take values in $\mathds{C}$ (in particular both functions could be real valued). In this paper we will occasionally say that there exists a constant $C>0$ such that for all $x\in D$, we have:
    $$
    f(x) \asymp_C g(x),
    $$
    to mean that for all $x\in D$,
    $$
    C^{-1}|g(x)|\leq |f(x)|\leq C|g(x)|,
    $$
    while we will write
    $$
    f(x)\simeq_C g(x),
    $$
    if for all $x\in D$,
    $$
    |g(x)-f(x)|\leq C.
    $$
   We will also use $\mathds{H}'$ to denote the half plane defined by $\mathds{H}' = \mathds{H}-i$; elements in the complex plane with imaginary part greater than $-1$. Further, I will sometimes write $\mathds{H}'_{[a,b]}:= \{ w\in\mathds{C}: \mathrm{Re}(w)\in[a,b], \mathrm{Im}(w)>-1$\} for $a<b\in\mathds{R}$. 
    \par
    The notation $[x]\in\mathds{R}/\mathds{Z}$ is used to denote the equivalence class $x + \mathds{Z}$ for $x\in\mathds{R}$ (this was already used above). Furthermore, as is common in iterative dynamics, $f^{\circ n}$ denoting $n$ compositions of $f$ will be used (e.g. $f^{\circ 2}=f\circ f$), while we also remark that a ``critical point", $c\in\mathds{C}$, for a function $f$, means a point such that $f'(c)=0$.

\tableofcontents   
\section{Arithmetic Conditions}
\label{arithmeticconditions}

In this section we introduce the new bi-critical ``Brjuno and Herman" conditions, which will depend on two real parameters, rather than just one. We also introduce some key arithmetic concepts that arise in the study of these irrationally indifferent maps. The arithmetic class of the Brjuno condition was initially studied by Siegel-Brjuno in \cite{siegelref,MR377192} for analytic linearisation of irrationally indifferent fixed points. The Herman condition was introduced by Herman-Yoccoz \cite{MR538680, MR1924912} in their study of circle diffeomorphisms. As established by Perez-Marco in \cite{perezmarco}, there is a deep connection between circle diffeomorphisms and irrationally indifferent fixed points, however in our case the new conditions do not arise from circle diffeomorphisms but naturally from the model itself.
\subsection{The Continued Fraction}

A \emph{continued fraction} is an expression of the form
\[
\alpha=\tilde a_0 + \cfrac{1}{\tilde a_1 + \cfrac{1}{\tilde a_2 + \cfrac{1}{\tilde a_3 + \ddots}}},
\]
where \(\tilde a_0\) is an integer and \(\tilde a_i\) (\(i \geq 1\)) are positive integers. It is commonly denoted by
\[
\alpha=[\tilde a_0; \tilde a_1, \ldots].
\]

Every irrational real number admits a unique infinite continued fraction expansion of this form. In particular, the ``best rational approximants" mentioned previously will satisfy:
$$
[\tilde a_0;\tilde a_1,...,\tilde a_n]=p_n/q_n.
$$
There is a natural action of the Gauss map, $x\mapsto1/x$, on these continued fractions. Let $\langle x \rangle \in (0,1)$ denote the
fractional part of $x$, that is,
\[
x \in \mathds{Z} + \langle x \rangle .
\]
Now, define $\tilde \alpha_0=\langle \alpha \rangle$ and write $\tilde\alpha_i:=[\tilde a_i;\tilde a_{i+1},...]=\langle 1/\tilde \alpha_{i-1}\rangle$ for $i\geq1$. Then in particular this Gauss map action will induce:
$$
\langle 1/\tilde\alpha_i \rangle=\tilde\alpha_{i+1}.
$$
Given that the initial $\tilde\alpha_0$ is irrational, each $\alpha_i$ will be irrational also.
\subsection{Modified Continued Fraction}
\label{modifiedcontfract}
Following the work of Cheraghi in \cite{cheraghi2022arithmetic}, we introduce a modified continued fraction defined in the following fashion. For each  a sequence $\alpha_n\in(0,1/2)\setminus\mathds{Q}$ for every $n\geq0$ in the following way:
\begin{equation*}
    \begin{aligned}
        \alpha_0 = d(\alpha,\mathds{Z}),&\text{ and } \alpha_{n+1} = d(1/\alpha_n,\mathds{Z}).
    \end{aligned}
\end{equation*}
There will be a unique sequence of integers $a_{n}\in\mathds{Z}$ for $n\geq-1$ where $a_{-1}\in\mathds{Z}$ and $a_n$ is positive for $n\geq0$, along with a sequence $\epsilon_n\in\{+1,-1\}$ for $n\geq0$ such that:
\begin{equation*}
    \begin{aligned}
        \alpha = a_{-1}+\epsilon_0\alpha_0,&\text{ and } 1/\alpha_n = a_n +\epsilon_{n+1}\alpha_{n+1}.   
    \end{aligned}
\end{equation*}

We note that both of the modified continued fractions arise from considering the action of the Gauss map on different fundamental domains to $(0,1)$, as it was classically treated.

\subsection{The Brjuno and Herman Conditions}

The initial Brjuno condition studied was defined by
$$
\sum_{n\ge 0}\frac{\log q_{n+1}}{q_n}<\infty,
$$
where $q_n$ denotes the denominator of the $n$-th convergent of the (standard) continued fraction expansion of $\alpha$. In \cite{MR1367353} Yoccoz defines a very closely related series which enjoys equivariant properties with respect to the action of $\text{PGL}(2,\mathds{Z})$. In particular let $\tilde \beta_{-2}=\alpha$, $\tilde \beta_{-1}=+1$, and $\tilde \beta_n=\prod_{i=0}^n\tilde \alpha_i$ for all $n\geq 0$, then we can define the Yoccoz version of the Brjuno function as:
$$
\tilde{\mathcal{B}}(\alpha)=\sum_{n=-1}^\infty \tilde\beta_n\log1/
\tilde\alpha_{n+1}.
$$
This will satisfy the relations:
\begin{equation*}
    \begin{aligned}
        \tilde{\mathcal{B}}(\alpha)=\tilde {\mathcal{B} } (\alpha +n), & & \tilde{\mathcal{B}}(\alpha)=\alpha\tilde{\mathcal{B}}(1/\alpha)+\log(1/\alpha),
    \end{aligned}
\end{equation*}
which show that these numbers are preserved under the $\text{PGL}(2,\mathds{Z})$ action generated by Möbius transformations. In particular, we note that $|\sum_{n\ge 0}\frac{\log q_{n+1}}{q_n}-\tilde{\mathcal{B}}(\alpha)|$ is uniformly bounded above, therefore $\alpha$ is Brjuno if and only if $\tilde{\mathcal{B}}(\alpha)<\infty$.

The arithmetic class of Brjuno numbers have been studied by many authors in several different contexts. It was observed by Yoccoz in \cite{yoccozfirst} that there is a deep connection to the action of the functions that can be ``defined" as $[x]\mapsto[1/[x]]$ on $[x]\in\mathds{R}/\mathds{Z}$. In this pseudo-function we deliberately obscure what the action of $1/[x]$ will be, as it of course depends on the particular representatives of $[x]$ chosen, otherwise known as the \textbf{fundamental domain} (e.g. $[2]=[1]$ but $[1/2]\neq[1]$).

The arithmetic class of Herman was also introduced by Yoccoz in honour of the work of Herman. It is significantly more unwieldy, and was first introduced in the study of circle diffeomorphisms, as it is the precise condition on the rotation number for which analytic circle diffeomorphisms will be smoothly conjugate to the rigid rotation map \cite{MR1924912}. However it emerges in an essentially independent manner through the study of irrationally indifferent attractors, as seen by \cite{cheraghi2022arithmetic} and \cite{MR4956555}. To define the condition it is first necessary to define a family of diffeomorphisms $h_r:(0,\infty)\rightarrow(-\infty,\infty)$ for each $r>0$ by:
    \[
h_r(y) :=
\begin{cases}
r^{-1}\left( y - \log r^{-1} + 1 \right), & \text{if } y \ge \log r^{-1}, \\[6pt]
e^{y}, & \text{if } y \le \log r^{-1}.
\end{cases}
\]
\begin{definition}
    Suppose that $\alpha$ is Brjuno, and furthermore for all $n\geq 0$ there is $m\geq n$ such that:
    $$
    h_{\tilde\alpha_{m-1}}\circ ...\circ h_{\tilde\alpha_n}(0)\geq \tilde B(\alpha_m).
    $$
Then we say that $\alpha$ is Herman (or $\alpha\in\mathcal{H}\subset\mathcal{B}$).
\end{definition}

We note that it is known that both Herman and Brjuno numbers are dense and full measure on $\mathds{R}/\mathds{Z}$, while their complements are additionally dense. The set $\mathcal{B}\setminus\mathcal{H}$ is also dense, as shown in \cite{cheraghi2022arithmetic}, Section 6.1.

\subsection{Modified Brjuno and Herman Conditions}
 Using the modified continued fraction defined in Section 3.1.1, arising from the action of$\alpha\mapsto d(1/\alpha,\mathds{Z})$, recall that we can define a sequence $\alpha_i\in(0,1/2)\setminus\mathds{Q}$ for any $\alpha\in\mathds{R}\setminus\mathds{Q}$. A priori, the modified choice of fundamental domain to $(-1/2,1/2)$ can lead to the definition of different ``Brjuno" functions, however a consequence of the intricate studies by \cite{MR1462766}, is that all of these variants yield equivalent classes of ``Brjuno numbers". This means that the modified definition of successive numbers in Section \ref{modifiedcontfract} gives rise to a Brjuno condition that is exactly equivalent to the original. Let $ \beta_{-2}=\alpha$, $\beta_{-1}=+1$, and $ \beta_n=\prod_{i=0}^n \alpha_i$ for all $n\geq 0$, then:
\begin{definition}[Alternative Brjuno Definition]
    For $\alpha\in\mathds{R}\setminus\mathds{Q}$, $\alpha\in\mathcal{B}$ ($\alpha$ is Brjuno) if and only if:
    $$
    \mathcal{B}(\alpha):= \sum_{n=0}^{\infty} \beta_n\log(1/\alpha_n) <\infty.
    $$
\end{definition}
The choice of fundamental domain will also alter the Herman condition, but it was shown by Cheraghi in Section 6.2 of \cite{cheraghi2022arithmetic} that the class of numbers satisfying the Herman condition using the fundamental domain $[-1/2,1/2]$ is equivalent to the original version defined on $[0,1]$. In general it is unknown if the Herman condition is entirely independent of fundamental domain.
\begin{definition}
    Suppose that $\alpha$ is Brjuno, and furthermore for all $n\geq 0$ there is $m\geq n$ such that:
    $$
    h_{m-1}\circ ...\circ  h_n(0)\geq \mathcal{B}(\alpha_m),
    $$
    where the $h_n(y):=h_{r=\alpha_n}(y)$ are the same diffeomorphisms defined above. Then we say that $\alpha$ is Herman (or $\alpha\in\mathcal{H}\subset\mathcal{B}$).
\end{definition}
Both the Brjuno and Herman conditions are dense on the fundamental domain, as in particular they contain all bounded type numbers. In fact the Brjuno condition is full measure. On the other hand, the set $\mathcal{B}\setminus{\mathcal{H}}$ is also dense, as is $\mathds{R}\setminus\mathcal{B}$. A survey detailing some of these properties is given in Section 6 of \cite{cheraghi2022arithmetic}.

\subsection{Bi-Critical Brjuno and Herman conditions}
In this section we first describe a modified ``bi-critical" continued fraction that takes into account \textbf{both} parameters $\alpha$ and $\iota$. The action on $\alpha$ and $\iota$ is related to considering the dynamical system of rigid rotations on a circle with two marked points.
\par
As previously employed by Cheraghi \cite{cheraghi2022arithmetic}, the context of renormalisation for irrationally indifferent fixed points lends itself to modified action of the Gauss map with fundamental domain $[-1/2,1/2]$. In particular recall the sequences described in Section 3.1.1 denoted by $\alpha_n$ for $\alpha\in\mathds{R}\setminus\mathds{Q}$, the sequence of integers $a_n$, and the sequence $\epsilon_n\in\{+1,-1\}$. 
\par
However, because the model will now depend in addition on an extra $\iota$ parameter, arising from the two marked ``critical" points at $+1$ and $e^{2\pi i\iota}$, there will be another arithmetic sequence that emerges naturally from the renormalisation, as depicted in Figure \ref{rigidrotation}. 
\begin{figure}[H]
\captionsetup{justification=centering}
    \centering
\includegraphics[width=\textwidth]{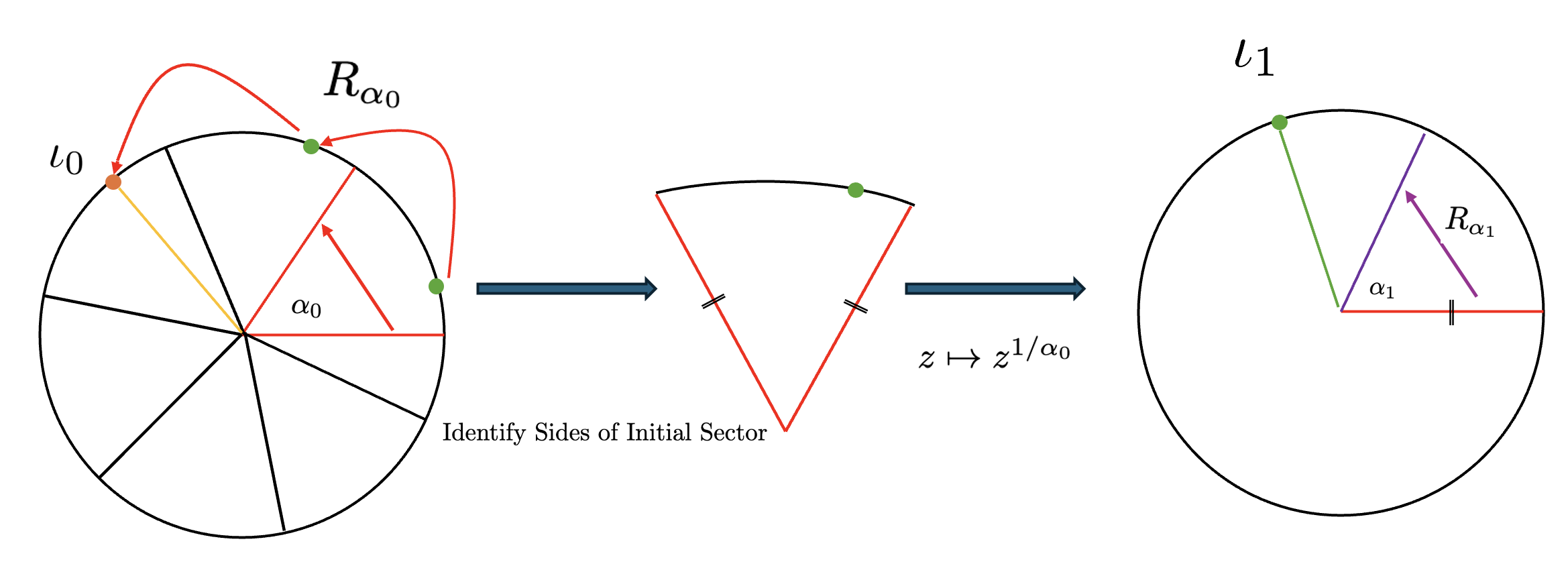}

    \caption{\label{rigidrotation} The map $R_{\alpha}:\mathds{C}\mapsto\mathds{C}$ is defined as the rigid rotation map, which acts as $z\mapsto e^{2\pi i\alpha}z$. Because any map with an irrationally indifferent fixed point will act as the rigid rotation map on the internal \textbf{angles} of the irrationally indifferent attractor, the rigid rotation gives a model for the angular location of the critical points under renormalisation.}
\end{figure}
For each $\iota\in\mathds{R}$ define:
\begin{equation*}
    \begin{aligned}
        \iota_0 = d(\iota,\mathds{Z}),\text{ and } & \iota_{n+1} = d(\iota_n/\alpha_n,\mathds{Z}).
    \end{aligned}
\end{equation*}
There exists a sequence of integers $b_n\in\mathds{Z}$ for $n\geq -1$ (which is again positive for $n\geq0$) along with a sequence $\delta_n\in\{+1,-1\}$ for $n\geq0$ such that:
\begin{equation*}
    \begin{aligned}
        \iota=b_{-1} + \delta_0\iota_0, \text{ and } &
        \iota_n/\alpha_n = b_n + \delta_{n+1}\iota_{n+1}.
    \end{aligned}
\end{equation*}
This operation is equivalent to what is called an ``Ostrowski expansion", which is related to the concept of an ``Ostrowski enumeration". For a partial survey of this topic see \cite{berthé2023dynamicsostrowskiskewproducti}. Our version is slightly modified, as we introduce the extra $\epsilon_n$ and $\delta_n$ parameters as in \cite{cheraghi2022arithmetic} so as to keep all $\alpha_i$ and $\iota_i$ parameters positive, but the process is largely equivalent. Many authors have studied the dynamics of this process, but we do not focus on that in this work.
\begin{lemma}
    For all $\alpha \in \mathds{R}\setminus\mathds{Q}$, and all sequences of integers $\tilde{b}_n$ and $\tilde \delta_n\in\{+1,-1\}$ on $n\geq0$ satisfying $0\leq \tilde{b}_n<a_n/2$, it is possible to pick a unique $[\iota]\in\mathds{R}/\mathds{Z}$ such that $\iota_n \in [\tilde\delta_n(\tilde b_n-1/2)\alpha_n,\tilde \delta_n(\tilde b_n+1/2)\alpha_n]$.
\end{lemma}
\begin{proof}
    We can explicitly exhibit such a $\iota$ as $\iota := \sum_{n=0}^{\infty}\epsilon_n\tilde \delta_n\tilde{b}_n \alpha_n$
    and clearly this will have the desired property. Since addition of an integer will not affect the modified Ostrowski expansion for $n\geq 0$, this gives a unique class $[\iota]\in\mathds{R}/\mathds{Z}$.
\end{proof}

Now using the modified Ostrowski enumeration, we define bi-critical arithmetic functions that will now in addition depend on this ``internal angle” or $\iota$ parameter. Like the Brjuno function, this will also be a co-cycle for the action that defines the Ostrowski enumeration with fundamental domain $(-1/2,1/2)\times[-1/2,1/2]$. For $\alpha\in(0,1/2)\setminus\mathds{Q},\iota\in[0,1/2]$ the bi-critical function is defined as:
$$
\mathcal{B}(\alpha,\iota)= \sum_{n=0}^{\infty} \beta_n \left(\frac{1}{2}\log\frac{1}{\alpha_n}+\frac{1}{2}\log\frac{1}{\alpha_n+\iota_n}\right).
$$
We note from the definition that it is immediate that 
$$
\frac{1}{2}\mathcal{B}(\alpha)\leq \mathcal{B}(\alpha,\iota)\leq\mathcal{B}(\alpha),
$$
and hence $\mathcal{B}(\alpha,\iota)=\infty$ if and only if $\mathcal{B}(\alpha)=\infty$. The upper bound is attained when $\iota=0$, and there exists a special choice of $\iota=\tilde{\iota}$ (which is close to but not equal to $1/2$ in general) and a small uniform constant $C>0$ such that 
    $$
    \frac{1}{2}\mathcal{B}(\alpha)<\mathcal{B}(\alpha,\tilde{\iota})<\frac{1}{2}\mathcal{B}(\alpha)+C.
    $$ This is because we can choose a ``maximal" $\iota$ such that $\iota_n \in[1/2-\alpha_n,1/2] $ for all $n$. This will then generate the constant given in the statement, as the $\frac{1}{2}\log\frac{1}{\alpha_n+\iota_n}$ term will now be uniformly bounded above by $\frac{1}{2}\log\frac{1}{2}$ which sum up after multiplication by the $\beta_n\leq (1/2)^n$ to at most $\log\frac{1}{2}$. This implies that the new function does indeed reflect a genuinely distinct geometric structure.
\par
To justify the ``co-cycle” description we note that in particular we have that when $\alpha\in(-1/2,1/2)\setminus\mathds{Q}$, and $\iota\in[-1/2,1/2]$ we have:
$$
\mathcal{B}(\alpha,\iota)=|\alpha|\mathcal{B}(1/\alpha,\iota/\alpha)+\frac{1}{2}\log(1/|\alpha|)+\frac{1}{2}\log1/(|\alpha|+|\iota|),
$$
and furthermore for all $m,n\in\mathds{Z}$:
$$
\mathcal{B}(\alpha+n,\iota+m)=\mathcal{B}(\alpha,\iota).
$$
Similarly there is a ``bi-critical” equivalent of the Herman condition. Now for $0<r\leq1/2$ and $0\leq s\leq1/2$ we introduce a modified bi-critical Herman function as follows (shown in Figure \ref{hermanfunctions}):
\[ h^{-1}_{r,s}(y) = \begin{cases} 
        \frac{1}{2}\log y+\frac{1}{2}\log\frac{s+ry}{s+r} & 0\leq y\leq (1-s)/r, \\
          \frac{1}{2}\log y+\frac{1}{2}(ry +\log\frac{1}{s+r}-(1-s)) & (1-s)/r\leq y \leq 1/r, \\
          
          ry + \frac{1}{2}\log\frac{1}{r}+\frac{1}{2}\log\frac{1}{r+s} - (1-\frac{1}{2}s)&  1/r \leq y .
       \end{cases}
    \]

\begin{figure}[H]
\captionsetup{justification=centering}
    \centering
\includegraphics[width=0.8\textwidth]{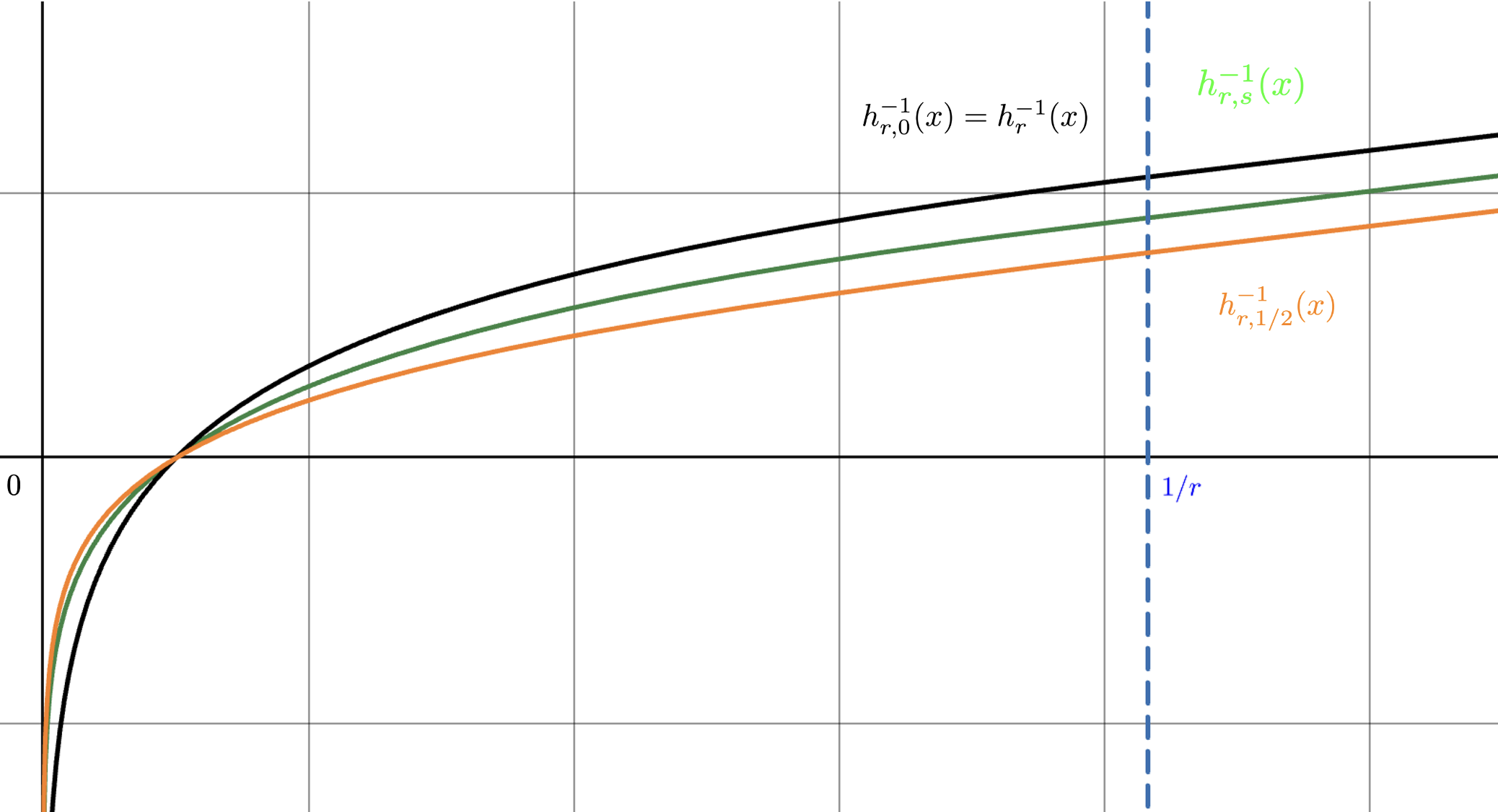}

    \caption{\label{hermanfunctions} The functions $x\mapsto h^{-1}_{r,s}(x)$ graphed for $r=0.12$ and $s$ varying between $0$ and $1/2$.} 
\end{figure}
This ``extended Herman function" is in the spirit of the original, in part because it satisfies the additional condition of being a diffeomorphism. Since this function is strictly increasing, it is definitely invertible, however $h_{r,s}$, its inverse, is complicated to write down. We use this ``backwards" definition to emphasize the connection to the original Herman function, noting that $h_{r,0}=h_r$. Now define $\tilde h_n:= h_{\alpha_n,\iota_n}$.
\begin{definition}
\label{bicrithermandef}
    Say that $(\alpha,\iota)\in([-1/2,1/2]\setminus\mathds{Q})\times[-1/2,1/2]$ is of \textbf{bi-critical} Herman type if $\mathcal{B}(\alpha,\iota)<\infty$ (which in particular implies $\alpha$ is Brjuno), and in addition for all $n\geq 0$ there is $m\geq n$ such that:
    $$
    \tilde h_{m-1}\circ ...\circ \tilde h_n(0)\geq \mathcal{B}(\alpha_m,\iota_m).
    $$
\end{definition}

For shorthand, we will denote $\mathcal{H}\subset \mathds{R}\setminus\mathds{Q}$ to mean the original Herman set, and $\mathcal{H}_2\subset(\mathds{R}\setminus\mathds{Q})\times\mathds{R}$ to denote the bi-critical Herman condition. This bi-critical Herman condition will be studied extensively in Section \ref{bicriticalbrjun}, and it will be seen that almost every pair is in fact of Herman type, but the complement will also be dense. Unlike the original Herman condition, which arose from the study of circle diffeomorphisms, this bi-critical Herman condition arises purely from the model.
\section{Change of Coordinates for the Bi-Critical Model}
\label{changeofcoordinates}
In this section we introduce the key analytic technical tools that are required to define the new $(\mathds{M}_{[\alpha],[\iota]},\mathds{T}_{[\alpha],[\iota]})$ model. These come in the form of a bi-critical  ``geometric change of coordinates" function $G_{r,s}$, which outlines the special bi-critical geometry, and a change of coordinates function, $Y_{r,s}$, which combines the geometry of the $G_{r,s}$ with some special \textit{functional relations}.
\subsection{Geometry for the Change of Coordinates}
\begin{figure}[ht]
\captionsetup{justification=centering}
    \centering
\includegraphics[width=0.5\textwidth]{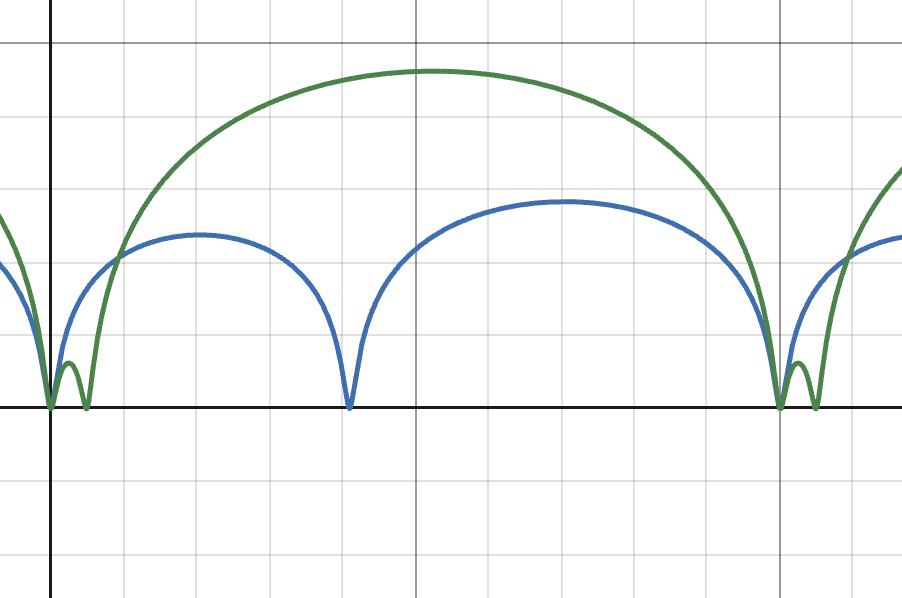}

    \caption{\label{changecoordpict} The figure depicts the image of a horizontal line at $y=0$ for the function $G_{r=0.004,s}(w)-G_{r=0.004,s}(0)$, which is within a uniform additive constant of the ultimate change of coordinate function we seek to define, $Y_{r,s}$. The green line represents the ``nearly uni-critical case" when $s$ is small, and the blue line represents $s$ close to $1/2$.}
\end{figure}

For $r\in(0,1/2]$ define:
$$
g_r(w) := |e^{-3\pi r}-e^{-2\pi r i w}|,
$$
which traces the modulus of a circle slightly displaced from $0$. Now for $s\in[0,1/2]$ write:
$$
G_{r,s}(x+iy) := rx+\frac{i}{4\pi}(\log g_r (x+iy+1/2) + \log g_r(x-s/r +iy+1/2)).
$$
The $G_{r,s}$ function (as shown in Figure \ref{changecoordpict}) bears a relation to the uni-critical change of coordinates defined by Cheraghi:
\begin{lemma}
\label{cheraghi}
Let $Y_r(x+iy) := r\mathrm{Re}(w)+\frac{i}{2\pi}\log \frac{g_r(w+1/2)}{g_r(1/2)}$
be the change of coordinates employed by Cheraghi in \cite{cheraghi2022topology,cheraghi2022arithmetic}. Then there exists a $C(r)\in\mathds{R}$ that only depends on $r$ such that:
$$
G_{r,s}(x+iy) = \frac{1}{2}(Y_r(x+iy) + Y_{r}(x+iy-s/r)+s)+iC(r).
$$
\end{lemma}
To see this is true we note that from the explicit form of the equations, it is true with $C(r) = -\frac{1}{2\pi}\log g_r(1/2)$.
\par
The $G_{r,s}$ function will give us the fundamental geometry of the bi-critical change of coordinates, in the sense that the $x+iy \mapsto G_{r,s}(x+iy)-G_{r,s}(0)$ function will be a bounded alteration of the actual change of coordinates function in the bi-critical renormalisation tower.
\begin{proposition} For every $r\in(0,1/2]$, $s\in[0,1/2]$, the map $G_{r,s}$ satisfies the following properties:
\label{GGrr}
\begin{enumerate}
        \item The map $G_{r,s}$ is injective on $\overline{\mathds{H}'}$ and $G_{r,s}(\overline{\mathds{H}'})\subset\mathds{H}'$,
        \item for every $w\in\overline{\mathds{H}'}$,
        $$
        G_{r,s}(w+1/r) = G_{r,s}(w)+1,
        $$
        \item for all $w_1,w_2 \in \overline{\mathds{H}'}$,
        $$
        |G_{r,s}(w_1)-G_{r,s}(w_2) | \leq 0.9 |w_1 - w_2|.
        $$
    \end{enumerate}
\end{proposition}
\begin{proof}
These results essentially follow from Lemma \ref{cheraghi} and the fact that Cheraghi proves that $Y_r(x+iy)$ satisfies all three of the conditions itself (see Section 2 in \cite{cheraghi2022arithmetic}). Since vertical lines are sent to vertical lines we must just check how $G_{r,s}(x+iy)$ behaves on each line. However it is clear that both functions $y\mapsto \mathrm{Im}(Y_r(x+iy)$ and $y\mapsto \mathrm{Im}(Y_r(x+iy-s/r)$ are strictly increasing on $\mathds{H}'$, therefore $y\mapsto \mathrm{Im}G_{r,s}(x+iy)$ is too, which is enough to prove injectivity. Then to show that $\mathrm{Im}(G_{r,s}(x+iy))>-1$ we note that since $\mathrm{Im}(G_{r,s}(x+iy))$ is an average of two $\mathrm{Im}(Y_r)$ functions, the vertical heights always stay above and below these values of these two functions which are in fact always greater than $-1$ on $\mathds{H}$'. To prove contractivity, note that for all $w_1,w_2\in\mathds{H}'$:
    \begin{equation*}
        \begin{aligned}
            \left|G_{r,s}(w_1)-G_{r,s}(w_2)\right| &= |\frac{1}{2}(Y_r(w_1)-Y_r(w_2) + Y_{r}(w_1-s/r)-(Y_{r}(w_2-s/r)))| \\
            &\leq |\frac{1}{2}(Y_r(w_1)-Y_r(w_2))| + |\frac{1}{2}(Y_r(w_1-s/r)-Y_r(w_2-s/r))|
            \\
            &\leq \frac{1}{2}(0.9+0.9)|w_1-w_2| = 0.9|w_1-w_2|. \qedhere
        \end{aligned}
    \end{equation*}
\end{proof}
\begin{lemma}[''$g_r$" lemma]
\label{grlem}
    Let $0<t<\pi$, then there exists a $C_{\ref{grlem}}(t)>0$ such that for all $0<r\leq1/2$:
    \begin{equation*}
        \begin{aligned}
            (C_{\ref{grlem}}(t))^{-1}<g_r(x+iy)/e^{2\pi ry}<C_{\ref{grlem}}(t), & & \forall x+iy \in \mathds{H}'_{[t/r,(1-t)/r]+\frac{1}{r}\mathds{Z}}.
        \end{aligned}
    \end{equation*}
    In a similar spirit, let $y\geq M/r$ for $M>0$ arbitrary. Then there exists a $D_{\ref{grlem}}(M)>0$ such that:
    \begin{equation*}
        \begin{aligned}
            (D_{\ref{grlem}}(M))^{-1}<g_r(x+iy)/e^{2\pi ry}<D_{\ref{grlem}}(M), & & \forall x+iy \in \mathds{H}' + Mi.
        \end{aligned}
    \end{equation*}
\end{lemma}
\begin{proof}
    From the explicit form $g_r(w) = |e^{-3\pi r}-e^{-2\pi r i w}|$ we see:
    $$
    (g_r(x+iy))^2= (e^{2\pi r y})^2 - 2 e^{2\pi r y-3\pi r}\cos(2\pi r x) +(e^{-3\pi r})^2.
    $$
    From this it is clear that on $x\in[0,1/r)$ the function $x\rightarrow g_r(x+iy)$ only has turning points at $x=0$ and $x=1/2$, which will be minima and maxima respectively. The maximal value that $g_r(w)$ takes will be $e^{-3\pi r}+e^{2\pi r y}\leq 2e^{2\pi r y}$, while since it will be strictly decreasing as $x$ moves away from $x=1/2$, the minimum value that it will take on $[t/r,(1-t)/r]$ is $|e^{-3\pi r}-e^{2\pi it}e^{2\pi r y}|$, but then it is clear by the geometric picture in Figure \ref{grfigure} we have
     
    \begin{equation*}
        \begin{aligned}
            |e^{-3\pi r}-e^{2\pi i t}e^{2\pi r y}|\geq \mathrm{Im}(e^{2\pi it}e^{2\pi r y}) = \sin(2\pi t)e^{2\pi r y}.
        \end{aligned}
    \end{equation*}
    So the lemma is proven with $C_{\ref{grlem}}(t) = \max\{1/\sin(2\pi t),2\}$. 
    \begin{figure}[H]
\captionsetup{justification=centering}
    \centering
\includegraphics[width=0.4\textwidth]{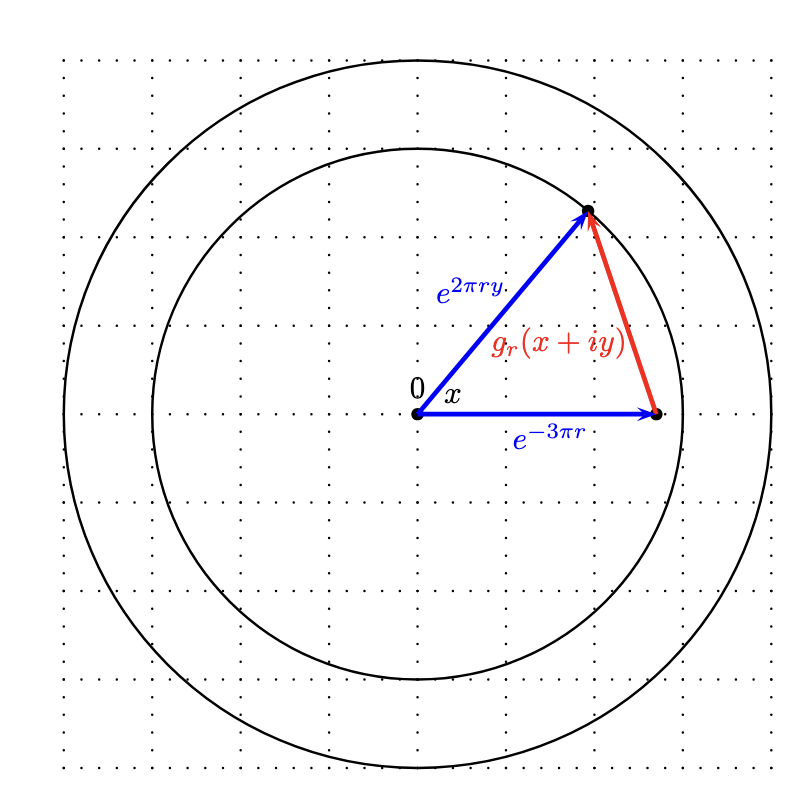}

    \caption{\label{grfigure}The $g_r(x+iy)$ functions are generated by the distances from the point $e^{-3\pi r}$ to a family of circles}
\end{figure} 
    For the second part we have that $\forall x \in\mathds{R}$ and $y\geq M/r$:
    $$
    g_r(x+iy)/e^{2\pi r y} = |e^{-3\pi r-2\pi r y}-e^{-2\pi ix}| \leq 1+e^{-3\pi r-2\pi r y}\leq 1+e^{2M\pi},
    $$
    and also 
    $$
    |e^{-3\pi r-2\pi r y}-e^{-2\pi i x}| \geq 1-e^{-3\pi r-2\pi r y}-1\geq 1- e^{-2\pi M},
    $$
    Hence the result is true by setting $D_{\ref{grlem}}(M) = \max\{1+e^{2M\pi},1/(1- e^{-2\pi M})$.
\end{proof}
The first part says that for any set of $x$ values uniformly bounded away from the minima points $x\mapsto g_r(x+iy)$ will be close to $e^{2\pi r y}$. On the other hand the second statement proves that as long as $yr$ is bounded below by a value greater than $0$, the map $y\mapsto g_r(x+iy)$ will be uniformly close to $e^{2\pi r y}$. 
\par
It is useful to define the following two functions:
\begin{definition}
\label{boundab}
    For all $r\in(0,1/2]$, $s\in[0,1/2]$ let 
    $$
    a_{r,s}(x,y)=\min\left\{\frac{1}{2\pi}\log g_r (x+iy+1/2),\frac{1}{2\pi}\log g_r(x-s/r +iy+1/2)\right\},
    $$ and 
    $$
    b_{r,s}(x,y) = \max\left\{\frac{1}{2\pi}\log g_r (x+iy+1/2),\frac{1}{2\pi}\log g_r(x-s/r +iy+1/2)\right\},
    $$ 
    \end{definition}
    Clearly for all $r,s,x,y,$ we have $\mathrm{Im}G_{r,s}(x+iy)\in [a_{r,s}(x,y),b_{r,s}(x,y)]$ (see Figure \ref{maxminpict}) and furthermore each of the two $a_{r,s}$, $b_{r,s}$ functions is equal to a translated $\frac{1}{2\pi}\log g_r$ function.

\begin{figure}[H]
\captionsetup{justification=centering}
    \centering
\includegraphics[width=0.6\textwidth]{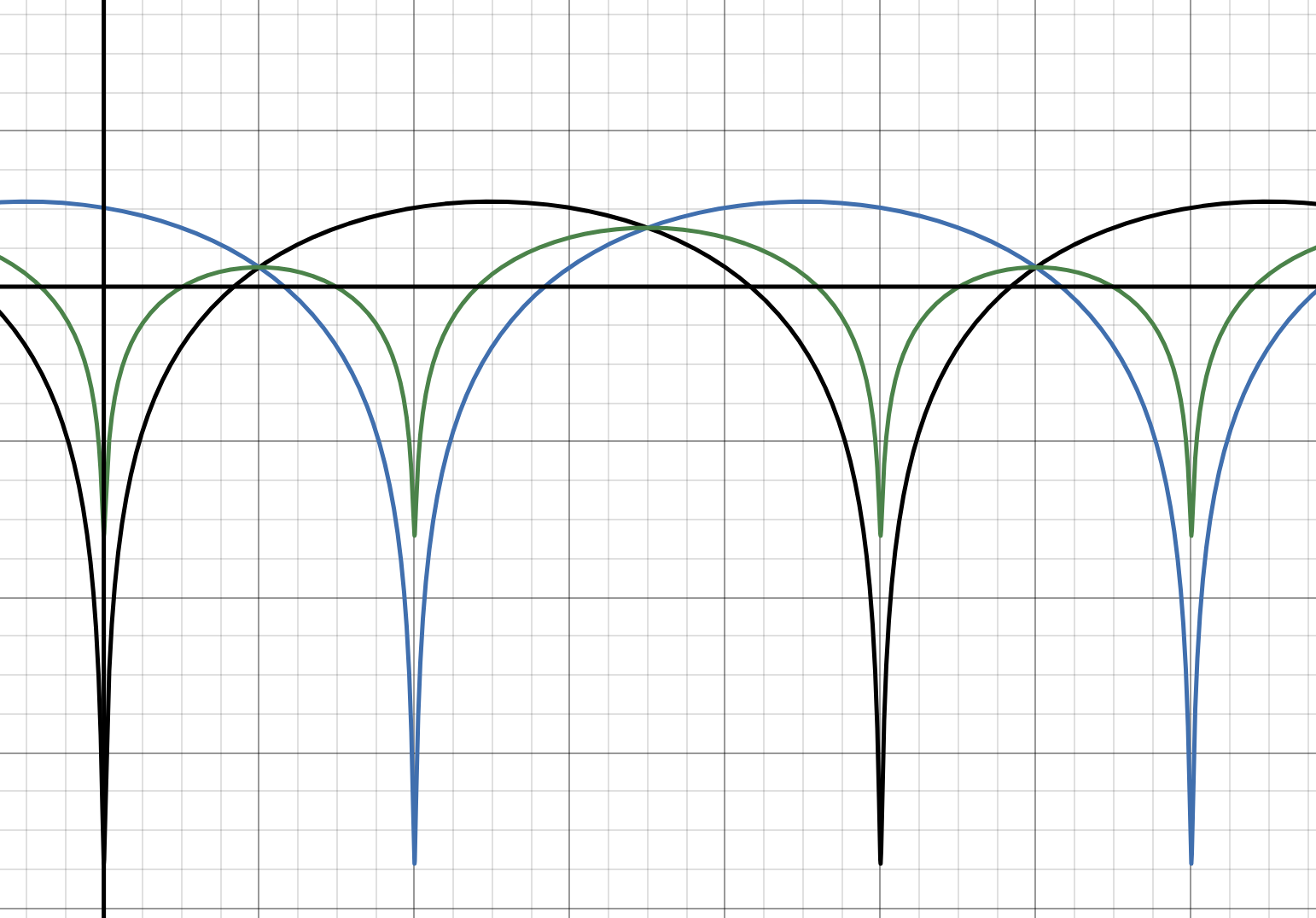}

    \caption{\label{maxminpict}The $G_{r,s}$ function (represented in green), always lies between the space drawn by the two $g_r$ functions that generate the $a_{r,s}$ and $b_{r,s}$. This image was generated with $r$ approximately $0.001$ and $s$ far away from $0$.}
\end{figure}
Now let $x_{max}:=(1+s)/2r-1/2$. We catalog the (approximate) locations of the minima and maxima of the functions $x\mapsto \mathrm{Im}(G_{r,s}(x+iy))$:
\begin{proposition}
\label{maxmin}
There exists a constant $C_{\ref{maxmin}}>0$ such that for all $r\in(0,1/2]$, $s\in[0,1/2]$, and each $y\geq-1$ :
$$
|\mathrm{Im}(G_{r,s}(x_{max}+iy)) - \sup_{x\in\mathds{R}}\mathrm{Im}(G_{r,s}(x+iy))| \leq C_{\ref{maxmin}},
$$
which also is such that for $\tilde{x}\in\{0,s/r\}$ and for all $y,r,s$:
$$
|\mathrm{Im}(G_{r,s}(\tilde{x}+iy)) - \inf_{x\in\mathds{R}}\mathrm{Im}(G_{r,s}(x+iy))| \leq C_{\ref{maxmin}}.
$$
\end{proposition}
Even if the actual maximum or minimum will lie at another location it does not matter because the heights at these particular choices of points will always be close enough.

\begin{proof} We split the result into the two parts:
\\
\textbf{Maxima Result}:
Now as in the $g_r$ lemma set $t_0=2^{-5}$. For each fixed $y$ we immediately obtain that for $x\in \{B_{t_0/r}(-1/2)+1/r\mathds{Z}\}\cup\{B_{t_0/r}(s/r-1/2)+1/r\mathds{Z}\}$ both $x\mapsto a_{r,s}(x,y)$ and $x\mapsto b_{r,s}(x,y)$ will attain their respective maxima in this domain and in addition be uniformly close to $2\pi r y$. Because $s\leq 1/2 $, $x_{max}=(1+s)/2r$ will always be outside of these balls, yielding the maxima result. 
\\
\textbf{Minima Result}:
For each fixed $y$, the $x\mapsto a_{r,s}(x,y)$ and $x\mapsto b_{r,s}(x,y)$ functions attain their minimums inside the set $x\in \{B_{t_0/r}(-1/2)+1/r\mathds{Z}\}\cup\{B_{t_0/r}(s/r-1/2)+1/r\mathds{Z}\}$, which implies that $G_{r,s}(w)$ must attain its minimum here. Now note that:
\begin{equation*}
            \begin{aligned}
                &\log g_r ((s/2r-1/2-x)+s/r+iy+1/2)\\
                &=\frac{1}{2}\log((e^{2\pi r y})^2 - 2 e^{2\pi r y-3\pi r}\cos(2\pi r ((s/2r-1/2-x)+s/r+1/2)) +(e^{-3\pi r})^2 )& \\
                &=\log g_r ((s/2r-1/2+x)+iy+1/2),
            \end{aligned}
        \end{equation*}
which has the consequence that there is a symmetry across the point $x = s/2r-1/2$, and hence that:
$$
    \mathrm{Im}(G_{r,s}((s/2r-1/2)+x+iy)) = \mathrm{Im}(G_{r,s}((s/2r-1/2)-x+iy)) .
$$
This implies that for each $y$ the values of $\mathrm{Im}(G_{r,s})$ at $x=s/r-1/2$ are the same as at $x=-1/2$, so proving that $x=-1/2$ gives a minima up to a constant is enough to say the same for $x=s/r-1/2$. Note that from the explicit form of the $g_r$ functions, $x\mapsto a_{r,s}(x,y)$ and $x\mapsto b_{r,s}(x,y)$ will be strictly decreasing on $-1/2-t_0/r\leq x\leq-1/2$. In fact $x=-1/2$ will be a turning point for $a_{r,s}(x,y)$, which will increase up until the point of symmetry at $x=s/2r-1/2$ (where $a_{r,s}(x,y)=b_{r,s}(x,y)$). This implies that the value of $\mathrm{Im}G_{r,s}(w)$ at $x=-1/2$ will give an upper bound on the minimum. We now require a slightly finer estimate on $G_{r,s}(x+iy)$:

\begin{equation*}
    \begin{aligned}
     \mathrm{Im}G_{r,s}(iy-1/2)&\geq \min_{x\in\mathds{R}}\mathrm{Im}G_{r,s}(x+iy)  \\ & \geq\frac{1}{2} \min_{x\in[-1/2,s/2r-1/2]} (a_{r,s}(x,y))+\frac{1}{2} \min_{x\in[-1/2,s/2r-1/2]} (b_{r,s}(x,y)).   
    \end{aligned}
\end{equation*}
This yields a lower bound, so combining the two we obtain:
\begin{equation*}
    \begin{aligned}
        \frac{1}{4\pi}(\log g_r (iy) + \log g_r(s/r +iy))&\geq\min_{x\in\mathds{R}}\mathrm{Im}(G_{r,s}(x+iy)) 
        \geq\frac{1}{4\pi}(\log g_r (iy) + \log g_r(s/2r +iy)).
    \end{aligned}
\end{equation*}
To conclude we now claim $|\log g_r(s/2r +iy) - \log g_r(s/r +iy)|$ is bounded above by a constant for all $y,r$ and $s\leq 1/2$ which proves that $x\in\{-1/2,s/r-1/2\}$ are both uniformly close to the minimum of the function $x\mapsto G_{r,s}(w)$. But from the explicit form of $g_r$:
$$
\frac{(g_r(-s/r+iy))^2}{(g_r(-s/2r+iy))^2}= \frac{(e^{2\pi r y})^2 - 2 e^{2\pi r y-3\pi r}\cos(2\pi s) +(e^{-3\pi r})^2}{(e^{2\pi r y})^2 - 2 e^{2\pi r y-3\pi r}\cos(\pi s) +(e^{-3\pi r})^2},
$$
and since all terms are non-zero and the $(e^{2\pi r y})^2$ part will dominate as $r$ tends to $0$, we obtain that there is a constant $C>0$ such that:
$$
C^{-1}\leq\left|\frac{(g_r(-s/r+iy))^2}{(g_r(-s/2r+iy))^2}\right|\leq C,
$$
which gives the result that $|\log g_r(s/2r +iy) - \log g_r(s/r +iy)|$ is bounded above by a constant, concluding the argument.
\end{proof}
\subsection{Definition and Properties of the Bi-Critical Change of Coordinates}
Finally we explicitly define a new function, $Y_{r,s}$, which takes values within a uniform constant of $G_{r,s}(w)-G(0)$, while satisfying additional important \textbf{functional relations}. Recall that $r\in(0,1/2]$, $s\in[0,1/2]$. To do this, we introduce the preliminary function:
    \[ \tilde{Y}^{\text{pre}_1}_{r,s}(x+iy) := \begin{cases} 
          G_{r,s}(x+iy) &   \mathrm{Im}(G_{r,s}(-1+iy))\leq\mathrm{Im}(G_{r,s}(x+iy)),\\
          G_{r,s}(-1+i\mathrm{Im}(w))
          & \mathrm{Im}(G_{r,s}(x+iy))\leq\mathrm{Im}(G_{r,s}(-1+iy)).
       \end{cases}
    \]
Now when $s\geq r$ define:
\[ \tilde{Y}^{\text{pre}_2}_{r,s}(x+iy) := \begin{cases} 
          G_{r,s}(x+iy-s/r)+s &   x\in(s/r-1,s/r-1/2]+1/r\mathds{Z},\\
          G_{r,s}(x+iy) &   x\in[s/r-1/2,s/r)+1/r\mathds{Z},\\
          G_{r,s}(x+iy) & x\in (1/r-1,1/r-1/2] + 1/r\mathds{Z}, \\
          G_{r,s}(x+iy+s/r)-s & x \in [1/r-1/2,1/r) +1/r\mathds{Z,} \\
          \tilde{Y}^{\text{pre}_1}_{r,s}(x+iy) & \text{otherwise},
       \end{cases}
    \]
while when $0\leq s<r$ define:
\begin{equation*}
    \begin{aligned}
        &\tilde{Y}^{\text{pre}_2}_{r,s}(x+iy) := \\
        & \begin{cases} 
           G_{r,s}(x+iy-1)+r &   x\in(0,s/2r]+1/r\mathds{Z}, \\
           G_{r,s}(x+iy) & x\in [s/2r,s) + 1/r\mathds{Z},\\
           G_{r,s}(x+iy) & x\in (1/r-1,1/r-1+s/2r] +1/r\mathds{Z} ,\\
           G_{r,s}(x+iy+1)-r & x\in (1/r-1+s/2r,1/r+s/r-1] + 1/r\mathds{Z} ,\\
          G_{r,s}(x+iy +s/r-1) - (s-r) & x\in (1/r+s/r-1,1/r+s/2r -1/2]+1/r\mathds{Z} ,\\
          G_{r,s}(x+iy+s/r)-s & x\in (1/r+s/2r -1/2,1/r]+1/r\mathds{Z}, \\
          \tilde{Y}^{\text{pre}_1}_{r,s}(x+iy) & \text{otherwise}.
       \end{cases}        
    \end{aligned}
\end{equation*}

\begin{definition}
    For all $r\in(0,1/2]$, $s\in[0,1/2]$, define: 
$$ 
Y_{r,s}(w) := \tilde{Y}^{\text{pre}_2}_{r,s}(x+iy) - \tilde{Y}^{\text{pre}_2}_{r,s}(0).
$$
\end{definition}

\begin{lemma}
    For all $r\in(0,1/2]$, $s\in[0,1/2]$, $Y_{r,s}(w)$ is continuous and well defined. It is analytic in $y$, and piece-wise analytic in $x$. 
\end{lemma}
\begin{proof}
The $G_r(x+iy)$ functions are clearly separately analytic in each variable. Continuity follows from the fact that we have chosen the map such that each of the piece-wise areas of definition will match up on their respective boundaries. Thus the result follows. However it will not be analytic, or even differentiable in general, on the boundary of these pieces. 
\end{proof}
 \begin{corollary}
 For all $r\in(0,1/2]$, $s\in[0,1/2]$, and all $w_1,w_2 \in \overline{\mathds{H}'}$,
        $$
        |Y_{r,s}(w_1)-Y_{r,s}(w_2) | \leq 0.9 |w_1 - w_2|.
        $$
\end{corollary}
\begin{proof}
    Due to Proposition \ref{GGrr}, we know that $G_{r,s}(w)$ is contracting in this sense. This immediately implies that, on each vertical strip where the definition of the function does not change $Y_{r,s}(w)$ is contracting. This argument is finished by matching up these sectors by continuity, proving that even if we pick points from different ``strips", the contraction property will hold.
\end{proof}
Despite us defining $\tilde{Y}^{\text{pre}_2}$ to be rather different to $G_{r,s}(x+iy)$ (and indeed sometimes constant with respect to $x$) the image of points is nonetheless always close to the image under $G_{r,s}(x+iy)$. In particular:
\begin{lemma}
    \label{close}
    There exists a $C_{\ref{close}}>0$ such that for all $r\in(0,1/2]$, $s\in[0,1/2]$ and for all $w\in\mathds{H}'$:
    $$
    |\tilde{Y}^{\text{pre}_2}_{r,s}(w) - G_{r,s}(w)|\leq C_{\ref{close}}.
    $$
\end{lemma}
Note that this immediately implies that $|Y_{r,s}(w) - (G_{r,s}(w)-G_{r,s}(0))|\leq 2C_{\ref{close}}$.
\begin{proof}
    There are a couple of key ingredients here - firstly the adjustments made in $\tilde{Y}^{\text{pre}_2}$ are very local, therefore by the contractivity property of $G_{r,s}$ they cannot be far from $\tilde{Y}^{\text{pre}_1}_{r,s}$. It remains to prove that $\tilde{Y}^{\text{pre}_1}_{r,s}$ is not far from $G_{r,s}$. But this follows entirely from the fact that $G_{r,s}(-1+iy)$ is uniformly close $G_{r,s}(-1/2+iy)$ by contractivity, which then by the minima argument is uniformly close to the minimum of $G_{r,s}$. Since the primary adjustment at this stage is to points with imaginary height below $G_{r,s}(-1+iy)$ we can conclude that $G_{r,s}$ and $\tilde{Y}^{\text{pre}_2}$ are uniformly close. Finally, by contractivity it is clear that $\mathrm{Im}(G_{r,s}(iy))$ will be uniformly close to $\mathrm{Im}(G_{r,s}(-1+iy)) = \mathrm{Im}(\tilde{Y}_{r,s}^{\text{pre}_2}(0))$. Therefore we can conclude that both are uniformly close by some constant $C_{\ref{close}}$>0.
\end{proof}
The maxima and minima statement is inherited from the statement about $G_{r,s}$ due to Lemma \ref{close}. Recall the definition of $x_{max}:=(1+s)/2r-1/2$:
\begin{corollary}
\label{maxminy}
     There exists a constant $C_{\ref{maxminy}}>0$ such that for all $r\in(0,1/2]$, $s\in[0,1/2]$ and for all $y\geq-1$:
$$
|\mathrm{Im}(Y_{r,s}(x_{max}+iy)) - \sup_{x\in\mathds{R}}\mathrm{Im}(Y_{r,s}(x+iy))| \leq C_{\ref{maxminy}},
$$
which also is such that for $\tilde{x}\in\{0,s/r\}$ and for all $y,r,s$:
$$
|\mathrm{Im}(Y_{r,s}(\tilde{x}+iy)) - \inf_{x\in\mathds{R}}\mathrm{Im}(Y_{r,s}(x+iy))| \leq C_{\ref{maxminy}}.
$$
\end{corollary}

    This corollary is clearly true with $C_{\ref{maxminy}} = C_{\ref{maxmin}}+C_{\ref{close}}$ due to the fact that $Y_{r,s}$ and $G_{r,s}$ are $C_{\ref{close}}$ close. We now prove that the change of coordinates will satisfy important \textbf{functional relations}, which we denote and number by \textbf{F1-3}.
\begin{proposition}[Fundamental Conditions for a Renormalisation Scheme]
    For all $r\in(0,1/2]$, $s\in[0,1/2]$, $Y_{r,s}:\mathds{H}' \mapsto \mathds{H}'$ preserves vertical lines, and has the following properties:
    \begin{enumerate}
        \item The map $Y_{r,s}$ is injective on $\overline{\mathds{H}'}$ and $Y_{r,s}(\overline{\mathds{H}'})\subset\mathds{H}'$
        \item (\textbf{F1}) for every $w\in\overline{\mathds{H}'}$,
        $$
        Y_{r,s}(w+1/r) = Y_{r,s}(w)+1,
        $$
        \item (\textbf{F2}) for every $t\geq -1$,
        $$
        Y_{r,s}(it+1/r-1)=Y_{r,s}(it) + 1 - r,
        $$
        \item (\textbf{F3}) and furthermore there is a symmetry about the two critical points:
        $$
        Y_{r,s}(s/r-w) -s= Y_{r,s}(w) ,
        $$
        \item for all $w_1,w_2 \in \overline{\mathds{H}'}$,
        $$
        |Y_{r,s}(w_1)-Y_{r,s}(w_2) | \leq 0.9 |w_1 - w_2|.
        $$
    \end{enumerate}
\end{proposition}
\begin{proof}
The remaining unproven points are that the three functional relations hold, and proving that $Y_{r,s}(x+iy)$ depends continuously on $r$ and $s$.
\\ \textbf{Step 1: The functional relations}:
\\ Because $Y_{r,s}$ has been defined on $[0,1/r]$  and extended by translations of $1/r$, it is immediate that the translation property of (F1) holds. The second function relation may be proved by only looking on $[0,1/r]$, but this actually follows from definition, as the piecewise definition given for $\tilde{Y}_{r,s}^{\text{pre}_2}$ ensures that vertical heights at $x=1/r-1$ are the same as at $x=1/r$. The piece-wise definition is also carefully chosen so that there is a symmetry across $s/2r-1/2$ is inherited in the $Y_{r,s}$, therefore we can also obtain (F3).
        \\\textbf{Step 2: Continuous dependence on $r$ and $s$}:
        \\ 
    This is clear for $G_{r,s}(x+iy)$, to conclude we note that the piece-wise definition depends continuously on $r$ and $s$ too, as $Y_{r,s}$ is equivalent to a translation of $G_{r,s}(x+iy)$ on each piece-wise slice.
\end{proof}
Now define $\mathcal{M}(r,s) := \frac{1}{2}\log(1/r)+\frac{1}{2}\log(1/(s+r))$. 
\begin{proposition}[Bi-Critical Geometry]
\label{height}
    There exists a constant $C_{\ref{height}}>0$ such that for all $r\in(0,1/2]$, $s\in[0,1/2]$ and for all $x\in[(1+s)/2r-1,(1+s)/2r]$ we have:
    \begin{equation*}
                \begin{aligned}
                    2\pi r y + \mathcal{M}(r,s) - C_{\ref{height}}  \leq 2\pi\mathrm{Im}(Y_{r,s}(x+iy)) \leq2\pi r y + \mathcal{M}(r,s) + C_{\ref{height}}.
                \end{aligned}
            \end{equation*}
\end{proposition}
\begin{proof}
    All that there is to do is calculate $\mathrm{Im}(Y_{r,s}((1+s)/2r+iy))$ and then we can conclude by contractivity. It is clear that $\mathrm{Im}(Y_{r,s}((1+s)/2r+iy)) \simeq_{C_{\ref{close}}} \mathrm{Im}(G_{r,s}((1+s)/2r+iy)) - \mathrm{Im}(G_{r,s}(0))$, and by the $g_r$ lemma the point $(1+s)/2r+iy$ lies upon the range where both $\log g_r(x+iy+1/2)$ and $\log g_r(x+iy+s/r+1/2)$ are uniformly close to $2\pi r y$. Therefore we just need to compute
    $\frac{1}{4\pi}\mathrm{Im}(G_{r,s}(0)) = \log g_r(1/2) + \log g_r(s/r+1/2)$. To achieve this we will prove that there exists a $D_{\ref{height}}>0$ such that $g_r(s/r+1/2) \asymp_{D_{\ref{height}}} (s+r)$. This yields:
    $$
    \frac{1}{4\pi}(\log g_r(1/2) + \log g_r(s/r+1/2))\simeq_{C_{\ref{height}}} \frac{1}{4\pi}(\log r+\log(s+r)),
    $$
    thus concluding the statement. To see that such a $D_{\ref{height}}$ exists, note that $g_r(s/r+1/2)=|e^{-3\pi r}-e^{-2\pi i s}|$, and for each fixed $r$, as $s\rightarrow 0$, we have $e^{-2\pi i s}\asymp 1-i2\pi s$, yielding $g_r(s/r+1/2)\asymp | 1-3\pi r - 1 + 2\pi s i| \asymp r+s$.
\end{proof}

This proposition essentially describes the sizes of the "close approaches" in the model. To see why this correlates with actual bi-critical holomorphic maps, see Section \ref{geometryofbicritjust}. Furthermore: 
\begin{corollary}
\label{heightsup}
    There exists a constant $C_{\ref{heightsup}}$ such that for all $r\in(0,1/2]$, $s\in[0,1/2]$ and for all $x\in[(1+s)/2r-1,(1+s)/2r]$ we have:
    \begin{equation*}
                \begin{aligned}
                    2\pi r y + \mathcal{M}(r,s) - C_{\ref{heightsup}}  \leq 2\pi\sup_{x\in\mathds{R}}\mathrm{Im}(Y_{r,s}(x+iy)) \leq2\pi r y + \mathcal{M}(r,s) + C_{\ref{heightsup}}.
                \end{aligned}
            \end{equation*}
\end{corollary}
\begin{proof}
    This follows directly from the prior maxima result in \ref{maxmin} combined with \ref{height}, using the fact that $x_{\max}$ will always lie in $[(1+s)/2r-1,(1+s)/2r]$. By taking $C_{\ref{heightsup}}:= C_{\ref{height}}+C_{\ref{maxmin}}$ the result immediately follows.
\end{proof}

\section{Defining The Bi-Critical Model}
\label{definemodel}
In this section we use the change of coordinates, along with the modified expansions of $\alpha$ and $\iota$, in order to define the dynamical system $(\mathds{M}_{[\alpha],[\iota]},\mathds{T}_{[\alpha],[\iota]})$ and its renormalisations mentioned in Theorem \ref{mainA}. Loosely speaking, we will ``reverse" the renormalisation tower, which arises from applying infinitely many renormalisations, in order to define the dynamics. Comparing with Sections 3-5 of \cite{cheraghi2022arithmetic}, it is clear that the added complexity of the bi-critical case does not especially change many of the steps to \textit{define} a model, because the key functional relations used remain the same, although it does fundamentally alter the nature of the model that we obtain. For this reason we generally refer the reader to \cite{cheraghi2022arithmetic} for further details.
\subsection{The Space $\mathds{M}_{[\alpha],[\iota]}$}

Now recall the sequence $\{\alpha_n,\iota_n\}_{n=0}^{\infty}$ defined in Section \ref{arithmeticconditions} for all $\alpha\in\mathds{R}\setminus\mathds{Q}$, and $\iota\in\mathds{R}$. Unless stated otherwise, this sequence will remain fixed throughout.
\begin{definition}
    Define:
    \[ \mathds{Y}_{n}(w) = \begin{cases} 
          Y_{\alpha_n,\iota_n}(w) & (\epsilon_n ,\delta_n) = (-1,-1),\\
          -s( Y_{\alpha_n,\iota_n}(w)) & (\epsilon_n ,\delta_n) = (1,1), \\
          Y_{\alpha_n,\iota_n}(w-\iota_n/\alpha_n)+\iota_n & (\epsilon_n ,\delta_n) = (-1,1) ,\\
          -s(Y_{\alpha_n,\iota_n}(w-\iota_n/\alpha_n)+\iota_n) & (\epsilon_n ,\delta_n) = (1,-1).
       \end{cases}
    \]
\end{definition}
So now each $\mathds{Y}_n$ is either orientation preserving or reversing, depending on $\epsilon_n$. 
\begin{lemma}
    \label{dsY}
The following will hold:

    \begin{enumerate}
        \item The map satisfies $\mathds{Y}_n(i[-1,+\infty))\subset i(-1,+\infty)$, $\mathds{Y}_n(0)=0$,
        \item and $\forall n\geq 0$ and all $w\in\overline{\mathds{H}'}$,
    
    \begin{equation*}
        \label{ff1}\mathds{Y}_n(w+1/\alpha_n) = \begin{cases} 
          \mathds{Y}_n(w)+1 & \epsilon_n=-1,\\
          \mathds{Y}_n(w)-1 & \epsilon_n=+1 ,\\
       \end{cases}
    \end{equation*}
    \item and $\forall n\geq0$ and all $t\geq-1$,
    
    \begin{equation*}
        \label{ff2}\mathds{Y}_n(it+1/\alpha_n-1) = \begin{cases} 
          \mathds{Y}_n(it)+(1-\alpha_n) & \epsilon_n=-1,\\
          \mathds{Y}_n(it)+(\alpha_n-1) & \epsilon_n=+1, \\
       \end{cases}
    \end{equation*}
    \item for all $w_1,w_2 \in \overline{\mathds{H}'}$,
        $$
        |\mathds{Y}_n(w_1)-\mathds{Y}_n(w_2) | \leq 0.9 |w_1 - w_2|,
        $$
    \item for all $n\geq0$ and all $y\geq -1 $:
    $$
    \mathds{Y}_n(\delta_n\iota_n/\alpha_n+it)=\mathds{Y}_n(it)-\epsilon_n\iota_n.
    $$
    \end{enumerate}
\end{lemma}
    These all follow directly from applying the functional relations and fundamental properties of $Y_{r,s}$ proven above. See \cite{cheraghi2022arithmetic}, Section 3, for more details. The only new relation in that list is number $5$, but this arises from the extra symmetry that comes from ``swapping" the two marked points at $\iota$ and at $0$, because each $Y_{r,s}$ function takes the same imaginary values at any point $iy$, as it does at any point $s/r+iy$. Now define:
$$
I^0_n=\{w \in \mathds{H}': \mathrm{Re}(w) \in [0,1/\alpha_n] \},
$$
\begin{equation*}
    \begin{aligned}
        J^0_n = \{ w \in I^0_n: \mathrm{Re}(w) \in [1/\alpha_n-1,1/\alpha_n] \}, && K^0_n = \{ w \in I^0_n: \mathrm{Re}(w) \in [0,1/\alpha_n-1] \}.
    \end{aligned}
\end{equation*}
Fix an arbitrary $n\geq0$ and if $\epsilon_{n+1}=-1$ let:
$$
I^{j+1}_n = \bigcup_{l=0}^{a_n-2} (\mathds{Y}_{n+1}(I^j_{n+1})+l) \bigcup (\mathds{Y}_{n+1}(K^j_{n+1})+a_n - 1) ,
$$
while if $\epsilon_{n+1} = +1$:
$$
I^{j+1}_n = \bigcup_{l=1}^{a_n} (\mathds{Y}_{n+1}(I^j_{n+1})+l) \bigcup (\mathds{Y}_{n+1}(J^j_{n+1})+a_n + 1).
$$
Regardless of the sign define:
\begin{equation*}
\begin{aligned}
        J^{j+1}_n = \{ w \in I^{j+1}_n: \mathrm{Re}(w) \in [1/\alpha_n-1,1/\alpha_n] \} & & K^{j+1}_n = \{ w \in I^{j+1}_n: \mathrm{Re}(w) \in [0,1/\alpha_n-1] \}.
    \end{aligned}
\end{equation*}
These will be closed and connected subsets of $\mathds{C}$, bounded by piece-wise analytic curves, and in addition $\{\mathrm{Re}(w)\mid w\in I^j_n\} = [0,1/\alpha_n]$ (see Figure \ref{imageconstruct}).
\begin{figure}[H]
\captionsetup{justification=centering}
    \centering
\includegraphics[width=0.4\textwidth]{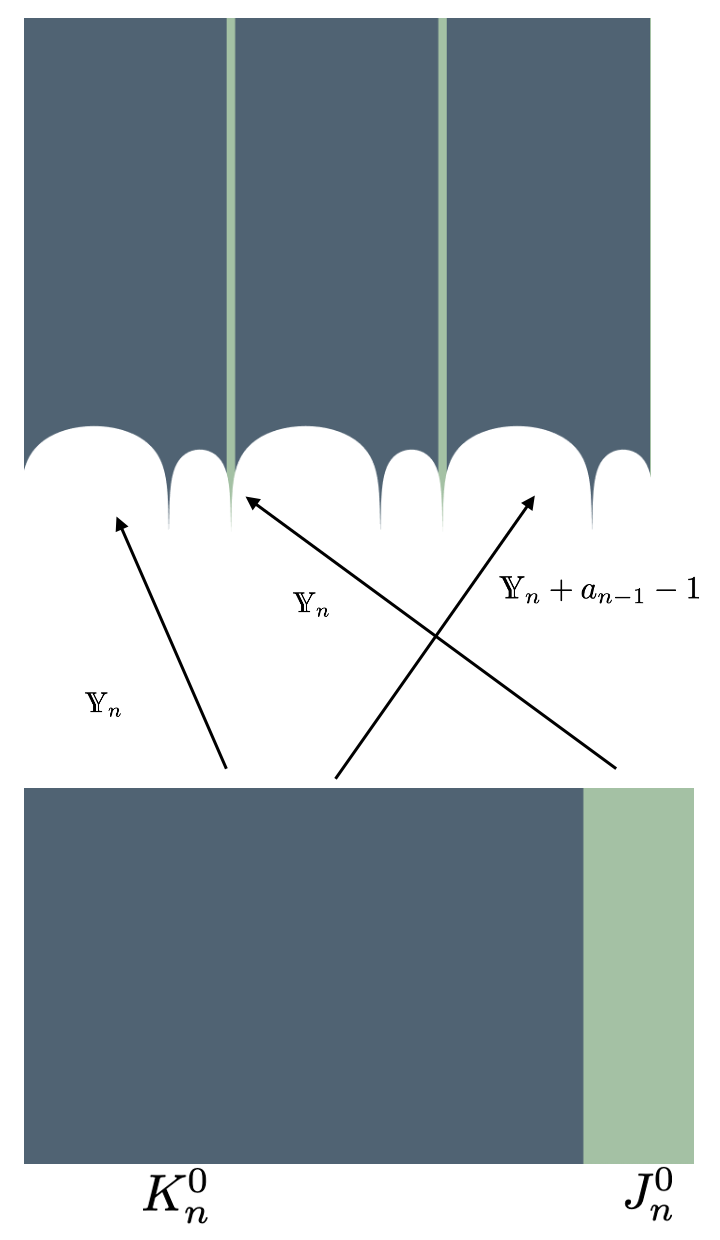}

    \caption{\label{imageconstruct}An image of the construction for $\epsilon_n=-1$. The set $I^1_{n-1}$ is pictured in the upper row.}
\end{figure}
\begin{corollary}
    For all $n,j\geq0$: 
    \begin{enumerate}
        \item for all $w\in\mathds{H}'$,  $w\in I^j_n$ if and only if $w+1\in I^j_n$,
        \item for all $t\in\mathds{R}$, $it\in I^j_n \iff it +1/\alpha_n\in I^j_n$.
    \end{enumerate}
\end{corollary}
\begin{proof}
    Both statements clearly hold for $j=0$ and all $n\geq0$, so we proceed by an inductive argument on $j$. In particular equation \ref{ff1} combined with (ii) imply that $it\in I^{j+1}_n$ if and only if $it+1\in I^{j+1}_n$, so then by the definition of the $I^{j+1}_n$ the result follows. For part (ii) we need to separate $\epsilon_{n+1}=\pm1$. Fixing $\epsilon_{n+1}=-1$ note that for $t,t'\in[-1,+\infty)$ satisfying $\mathds{Y}_{n+1}(it')=it$, we have 
\[\mathds{Y}_{n+1}(it'+ 1/\alpha_{n+1} - 1)+a_n-1=
\mathds{Y}_{n+1}(it')+ (1-\alpha_{n+ 1}) + a_n - 1=
it + 1/\alpha_n .\] 
If $i t \in I_n^{j+1}$, then by definition of the sets there is $it' \in I_{n+1}^j$ with $\mathds{Y}_{n+1}(it')=it$. 
By the induction hypotheses $it'+ 1/\alpha_{n+1} - 1 \in K_{n+1}^j$, so 
$it+1/\alpha_n \in I_n^{j+1}$. On the other hand, if $it+1/\alpha_n \in I_n^{j+1}$, then by the induction hypotheses, 
there is $it'\in I_{n+1}^j$ such that $it'+1/\alpha_{n+1} -1 \in K_{n+1}^j$ and 
\[\mathds{Y}_{n+1}(it'+1/\alpha_{n+1} -1) + (a_n-1) = it+1/\alpha_n,\] which implies that 
$it+1/\alpha_n=\mathds{Y}_{n+1}(it')+1/\alpha_n$. Hence, $it= \mathds{Y}_{n+1}(it')$, which implies that $it\in I_n^{j+1}$. The proof when $\epsilon_{n+1}=-1$ is similar. For more details about this proof see Section 3 of \cite{cheraghi2022arithmetic} as it is exactly analogous.
\end{proof}
We also note that it follows via a simple inductive argument that for all $n\geq -1$, $j\geq 0$, $I^{j+1}_n\subset I^{j}_n$, allowing us to define:
$$
I_n := \cap_{j\geq0}I^{j}_n.
$$
Each $I_n$ will consist of closed half-infinite vertical lines, and it may or may not be connected. Now note by the corollary $it\in I_{-1}$ if and only if $it+1\in I_{-1}$, which allows us to write:
$$
\mathds{M}_{[\alpha],[\iota]}=\{s(e^{2\pi i w})\mid w\in I_{-1}\} \cup \{ 0\},
$$
where we say $\mathds{A}_{[\alpha],[\iota]} := \partial \mathds{M}_{[\alpha],[\iota]}$, which will represent the model for the post-critical set.
\begin{proposition}
\label{+1beta}
    For every $([\alpha],[\iota]) \in \mathds{R}/\mathds{Q}/\mathds{Z}\times\mathds{R}/\mathds{Z}$, the following holds:
    \begin{itemize}
        \item $\mathds{M}_{[\alpha],[\iota]}$ is a compact set which is star-like about $0$, $\{0,+1,e^{2\pi i\iota}\}  \subset \mathds{M}_{[\alpha],[\iota]}$, and $\mathds{M}_{[\alpha],[\iota]} \cap (1,\infty) = \emptyset$ ,
        \item 
        $\mathds{M}_{[\alpha],[\iota]}$ is well defined as depending on $([\alpha],[\iota]) \in \mathds{R}/\mathds{Q}/\mathds{Z}\times\mathds{R}/\mathds{Z}$, that is 
        $\forall (m,n)\in\mathds{Z}^2$,$\mathds{M}_{[\alpha],[\iota]} = \mathds{M}_{[\alpha+n],[\iota+m]}$, and in addition $s(\mathds{M}_{[\alpha],[\iota]}) = \mathds{M}_{[-\alpha],[-\iota]}$. 
    \end{itemize}
\end{proposition}
\begin{proof}
    Since $\mathds{Y}_n$ sends vertical lines to vertical lines, this implies that each $I^j_n$ and indeed $I_n$ is the graph above some function, i.e. each consists of half-infinite lines. This implies that $\mathds{M}_{[\alpha],[\iota]}$ is star like about $0$, and it is clearly compact as it is closed and bounded.
    \par
    Since $\mathds{Y}_n(0)=0$ for all $n$ it is clear that $+1\in\mathds{M}_{[\alpha]}$ and indeed by the contractivity property this will in fact be an end point. To see that $e^{2\pi i\iota}$ lies in $\mathds{M}_{[\alpha],[\iota]}$ it is sufficient to note that $\mathrm{Re}\mathds{Y}_n(\iota_n/\alpha_n)=0$ for all $n$ which follows by the modified extra functional relation \textbf{F3}. Since $\iota_0 = \delta_0 b_0 \alpha_0+\delta_1b_1\alpha_1\alpha_0$ this implies that the $I_n$ set really does preserve this relation in the sense that the half line at $\iota_n/\alpha_n$ always goes up until $0$. Then by contractivity this once again implies $e^{2\pi i\iota}$ is an end point.
    \par
    For the second point we note that $\mathds{M}_{[\alpha],[\iota]}$ is defined entirely from the continued fraction entries of $\alpha$ and $\iota$. These will be the same under translations of integers, immediately giving the result. A similar argument was explained in depth in Section 3 of \cite{cheraghi2022arithmetic}. To show that $s(\mathds{M}_{[\alpha],[\iota]}) = \mathds{M}_{[-\alpha],[-\iota]}$ it amounts to finding the continued fractions for $-\alpha$ and $-\iota$ which will have a negative sign for $\epsilon_0$, clearly yielding the necessary conditions for the equation to hold.
\end{proof}
\begin{lemma}
\label{renorms}
    For $\alpha\in(0,1/2)\setminus\mathds{Q}$, $\mathds{M}_{[1/\alpha],[\iota/\alpha]}= \{s(e^{2\pi i \alpha_0w}) \mid w\in I_0\}\cup\{0\}$, while for $\alpha\in(-1/2,0)$, $\mathds{M}_{[-1/\alpha],[-\iota/\alpha]}= \{s(e^{2\pi i \alpha_0w}) \mid w\in I_0\}\cup\{0\}$.
\end{lemma}
\begin{proof}
    In Section \ref{arithmeticconditions} an equivalent notion to the action of gauss map was defined which generates the sequence of numbers $(\alpha_n,\iota_n)$. Because this sequence is defined via 
    $$
    (\alpha_n,\iota_n)\mapsto (d(1/\alpha_n,\mathds{Z}),d(\iota_n/\alpha_n,\mathds{Z}))=(d(-1/\alpha_n,\mathds{Z}),d(-\iota_n/\alpha_n,\mathds{Z})),
    $$ 
    the result follows by noting that the set $\{s(e^{2\pi i \alpha_0w}) \mid w\in I_0\}\cup\{0\}$ is clearly equivalent in definition to $s(\mathds{M}_{[1/\alpha],[\iota/\alpha]})$, but this is equal to $\mathds{M}_{[-1/\alpha],[-\iota/\alpha]}$ by the third point of the above proposition.
\end{proof}
Although we do not have that $\mathds{M}_{[\alpha],[\iota]}=e^{-2\pi i\iota}\mathds{M}_{[\alpha],[-\iota]}$ in general, it is possible to say the following:
\begin{lemma}
\label{minuslem}

For every $([\alpha],[\iota]) \in \mathds{R}/\mathds{Q}/\mathds{Z}\times\mathds{R}/\mathds{Z}$ there exists a constant $C_{\ref{minuslem}}>1$ such that:
$$
\mathds{M}_{[\alpha],[\iota]}\subset C_{\ref{minuslem}}e^{-2\pi i\iota}\mathds{M}_{[\alpha],[-\iota]},
$$
and
$$
C_{\ref{minuslem}}^{-1}e^{-2\pi i\iota}\mathds{M}_{[\alpha],[-\iota]} \subset \mathds{M}_{[\alpha],[\iota]}.
$$
\end{lemma}
\begin{proof}
    For shorthand use $_{\alpha,\iota}I_j^n$ to denote the $I_n$ sets that arise from the initial choice of $(\alpha,\iota)$. We cannot get equality here because the curves that define the boundaries of the $_{\alpha,\iota}I_j^n$ will differ from that of the $_{\alpha,-\iota}I_j^n-\iota_n/\alpha_n$ at particular points because we rely on the second function relation to ``match up" the dynamics as we return to the initial sector. However, due to contractivity, this deviation will be at most by $0.9$ at each stage, thus yielding a maximal difference of the geometric sum, which is $10$. So take $C_{\ref{minuslem}}=e^{20\pi}$ and the result follows.
\end{proof}
We also note the following:
\begin{proposition}
    The space $\mathds{M}_{[\alpha],[\iota]}$ depends continuously on $([\alpha],[\iota])\in(\mathds{R}\setminus\mathds{Q}/\mathds{Z})\times(\mathds{R}/\mathds{Z}$ in the Hausdorff topology.
\end{proposition}
\begin{proof}
    This essentially follows directly from the fact that the change of coordinates depends continuously on $(\alpha,\iota)$, and furthermore that when $(\alpha',\iota')$ and $(\alpha,\iota)$ are sufficiently close, the ``continued fraction with respect to $\alpha$" will be the same up until some large $N$. By taking all of the multiples of $\alpha$, we therefore obtain a finite set of points in $I_{-1}$ which define the boundary of the set but are uniformly close to points in the $(\alpha',\iota')$ case. By taking a large enough ball in a manner that depends on what the large $N$ was, this will cover the entire boundary, yielding Hausdorff continuity. See \cite{cheraghi2022arithmetic} for more details.
\end{proof}
 
\subsection{The Map $\mathds{T}_{[\alpha],[\iota]}$}
Fix $(\alpha,\iota) \in \mathds{R}\setminus \mathds{Q} \times \mathds{R}$. Given $w_{-1} \in I_{-1}$, inductively define the integers $l_i$ and points $w_{i+1}$ such that:
$$
0\leq\mathrm{Re}(w_i-l_i)<1, \:if\: \epsilon_{i+1} = -1; \; \; \; -1 \leq \mathrm{Re}(w_i-l_i)\leq0, \:if \:\epsilon_{i+1}=+1.
$$
These correspond directly to the ``expansion of $x$ with respect to $\alpha$" referenced in Section \ref{arithmeticconditions}.
and 
$$
\mathds{Y}_{i+1}(w_{i+1}) + l_i = w_i.
$$
It follows that for all $n\geq0$, we have:
$$
w_{-1} = (\mathds{Y}_0 + l_{-1})\circ(\mathds{Y}_1 + l_0)\circ...\circ((\mathds{Y}_n + l_{n-1})(w_n),
$$
and in addition
$$
0\leq l_i\leq a_i +\epsilon_{i+1}, \: \: \text{and} \: \: 0 \leq \mathrm{Re}(w_i) \leq 1/\alpha_i.
$$
To see this quickly just note that the real part of the initial $w_{-1}\in I_{-1}$ always lies in $[0,1]$, hence inductively we get that the integer translations, $l_i$, correspond to the $a_i$ that arise in the continued fraction expansion. Call $(w_i,l_i)_{i\geq-1}$ the trajectory of $w_{-1}$, with respect to $\alpha$. We remark that if one sets $w_{-1} = \iota$, then $(\mathrm{Re}(w_i),l_i)_{i\geq-1} = (\tilde\iota_i,b_i)_{i\geq-1}$, where $\tilde\iota_i$ denotes an alternative Ostrowski expansion with use of fundamental domain $[0,1]$, instead of $[-1/2,1/2]$ as is employed in this paper. Define the map, $\tilde{T}_{\alpha,\iota}: I_{-1} \rightarrow I_{-1}$,
as follows: If $w_{-1} \in I_{-1}$, and $(w_i,l_i)_{i\geq-1}$ is its trajectory, then:
\begin{enumerate}
    \item if there is a $n\geq0$ such that $w_n \in K_n$, and for all $0\leq i \leq n-1$, $w_i \in I_i \setminus K_i$, then:
    $$
    \tilde{T}_{\alpha,\iota}(w_{-1}) = (\mathds{Y}_0 + \frac{\epsilon_0+1}{2})\circ(\mathds{Y}_1 + \frac{\epsilon_1+1}{2})\circ...\circ(\mathds{Y}_n + \frac{\epsilon_n+1}{2})(w_n+1),
    $$
    \item if for all $n\geq 0$, $w_n \in I_n\setminus K_n$, then:
    $$
    \tilde{T}_{\alpha,\iota}(w_{-1}) = \lim_{n\rightarrow+\infty}(\mathds{Y}_0 + \frac{\epsilon_0+1}{2})\circ(\mathds{Y}_1 + \frac{\epsilon_1+1}{2})\circ...\circ(\mathds{Y}_n + \frac{\epsilon_n+1}{2})(w_n+1-1/\alpha_n).
    $$
\end{enumerate}
\begin{proposition}
\label{welldefineditlde}
For all $\alpha\in\mathds{R}\setminus\mathds{Q}$ and $\iota\in\mathds{R}$, the map
    $$
\tilde{T}_{\alpha,\iota} : I_{-1}/\mathds{Z} \rightarrow I_{-1}/\mathds{Z},
$$
is well defined, continuous and injective.
\end{proposition}
\begin{proof}
    The main idea of the proof for the above statement is to partition the set $I_{-1}$ into infinitely many pieces, 
where the map is continuous on each piece, then show that the maps on the pieces match at the boundary points. In particular let
\[W^n= \{ w_{-1} \in I_{-1} \mid \text{ for all } 0 \leq i \leq n-1, w_i \in I_i \setminus K_i, \text{and}\: w_n \in K_n\},\]  
and
\[V^n =\{w_{-1} \in I_{-1} \mid \text{ for all } 0 \leq i \leq n, w_i  \in I_i \setminus K_i\},\]
and set $V^{\infty}:=\cap_{n\geq 0} V^n$, which
clearly implies $I_{-1}=\cup_{n\geq0}W^n\cup V^{\infty}$. It is clear that the maps $w_{-1}\mapsto w_n$ is continuous and injective on both sets $W^n$ and $V^n$. This directly implies that $\tilde{T}_{\alpha,\iota}$ is continuous and injective on $W^n$. To obtain continuity on $V^\infty$ we note that each $V^n$ will in fact be relatively open in $I_{-1}$. In fact it may sometimes be empty (for instance when $\epsilon_i=-1$ for all $i\geq0$), but assuming it is nonempty, we essentially make use of the second functional relation along with the contractivity property to ensure that the limit map generated by the chain of change of coordinates actually exists and is continuous.
\par
Finally to conclude we have to ensure that the map matches up along the boundaries, but this follows from the second functional relation. There is also the matter of whether the limit used in the second case of the definition even exists. For this we note that the composition mainly consists of small translations which contract down to nothing, apart from the final $w_n+1-1/\alpha_n$ translation. But then because each $w_n$ is in $I_n\setminus K_n$ this implies that $w_n-1/\alpha_n$ is always within a constant of $0$, so the limit must converge. For further details, see Section 4 of \cite{cheraghi2022arithmetic}, as the proof is exactly the same since it only depends on the first two functional relations.
\end{proof}
\begin{proposition}

For every $([\alpha],[\iota]) \in \mathds{R}/\mathds{Q}/\mathds{Z}\times\mathds{R}/\mathds{Z}$, the map $\tilde{T}_{\alpha,\iota}: I_{-1} \to I_{-1}$ induces a homeomorphism
\[
\mathds{T}_{[\alpha],[\iota]} : \mathds{M}_{[\alpha],[\iota]} \to \mathds{M}_{[\alpha],[\iota]},
\]
via the projection $w \mapsto s(e^{2\pi i w})$
that is,$
s\bigl(e^{2\pi i \tilde{T}_{\alpha,\iota}(w)}\bigr)
=
\mathds{T}_{[\alpha],[\iota]}\bigl(s(e^{2\pi i w})\bigr)$ for all  $w \in I_{-1}$. Further, there exists a continuous function dependent on $([\alpha],[\iota])$, $h_{[\alpha],[\iota]}: \mathds{R}^{\geq 0}\times\mathds{R}/\mathds{Z}\mapsto \mathds{R}^{\geq0}$ such that:
    $$
    \mathds{T}_{[\alpha],[\iota]} (re^{2\pi i\theta}) = h_{[\alpha],[\iota]}(r,\theta) e^{2\pi i (\theta + \alpha)}.
    $$
\end{proposition}
\begin{proof}
    The first part follows directly from the definition of the set $\mathds{M}_{[\alpha],[\iota]}$, and the fact that $\tilde{T}_{\alpha,\iota}$ is well defined and continuous on $I_{-1}/\mathds{Z}$ as in \ref{welldefineditlde}. To show that $\mathds{T}_{[\alpha],[\iota]}$ acts as a rotation by $2\pi\alpha$ in the tangential direction, it is enough to show that $\tilde{T}_{[\alpha],[\iota]}$ acts as a translation by $-\alpha$ on both the sets $W^n$ (for which the limiting union set is dense). But because each change of coordinates map $\mathds{Y}_n$ acts by multiplication of $-\alpha_n$ or $\alpha_n$ on the real line, it follows that the special definition of translation below will pull up to translation by $\alpha$. For more details see Section 4 of \cite{cheraghi2022arithmetic}. This is sufficient because the $\mathds{Y}_n$ sends vertical lines to vertical lines.
\end{proof}
\begin{proposition}
    For every $([\alpha],[\iota]) \in \mathds{R}/\mathds{Q}/\mathds{Z}\times\mathds{R}/\mathds{Z}$, we have $s\circ\mathds{T}_{[\alpha],[\iota]}\circ s = \mathds{T}_{[-\alpha],[-\iota]}$ on $\mathds{M}_{[-\alpha],[-\iota]}$.
\end{proposition}
\begin{proof}
Since this result also induces an action on $\iota$, it doesn't just follow from Cheraghi's work. Notice $s(\mathds{M}_{[\alpha],[\iota]})=\mathds{M}_{[-\alpha],[-\iota]}$, thus $s\circ\mathds{T}_{[\alpha],[\iota]}\circ s$ is defined here. When $(\alpha,\iota)$ changes to $(\alpha,-\iota)$, note that $(\epsilon_0,\delta_0)$ changes to $(-\epsilon_0,-\delta_0)$, but all subsequent numbers $(\alpha_i,\iota_i)$ and $(\epsilon_i,\delta_i)$ remain the same. Thus by definition we get that $-s\circ\tilde{\mathds{T}}_{[\alpha],[\iota]}\circ -s$ on $I_{-1}$, projecting onto $\mathds{M}_{[-\alpha],[-\iota]}$ the result follows.
\end{proof}
Recall the $Q_\alpha$ functions defined for $\alpha\in[-1/2,1/2]\setminus\{0\}$ in definition \ref{dyncurvem}. We now prove clause (iv) of Theorem \ref{mainA}, which provides the bi-critical geometry.
\begin{proposition}
There exists a $C_{\ref{mainA}}>0$ such that for every $([\alpha],[\iota]) \in \mathds{R}/\mathds{Q}/\mathds{Z}\times\mathds{R}/\mathds{Z}$ and $0\leq k\leq 1/|\alpha_0|$ we have:
$$
     C_{\ref{mainA}}^{-1}\sqrt{\frac{|\iota_0|+|\alpha_0|}{|\alpha_0|}Q_{\alpha_0}(k)Q_{\alpha_0}\left(k-\frac{\delta_0\iota_0}{\epsilon_0\alpha_0}\right)}<|\mathds{T}^{\circ k}_{[\alpha],[\iota]}(+1)|<C_{\ref{mainA}}\sqrt{\frac{|\iota_0|+|\alpha_0|}{|\alpha_0|}Q_{\alpha_0}(k)Q_{\alpha_0}\left(k-\frac{\delta_0\iota_0}{\epsilon_0\alpha_0}\right)}.
     $$
\end{proposition}
\begin{proof}
To obtain the result, we can evaluate the approximate value of the imaginary parts of the $\mathds{Y}_0$ function on the line the $y=0$. This is because for every $k$ iterates of the point $0$ under $\tilde{T}_{\alpha,\iota}$ which satisfy $0\leq k<1/\alpha_0$ will remain in the $K_n$. So for such $k$ we will have:
$$
\tilde{T}^k_{\alpha,\iota}(0)= \mathds{Y}_0(k).
$$
The point then is that $|\mathds{T}^{\circ k}_{[\alpha],[\iota]}(+1)| = |s(e^{2\pi i (\mathds{Y}_n(k))})|=e^{-2\pi \mathrm{Im}(\mathds{Y}_0(k)}$. We first prove the result assuming that $(\epsilon_0,\delta_0)=(-1,-1)$, so then in particular:
$$
\tilde{T}^k_{\alpha,\iota}(0) = Y_{\alpha_0,\iota_0}(k).
$$
In Section 4 of \cite{cheraghi2022arithmetic}, Cheraghi discussed his uni-critical model map, $\mathds{T}_{[\alpha]}$, which is defined in an exactly analogous way to $\mathds{T}_{[\alpha],[\iota]}$, but using the change of coordinates $Y_r$ mentioned in Lemma \ref{cheraghi}. It satisfies the following inequality for some uniform $C>0$ across all $\alpha\in\mathds{R}\setminus\mathds{Q}$: $    C^{-1}Q_{\alpha}(k)<|\mathds{T}^{\circ k}_{[\alpha]}(+1)|<CQ_{\alpha}(k)$ on $0\leq k\leq 1/|\alpha|$. We can apply this result to our map, because the geometry depends on the $G_{r,s}$ functions, which are an averaged sum of Cheraghi's change of coordinates. We note that Cheraghi's result associates the $Y_{\alpha_0}(k)$ and $Y_{\alpha_0}(k-\iota/\alpha)$ functions with $Q_\alpha(k)$ and $Q_{\alpha}(k-\iota/\alpha)$ via $s(e^{2\pi i (Y_{\alpha_0}(k))})\asymp Q_{\alpha_0}(k)$, as in Proposition 4.8 of \cite{cheraghi2022arithmetic}.
\par
Now note that we have:
\begin{equation*}
    \begin{aligned}
        \mathrm{Im}(Y_{r,s}(w)) &\simeq_C \mathrm{Im}(G_{r,s} -G_{r,s}(0))\\
        &= \frac{1}{4\pi}(\log g_r(w+1/2)+\log g_r(w+1/2-s/r)-\log g_r(1/2)-\log g_r(s/r+1/2)\\
        &= \frac{1}{4\pi}((\log g_r(w+1/2)-\log g_r(1/2))+( \log g_r(w+1/2-s/r)-\log g_r(1/2))\\
        &-\log g_r(s/r+1/2)+\log g_r(1/2))\\
        &=\frac{1}{2}(\mathrm{Im}(Y_r(w))+\mathrm{Im}(Y_r(w-s/r)))+\frac{1}{4\pi}(\log g_r(1/2)-\log g_r(s/r+1/2)).
    \end{aligned}
\end{equation*}
Now noting by the definition of $g_r$ we have $\log g_r(1/2)\simeq \log(r)$, and by following the same argument in the proof of Lemma \ref{maxmin}, we have $e^{-2\pi i s}\asymp 1-i2\pi s$, yielding $g_r(s/r+1/2)\asymp | 1-3\pi r - 1 + 2\pi s i| \asymp r.+s$. By Cheraghi's result this then obtains:
\begin{equation*}
    \begin{aligned}
        \mathds{T}^k_{[\alpha],[\iota]}(0)&=e^{-2\pi(\frac{1}{2}(\mathrm{Im}(Y_{\alpha_0}(k))+\mathrm{Im}(Y_{\alpha_0}(k-\iota_0/\alpha_0)))+\frac{1}{4\pi}(\log g_{\alpha_0}(1/2)-\log g_{\alpha_0}(\iota_0/\alpha_0+1/2)))}\\
        &=e^{-2\pi(\frac{1}{2}(\mathrm{Im}(Y_{\alpha_0}(k))}e^{-2\pi(\frac{1}{2}(\mathrm{Im}(Y_{\alpha_0}(k-\iota_0/\alpha_0)))}e^{-2\pi(\frac{1}{4\pi}(\log g_{\alpha_0}(1/2)-\log g_{\alpha_0}(\iota_0/\alpha_0+1/2)))} \\
        &\asymp \sqrt{Q_{\alpha_0}(k)Q_{\alpha_0}(k-\iota_0/\alpha_0)(\iota_0+\alpha_0)\alpha_0^{-1}}.
    \end{aligned}
\end{equation*}
This concludes the result for the $(\epsilon_0,\delta_0)=(-1,-1)$ case, for the other possible permutations we note that if $\iota$ has a different sign as $\alpha$ then the direction of the change of coordinates, $\mathds{Y}_{0}(k)$, as $k$ increases will change, but this is taken care of by taking $\iota$ and $\alpha$ to be ``signed" and using $\delta_0\iota_0$ and $\epsilon_0$ in the equation.
\end{proof}
\subsection{The Renormalisation Operator}
In this section we introduce our renormalisation operator. It is in the spirit of the renormalisation introduced by Douady \cite{MR936853} and further used for small divisor problems by Yoccoz \cite{MR1367353} (see Figure \ref{renormnorm}).
\begin{figure}[ht]

\captionsetup{justification=centering}
    \centering
\includegraphics[width=0.6\textwidth]{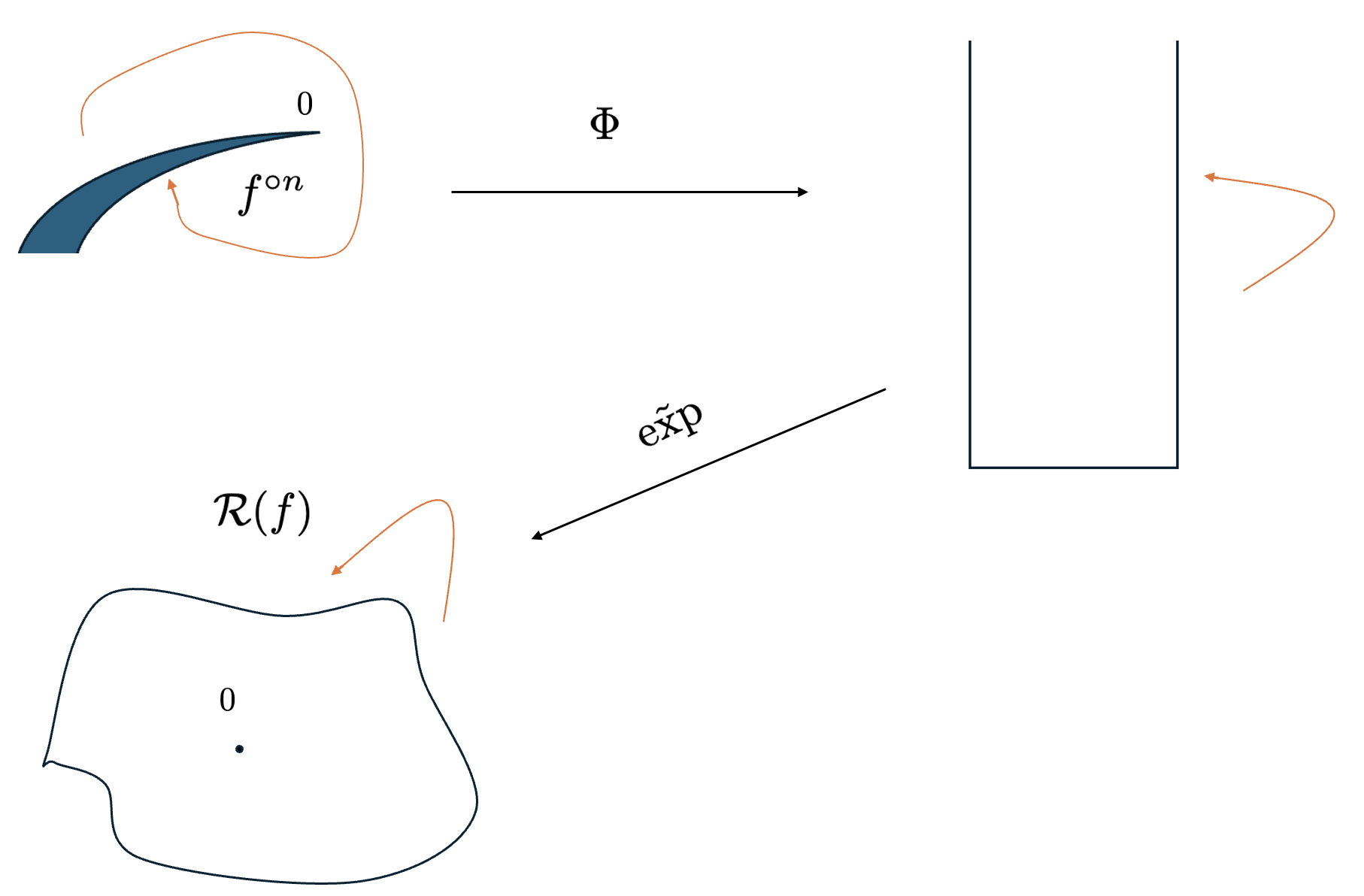}
    \caption{\label{renormnorm} Renormalisation for holomorphic maps with irrationally indifferent fixed points.}
\end{figure}
In order to apply the model to actual holomorphic maps, it is necessary to establish an appropriate analogue of the renormalisation operator. This is because defining a conjugacy from a holomorphic map to $\mathds{T}_{[\alpha],[\iota]}$ will necessarily involve comparing renormalisations. Define a sector by:
$$
S_{\alpha}=\{ z\in\mathds{M}_{[\alpha],[\iota]}\setminus \{0\} \:\:|\: \mathrm{arg}(z)\in[0,2\pi\alpha)+2\pi\mathds{Z}\}.
$$
Because $\mathds{T}_{[\alpha],[\iota]}$ acts as a rotation by $2\pi\alpha$ in the tangential direction, there will exist a return map to this sector, or an integer $k_z$ dependent on each $z\in S_{\alpha}$ such that $\mathds{T}^{\circ k_z}(z) \in S_\alpha$. Now define
$
\psi_{\alpha,\iota}: \mathds{H}'\rightarrow \mathds{C}\setminus\{0\}
$
as
$
\psi_{\alpha,\iota} := s(e^{2\pi i \mathds{Y}_0(w)})
$
Note that:
$$
S_\alpha \subset \psi_{\alpha,\iota} (\{ w \in \mathds{H}' \mid \mathrm{Re}(w)\in[0,1)\}).
$$
There is a continuous inverse branch of $\psi_{\alpha,\iota}$, $\phi_{\alpha,\iota}$ defined on $S_\alpha$ going into $\{ w \in \mathds{H}' \mid \mathrm{Re}(w)\in[0,1)\}$. This will be the analogue of the perturbed fatou coordinate for $\mathds{T}_{[\alpha],[\iota]}$. Now the return map for $\mathds{T}_{[\alpha],[\iota]}$ will induce a map:
$$
h_{\alpha,\iota}(w) = \phi_{\alpha,\iota} \circ \mathds{T}^{\circ k_{\psi_{\alpha,\iota}(w)}}_{[\alpha],[\iota]} \circ \psi_{\alpha,\iota}(w),
$$
which will project under $w\mapsto e^{2\pi i w}$ to a map $E_{\alpha,\iota}$ defined on $e^{2\pi i\phi_{\alpha,\iota}(S_{\alpha})}\subset \mathds{C}\setminus\{0\}$, which can be extended by setting $E_{\alpha,\iota}(0)=0$. This map, and it's domain of definition will be called the renormalisation $\mathcal{R}(\mathds{T}_{[\alpha],[\iota]}:\mathds{M}_{[\alpha],[\iota]} \rightarrow \mathds{M}_{[\alpha],[\iota]})$. For the case when $\alpha\in(-1/2,0)$ we simply apply complex conjugation to obtain a rotation number of $-\alpha$ and then proceed. Then the following holds:
\begin{proposition}
    For every $\alpha,\iota \in (-1/2,1/2)\setminus \mathds{Q}\times (-1/2,1/2)$ we have:
    $$
    \mathcal{R}(\mathds{T}_{[\alpha],[\iota]}:\mathds{M}_{[\alpha],[\iota]} \rightarrow \mathds{M}_{[\alpha],[\iota]}) = (\mathds{T}_{[-1/\alpha],[-\iota/\alpha]}:\mathds{M}_{[-1/\alpha],[-\iota/\alpha]} \rightarrow \mathds{M}_{[-1/\alpha],[-\iota/\alpha]}).
    $$
\end{proposition}
\begin{proof}
    The definition of the ``Fatou coordinates" given for $\phi_{\alpha,\iota}$ unwinds the change of coordinates map $\mathds{Y}_0$. Precisely we mean that $S_{\alpha}$ as defined above will satisfy $\mathds{Y}_0^{-1}(\psi_{\alpha,\iota}(S_{\alpha})) = I_0$, hence by \ref{renorms} it follows that the space upon which the renormalised map is defined is $\mathds{M}_{[-1/\alpha],[-\iota/\alpha]}$. Finally we prove that the induced map $E_{\alpha,\iota}=\mathds{T}_{[-1/\alpha],[-\iota/\alpha]}$.
    Recall the definitions: If for all $0\leq n\leq m$, $w_n \in I_m\setminus K_m$ but $w_n\in K_n$, then:
     $$
    \tilde{T}_{[-1/\alpha],[-\iota/\alpha]}(w_{-1}) = (\mathds{Y}_0 + \frac{\epsilon_1+1}{2})\circ(\mathds{Y}_1 + \frac{\epsilon_2+1}{2})\circ...\circ(\mathds{Y}_n + \frac{\epsilon_n+1}{2})(w_n+1),
    $$
    whereas if for all $n\geq 0$, $w_n \in I_n\setminus K_n$, then:
    $$
    \tilde{T}_{[-1/\alpha],[-\iota/\alpha]}(w_{-1}) = \lim_{n\rightarrow+\infty}(\mathds{Y}_0 + \frac{\epsilon_1+1}{2})\circ(\mathds{Y}_1 + \frac{\epsilon_2+1}{2})\circ...\circ(\mathds{Y}_n + \frac{\epsilon_n+1}{2})(w_n+1-1/\alpha_n).
    $$

The map $\mathds{Y}_0+ (1+\epsilon_0)/2= \mathds{Y}_0+1: I_0 \to I_{-1}$ is a homeomorphism, so this allows us to consider the map
$$
\hat{T}_{[\alpha],[\iota]}= (\mathds{Y}_0+1)^{-1} \circ \tilde{T}_{[\alpha],[\iota]} \circ (\mathds{Y}_0+1): I_0 \to I_0.
$$
Via the conjugation, the return map of $\hat{T}_{[\alpha],[\iota]}$ will correspond to the return map of $\mathds{T}_{[\alpha],[\iota]}$, and indeed also to $E_{\alpha,\iota}$ under the projection $w\mapsto e^{2\pi i\alpha}$. But in fact $\hat{T}_{[\alpha],[\iota]}$ is equivalent in definition to $\tilde{T}_{[-1/\alpha],[-\iota/\alpha]}$, thus completing the proof. For further details, see Section 5 of \cite{cheraghi2022arithmetic}.
\end{proof}

\section{Bi-Critical Brjuno and Herman Conditions}
\label{bicriticalbrjun}
In this section we study the modified ``bi-critical Brjuno and Herman" conditions which depend on two real parameters. In particular we prove the non-obvious result that this ``bi-critical Herman condition" can in fact be proven to be equivalent to the initial Herman condition. Recall from Section \ref{arithmeticconditions} that we use a modified arithmetic sequence arising from fundamental domain $[-1/2,1/2]$ to say that gor $\alpha\in\mathds{R}\setminus\mathds{Q}$, $\alpha\in\mathcal{B}$ ($\alpha$ is Brjuno) if and only if:
    $$
    \mathcal{B}(\alpha):= \sum_{n=0}^{\infty} \beta_n\log(1/\alpha_n) <\infty.
    $$
Recall from Section \ref{arithmeticconditions} that if $\alpha$ is Brjuno, and if for all $n\geq 0$ there is $m\geq n$ such that:
    $$
    h_{m-1}\circ ...\circ h_n(0)\geq \mathcal{B}(\alpha_m),
    $$
    where $h_n(y)=h_{r=\alpha_n}(y)$ were diffeomorphisms defined for $r\in(0,1/2]$ by:
    \[
h_r(y) =
\begin{cases}
r^{-1}\left( y - \log r^{-1} + 1 \right), & \text{if } y \ge \log r^{-1}, \\[6pt]
e^{y}, & \text{if } y \le \log r^{-1},
\end{cases}
\]
then $\alpha$ is Herman (or $\alpha\in\mathcal{H}\subset\mathcal{B}$). Recall the bi-critical weighted Brjuno sum as follows:
$$
\mathcal{B}(\alpha,\iota) = \sum_{n=0}^{\infty} \beta_n\mathcal{M}(\alpha_n,\iota_n).
$$
where we introduced the notation
$\mathcal{M}(\alpha,\iota):= \frac{1}{2}\log\frac{1}{\alpha} + \frac{1}{2}\log \frac{1}{\alpha+\iota}$ for $\alpha,\iota>0$. In particular, note that since $\frac{1}{2}\log\alpha^{-1}\leq \mathcal{M}(\alpha,\iota)\leq\log\alpha^{-1}$ we have that for all $\iota$:
\begin{equation*}
    \begin{aligned}
    &\frac{1}{2}\mathcal{B}(\alpha)\leq\mathcal{B}(\alpha,\iota)\leq \mathcal{B}(\alpha) \\
    \implies &( \mathcal{B}(\alpha) = \infty \iff \mathcal{B}(\alpha,\iota) = \infty).
    \end{aligned}
\end{equation*}
In the next section it will be shown how this modified sum appears in the model, however for now note that in Proposition \ref{height} the maximum height of each change of coordinates function will satisfy a relation that is analogous to the Brjuno sum:
$$
2\pi r y + \mathcal{M}(r,s) - C_{\ref{height}} \leq \mathrm{Im}(Y_{r,s}(x_{max}+iy)) \leq2\pi r y + \mathcal{M}(r,s) + C_{\ref{height}}.
$$
Now for $0<r\leq1/2$ and $0\leq s\leq1/2$ we recall the bi-critical Herman function introduced in Section \ref{arithmeticconditions} is as follows (see Figure \ref{hermanfunctions}):
\[ h^{-1}_{r,s}(y) = \begin{cases} 
        \frac{1}{2}\log y+\frac{1}{2}\log(\frac{s+ry}{s+r}) & 0\leq y\leq (1-s)/r ,\\
          \frac{1}{2}\log y+\frac{1}{2}(ry +\log(\frac{1}{s+r})-(1-s)) & (1-s)/r\leq y \leq 1/r ,\\
          
          ry + \mathcal{M}(r,s) - (1-\frac{1}{2}s)&  1/r \leq y.
       \end{cases}
    \]
Write $\tilde{h}_n:= h_{\alpha_n,\iota_n}$, now recall from Definition \ref{bicritdef} that $(\alpha,\iota)\in([-1/2,1/2]\setminus\mathds{Q}))\times[-1/2,1/2]$ is of bi-critical Herman type (i.e. $(\alpha,\iota)\in\mathcal{H}_2$) if $\mathcal{B}(\alpha,\iota)<\infty$ (which in particular implies $\alpha$ is Brjuno), and in addition for all $n\geq 0$ there is $m\geq n$ such that:
    $$
    \tilde h_{m-1}\circ ...\circ \tilde h_n(0)\geq \mathcal{B}(\alpha_m,\iota_m).
    $$
This bi-critical Herman condition arises naturally from the model, as will be explained in the subsequent section. 

\subsection{Proving the Bi-Critical Herman Condition is Equivalent to the Herman Condition}
Here we study the arithmetic of the bi-critical herman condition, in particular the question of if it is equivalent to the original Herman condition.

\begin{proposition}

\label{Hermanoneway}
    For all $(\alpha,\iota)\in \mathds{R}\setminus{Q}\times\mathds{R}$ we have the following:

    $$ \alpha\in\mathcal{H} \implies (\alpha,\iota)\in\mathcal{H}_2.
    $$
\end{proposition}

\begin{proof}
Note that it is clear that $h_{r,s}^{-1}(y)\leq h^{-1}_r(y)$ for all $s$, and then furthermore since $\mathcal{B}(\alpha,\iota)\leq\mathcal{B}(\alpha)$ we get for all positive integers $m$:
$$
h_{\alpha_{m},\iota_m}^{-1}(\mathcal{B}(\alpha_{m+1},\iota_{m+1}))\leq h_{\alpha_m}(\mathcal{B}(\alpha_{m+1})),
$$
and then inductively it is clear that if $\alpha$ is Herman, then for each $\alpha_n$, the $m$ required for the inverse bi-critical Herman functions to go below $0$ will be strictly less than that of the original Herman functions.
\end{proof}
Proving that $\mathcal{H}_2\implies\mathcal{H}$ is significantly more complex, but essentially boils down to showing that $h_{n+1}^{-1}\circ...\circ h^{-1}_{n+m}(\mathcal{B}(\alpha_{n+m+1}))\lesssim2\tilde{h}_{n+1}^{-1}\circ...\circ \tilde{h}^{-1}_{n+m}(\mathcal{B}(\alpha_{n+m+1},\iota_{n+m+1}))$ up to a uniform constant that does not depend on $n,m,\alpha,\iota$. The ``$m>n$" we give may be significantly larger than in the $\mathcal{H}$ case, as we do not optimise this. First we note that whenever $(\alpha,\iota)\in\mathcal{H}_2$, not only is there an $m>n$ for each $n$ such that $\tilde{h}_{n+1}^{-1}\circ...\circ \tilde{h}^{-1}_{n+m}(\mathcal{B}(\alpha_{n+m+1},\iota_{n+m+1}))<0$, but we can also take note of whether each $\tilde{h}_k^{-1}$ acts as a ``log" function, or the linear function. In particular for each $(\alpha,\iota)\in(\mathds{R}\setminus\mathds{Q})\times \mathds{R}$ and $m,n,k\in\mathds{N}$ define:
$$
B_k(n,m):=\tilde{h}_{n+k}^{-1}\circ...\circ \tilde{h}^{-1}_{n+m}(\mathcal{B}(\alpha_{n+m+1},\iota_{n+m+1}))
,$$ then $B_k(n,m)$ will either be above $(1-\iota_{n+k-1})/\alpha_{n+k-1}$, or below, leading to different behaviour for the action of $\tilde{h}_{n+k-1}^{-1}$ at this point. Similarly define 
$$
A_k(n,m):=h_{n+k}^{-1}\circ...\circ h^{-1}_{n+m}(\mathcal{B}(\alpha_{n+m+1})),
$$ where now all the Herman functions are the uni-critical versions. Note by the argument made in the ``$\Leftarrow$" portion we have 
$$
B_k(n,m)\leq A_k(n,m).
$$
We note that in the uni-critical case it is well known that if $\alpha$ is of bounded type (the continued fraction entries are bounded) then it is Herman \cite{cheraghi2022arithmetic}. The same holds here:
\begin{lemma}
    If $\alpha\in\mathds{R}\setminus{\mathds{Q}}$ is of bounded type then for all $\iota\in\mathds{R}$, $(\alpha,\iota)$ is bi-critical Herman. 
\end{lemma}
\begin{proof}
    The unicritical $h_r$ functions satisfy $h_r(y)\leq y-1$, and furthermore bounded type implies that the Brjuno sum $\mathcal{B}(\alpha_n)$ is also bounded above for all $n$.
    This follows because bounded type implies that there exists an $a>0$ such that $\alpha_n\geq a$ for all $n$, which in turn implies that (if we set ${}^n\beta_i$ to be the corresponding ``$\beta_i$" that arises in the Brjuno sum described in Section \ref{arithmeticconditions} for $\alpha_n$): 
    $$
    \mathcal{B}(\alpha_n)= \sum_{i=0}^{\infty}{}^n\beta_i\log(1/\alpha_{i+n}) \leq \log(1/a)\sum_{i=0}^{\infty}(1/2)^{i} ,
    $$
    but the right hand side of the equation is just a geometric sum which converges. This implies that $\mathcal{B}(\alpha_n,\iota_n)\leq \mathcal{B}(\alpha_n)$ is bounded above, and since $h_{r,s}(y)\leq h_r(y)$ for all $s$ we then obtain that both the Herman and bi-critical Herman conditions are satisfied.
\end{proof}
\begin{lemma}
\label{herm1}
    Suppose that $B_l(n,m)>(1-\iota_{l+n-1})/\alpha_{l+n-1}$ for all $l>k$, but at $k$ we obtain $B_k(n,m)<(1-\iota_{n+k-1})/\alpha_{n+k-1}$ with $A_k(n,m)>(1-\iota_{n+k-1})/\alpha_{n+k-1}$. Then there exists uniform constants such that:
    $$
    |h^{-1}_{n+k-1}(A_k(n,m))-\log(A_k(n,m))| <C_{\ref{herm1}},
    $$
    and
    $$
    A_l(n,m)<2B_l(n,m) + 2D_{\ref{herm1}}.
    $$
\end{lemma}
\begin{proof}
    First we claim that in this particular set up, we get:
    $$
    B_k(n,m)\leq(1-\iota_{n+k-1})/\alpha_{n+k-1} \Rightarrow  A_k(n,m)\leq 2(1-\iota_{n+k-1})/\alpha_{n+k-1}.
    $$
    This follows because as long as $B_l(n,m)$ satisfies the condition in the statement for $l>k$, then $B_l(n,m)$ is equal to a "partial Brjuno sum", in particular there exists some constant $D_{\ref{herm1}}>0$ such that:
    $$
    |B_l(n,m) -\sum_{i=l}^{n+m}\left(\prod_{t=l}^{i-1}\alpha_t\right)\mathcal{M}(\alpha_i,\iota_i)|<D_{\ref{herm1}}.
    $$
    Additionally, the same argument at $\iota=0$ yields an upper bound: 
    $$
    A_k(n,m)<\sum_{i=l}^{n+m}\prod_{t=l}^{i-1}\alpha_i\log (\alpha_i) + D_{\ref{herm1}}.
    $$ 
    So since $1/2\log(1/r)\leq\mathcal{M}(r,s)\leq\log(1/r)$ for $(r,s)\in(0,1/2]\times[0,1/2]$, we have that 
    $$
    A_l(n,m)<2B_l(n,m) + 2D_{\ref{herm1}}<2/\alpha_{n+l-1}+2D_{\ref{herm1}},
    $$
    for each $l>k$. Then at $l=k$ note that the function:
    $|\log(y)-\alpha_{k-1}y-\log(1/\alpha_{k-1})+1|$ is uniformly bounded above on the domain $y\in [1-\iota_k/\alpha_k,2/\alpha_k+2D_{\ref{herm1}}]$. This bound gives us the $C_{\ref{herm1}}$ to conclude.
\end{proof}
Now consider the following:
\begin{lemma}
    \label{loglemma}
    For ease of notation define $L_2(x):=1/2\log(x/2)$ and fix $n\geq1$. Now suppose that $y>1$ is chosen such that
    $L_2^{\circ n}(y)\geq 2 $. Then there exists a constant such that:
    $$
    |2L_2^{\circ n}(y)-\log^{\circ n}(2y)| < C_{\ref{loglemma}}.
    $$
\end{lemma}
\begin{proof}
    This follows because both logarithm functions have a ''contractivity" property as long as $y$ has been chosen such that $L_2^{\circ n}(y)\geq 2 $. In particular this implies that the derivative, which is bounded above in magnitude by $1/y$, will in particular be bounded above by $1/2$. Therefore $L_2^{\circ n}(y) $ will differ from the function $\log^{\circ n}(y)$ by a constant obtained from the geometric sum of $(1/2)^n$, yielding in particular $C_{\ref{loglemma}}=2$.
\end{proof}
\begin{lemma}
    \label{actualloglemma}
    Suppose we choose $B_k(n,m)$ such that every $B_k(n,m)=\tilde{h}_k^{-1}(B_{k+1}(n,m))$ lies in the ``logarithmic" portion of the functions $\tilde{h}_n^{-1}$. Then there exists a constant $C_{\ref{actualloglemma}}>0$ such that the following holds:
    $$
    1/2\log ^{\circ m-n-k} (\mathcal{B}(\alpha_{n+m+1},\iota_{n+m+1}))-C_{\ref{actualloglemma}}\leq B_{k}(n,m)\leq  \log ^{\circ m-n-k} (\mathcal{B}(\alpha_{n+m},\iota_{n+m+1}))+ C_{\ref{actualloglemma}}.
    $$
\end{lemma}
\begin{proof}
    In the general case the behavior of the ``logarithmic" portion of the $\tilde{h}_n^{-1}$ functions will satisfy:
    $$
    1/2\log(y)\leq\tilde{h}_n^{-1}(y)\leq\log(y).
    $$
    Combining with the above \ref{loglemma}, this then implies that since each composition will also satisfy an inequality under the compositions, we get the result with $C_{\ref{actualloglemma}}=C_{\ref{loglemma}}$.
\end{proof}
\begin{corollary}
\label{asympbound}
    Suppose the following situation does not occur:
    $$
    (B_k(n,m)\leq (1-\iota_{n+k-1})/\alpha_{n+k-1}\leq 1/\alpha_{n+k-1}\leq A_k(n,m))\cap ( 1/\alpha_{n+k-1}\leq \tilde{N}_{\ref{asympbound}}), 
    $$
    where $\tilde{N}_{\ref{asympbound}}$ is defined in the proof. Then there exists a $C_{\ref{asympbound}}>0$ such that:
    $$
    A_k(n,m) \leq 2 B_k(n,m) + C_{\ref{asympbound}}.
    $$
\end{corollary}
\begin{proof}
    When all $h$ and $\tilde{h}$ functions are in the logarithmic portion this can be concluded from Lemma \ref{actualloglemma}. Because we obtain $$
    \frac{1}{2}\log^{\circ m-n-k}(\mathcal{B}(\alpha_{n+m+1})/2)\leq\frac{1}{2}\log^{\circ m-n-k}(\mathcal{B}(\alpha_{n+m+1},\iota_{n+m+1}))\lesssim_{C_{\ref{actualloglemma}}} B_k(n,m),
    $$ which can be re-arranged to give $A_k(n,m)\leq 2 B_K(n,m) +C_{\ref{actualloglemma}} + C_{\ref{loglemma}}$ as long as we assume all values are sufficiently far above $1$ so that the logarithmic contractivity property can be used to absorb the $1/2$ factor. Because the logarithmic portion of the $h_r$ maps always lies below $1/r$, this implies that $1/\alpha_n$ is sufficiently high, say $1/\alpha_n\geq4$. 
    \par
    Conversely for the case when we are purely in the linear portion this follows directly from $\ref{herm1}$. When both cases are mixed this again follows from contractivity properties of both maps as long as $1/\alpha_{n+k}\geq4$ in all cases. To see this, note that the above two cases show that whenever we have only logarithmic, or only linear behaviour, we end up with the inequality passing through. So whenever the cases are ``mixed" we just end up at a starting point with an extra constant. As long as the maps are sufficiently contractive to halve the extra constant, this constant can be absorbed into the equation.
    \par
    This leaves the case where $(B_k(n,m)\leq (1-\iota_k)/\alpha_k\leq 1/\alpha_k\leq A_k(n,m))\cap ( 1/\alpha_k\geq \tilde{N}_{\ref{asympbound}}) $. So define $C_{\ref{asympbound}}=\max(C_{\ref{actualloglemma}},C_{\ref{herm1}})$. In this case we obtain that 
$$
1/\alpha_k \leq A_{k-1}\leq \log(2B_k+2 C_{\ref{asympbound}}),
$$
but as long as  $1/\alpha_k\geq \tilde{N}_{\ref{asympbound}}$ this implies by the contractivity of the $\log$ map that:
$A_{k-1}\leq 2B_{k-1}+C_{\ref{asympbound}}$.
\end{proof}
\begin{lemma}
    For any choice of $\alpha_n\in\mathds{R}\setminus\mathds{Q}$ and $\iota_n\in\mathds{R}$ we have that:
    $\tilde{h}_n^{-1}(y)\leq y/2$ which in particular implies that $h_n^{-1}(y)\leq 1/2y$.
\end{lemma}
\begin{proof}
    It is sufficient to prove this for $h_n(y)$ as $\tilde{h}_n\leq h_n$. On the logarithmic portion, this follows as a consequence of basic properties of the log function. Then since the linear portion has gradient that is at most equal to $1/2$, the conclusion follows as it generates a parallel line underneath $y/2$.
\end{proof}

\begin{theorem}

\label{Hermanstar}
    For all $(\alpha,\iota)\in \mathds{R}\setminus{Q}\times\mathds{R}$ we have the following equivalence:

    $$(\alpha,\iota)\in\mathcal{H}_2 \iff \alpha\in\mathcal{H}.
    $$
\end{theorem}

\begin{proof}
The reverse direction was already shown in \ref{Hermanoneway}, so we assume that $(\alpha,\iota)\in\mathcal{H}_2$. Let $n'$ be chosen such that $2^{n'} \geq \max(2\tilde N,4+C_{\ref{asympbound}})$. Now the presumption implies that there is some $m\geq n+n'$ such that $B_{n+n'}(n+n',m)\leq 0$. The result $\ref{asympbound}$ implies that as long as the special case 
$$
(B_k(n,m)\leq (1-\iota_{k-1})/\alpha_{k-1}\leq 1/\alpha_{k-1}\leq A_k(n,m))\cap ( 1/\alpha_{k-1}\leq \tilde{N}_{\ref{asympbound}}), 
$$
does not occur we will have that 
$$
A_{n+n'+1}(n+n',m)\leq 2B_{n+n'+1}(n+n',m)+C_{\ref{asympbound}} \leq C_{\ref{asympbound}} +4,
$$
which by the choice of $n'$ implies that $A_{n}(n+n',m)\leq 0$ concluding this case. On the other hand, if the special case does occur then we obtain that $A_k(n,m) \leq 2\tilde N$ for some $k\geq n+n'$ which again implies by the choice of $n'$ that $A_{n}(n+n',m)\leq 0$.
\end{proof}
\subsection{Bi-Critical Herman Condition in the Model}
In this section we justify our modified definition - the key step is just the following proposition:
\begin{proposition}
    
\label{hermy}
    There exists a $C_{\ref{hermy}}>0$ such that $\forall y\geq1$ and for all $r\in(0,1/2]$, $s\in[0,1/2]$:
        $$
        |2\pi Y_{r,s}(iy/2\pi)-h_{r,s}^{-1}(y)|\leq C_{\ref{hermy}}.
        $$
\end{proposition}

\begin{lemma}
\label{specialherm}
    Another way of writing the ``special" Herman function is:
    \[ h^{-1}_{r,s}(y) = \begin{cases} 
        \frac{1}{2}h^{-1}_r(y)+\frac{1}{2}\log(\frac{s+ry}{s+r}) & 0\leq y\leq (1-s)/r ,\\
          \frac{1}{2}h^{-1}_r(y)+\frac{1}{2}(ry +\log(1/(s+r))-(1-s)) & (1-s)/r\leq y.
       \end{cases}
    \]
    Indeed, define the ``special" function:
    \[ \tilde h^{-1}_{r,s}(y) = \begin{cases} 
        \log(\frac{s+ry}{s+r}) & 0\leq y\leq (1-s)/r, \\
          ry +\log(1/(s+r))-(1-s) & (1-s)/r\leq y.
       \end{cases}
    \]
    then 
    $$
    h^{-1}_{r,s}(y) = \frac{1}{2}(h^{-1}_r(y)+\tilde h^{-1}_{r,s}(y)).
    $$
\end{lemma}
This ``lemetta" follows directly from the definitions.
\begin{proof}[\textbf{Proof of Proposition \ref{hermy}}]
    We proceed using the $G_{r,s}-G_{r,s}(0)$ function, on the line $w=iy$:
$$
G_{r,s}(iy/(2\pi))-G_{r,s}(0)=\frac{i}{4\pi}(\log g_r (iy+1/2) + \log g_r(-s/r+1/2 +iy)) - G_{r,s}(0).
$$
Using the fact that $-G_{r,s}(0) \simeq \mathcal{M}(r,s)$, we can immediately apply Proposition 7.1 from Cheraghi in \cite{cheraghi2022arithmetic} to see that:
$$
G_{r,s}(iy/(2\pi))-G_{r,s}(0)\simeq \frac{1}{2} h_r^{-1}(y) + \frac{1}{4\pi}\log g_r(-s/r -1/2 +iy) - \frac{1}{2} \log (1/(s+r)).
$$
So by Lemma \ref{specialherm}, all we have to do in order to conclude is show that 
$$
\frac{1}{2\pi}\log g_r(-s/r -1/2 +iy) - \log (1/(s+r)) \simeq \tilde h^{-1}_{r,s}.
$$
We separate into different cases. To simplify, we prove that $\frac{1}{2\pi}\log g_r(-s/r -1/2 +iy) - \log (1/(s+r)) \simeq \log(\frac{s+ry}{s+r})$ on $1\leq y\leq1/r$, and $\frac{1}{2\pi}\log g_r(-s/r -1/2 +iy) - \log (1/(s+r)) \simeq ry +\log(1/(s+r))-(1-s)$ on $1/r\leq y<\infty$. This is enough to conclude the result because on $y\in [1/(2r),1/r]$:
$|\log(\frac{s+ry}{s+r}) - (ry +\log(1/(s+r))-(1-s))| \leq C$.
\\ \textbf{Case 1}: $1\leq  y \leq 1/r$ 
For this case note that for $1\leq  y \leq 1/r$:
$$
|e^{2\pi i s+ry }-e^{-3\pi r}| \asymp |e^{2\pi i s+ry }-1|\asymp|is+ry| \asymp s+ry,
$$
\\
which implies that 
$$
\log g_r(-s/r -1/2 +iy) \simeq \log(s+ry),
$$
and therefore 
$$
\frac{1}{2\pi}\log g_r(-s/r -1/2 +iy) - \log (1/(s+r)) \simeq \log\left(\frac{s+ry}{s+r}\right).
$$
\\ \textbf{Case 2}: $1/r \leq y < \infty$
This case turns out to be relatively trivial due to the $g_r$ lemma, which shows that for such $y$ values, $g_r(-s/r-1/2+iy)$ will be uniformly close to $e^{2\pi r y}$. Now if we consider the two functions:
$$
|G_{r,s}(iy/(2\pi)) - ry +\log(1/(s+r))-(1-s)| \simeq  |ry - \tilde{G}_{r,s}(1/2) - ry +\log(1/(s+r))-(1-s)| \simeq 0.
$$
\end{proof}
\begin{corollary}
    \label{telescopeherm}
Assume that for some integers $m > n\geq 0$, and $y \in (1, +\infty)$, the composition
$h_{\alpha_n,\iota_n}^{-1} \circ  \dots  \circ  h_{\alpha_m,\iota_m}^{-1}(y)$ is defined and is positive. Then, 
\[\big| 2\pi \mathrm{Im} \left( \mathds{Y}_{n} \circ \dots \circ \mathds{Y}_m( i  y/(2\pi))\right) - h_{\alpha_n,\iota_n}^{-1} \circ  \dots  \circ  h_{\alpha_m,\iota_m}^{-1}(y)\big| 
\leq C_{\ref{telescopeherm}}.\]
\end{corollary}
\begin{proof}
    Because both functions satisfy a kind of contractivity property, the telescoping sum that one can define can only expand the initial constant $C_{\ref{hermy}}$ by a factor defined by the geometric sum of the contractivity factor, $0.9^n$. Thus we obtain the result by setting $C_{\ref{telescopeherm}}:=10C_{\ref{hermy}}$. For more specific details see Lemma 8.3 of \cite{cheraghi2022arithmetic}.
\end{proof}
Furthermore, in order to use this fact to link together the topology of the model and the Herman condition, it is necessary to find some \textit{uniformity} results in the behavior of the function. Note that if we replace $Y_{r,s}$ in the following lemma with $Y_{r,0}=Y_r$, Cheraghi's paper \cite{cheraghi2022arithmetic} proves this in Section 7.
\begin{lemma}
\label{Herman2}
For all \( r \in (0, 1/2] \) and $s\in[0,1/2]$, we have
\begin{itemize}
    \item[(i)] \textit{there exists $C^1_{\ref{Herman2}}>0$ such that for all \( x \in [0, 1/r] \) and all \( y \geq -1 \),}
    \[
    \mathrm{Im}\, Y_{r,s}(x + iy) \geq \mathrm{Im}\, Y_{r,s}(iy) - C^1_{\ref{Herman2}};
    \]
    
    \item[(ii)] \textit{there exists $C^2_{\ref{Herman2}}>0$ such that for all \( x \in [0, 1/r] \) and all \( y_1 \geq y_2 \geq -1 \),}
    \[
    \mathrm{Im}\, Y_{r,s}(x + i y_1) - \mathrm{Im}\, Y_{r,s}(x + i y_2) 
    \leq \mathrm{Im}\, Y_{r,s}(i y_1) - \mathrm{Im}\, Y_{r,s}(i y_2) + C^2_{\ref{Herman2}};
    \]
    
    \item[(iii)] \textit{there exists $C^3_{\ref{Herman2}}>0$ such that for all \( y_1 \geq y_2 \geq -1 \) and \( y \geq 0 \),}
    \[
    \mathrm{Im}\, Y_{r,s}(i y + i y_1) - \mathrm{Im}\, Y_{r,s}(i y + i y_2) 
    \leq \mathrm{Im}\, Y_{r,s}(i y_1) - \mathrm{Im}\, Y_{r,s}(i y_2) +C^3_{\ref{Herman2}};
    \]
    
    \item[(iv)] \textit{there exists $C^4_{\ref{Herman2}}>0$ such that for all \( y_1 \geq y_2 \geq -1 \) and \( y \in [0, 5/\pi] \),}
    \[
    \mathrm{Im}\, Y_{r,s}(i y_1) - \mathrm{Im}\, Y_{r,s}(i y_2) 
    \leq \mathrm{Im}\, Y_{r,s}(i y + i y_1) - \mathrm{Im}\, Y_{r,s}(i y_1 + i y_2) + C^4_{\ref{Herman2}}.
    \]
\end{itemize}
\end{lemma}
\begin{proof}
    We prove each item separately.
\\
     \textbf{Proof of (i)}:
    \\ This result is an immediate corollary of result \ref{maxminy} proven in Section \ref{changeofcoordinates} that $\mathrm{Im}Y_{r,s}(iy)$ is always uniformly close to the minimum of the function $x\mapsto \mathrm{Im}Y_{r,s}(x+iy)$, so we can simply set $C^1_{\ref{Herman2}}=C_{\ref{maxminy}}$.
\\
\textbf{Proof of (ii)}:
    \\ We proceed by proving that there exists a constant which implies the result for $G_{r,s}$, which directly implies that the result will follow for $Y_{r,s}$ with a larger constant. Now note by the results of Cheraghi, (ii) will hold for the function $Y_r$ \cite{cheraghi2022arithmetic}. To see this we can note that from the explicit form of $g_r$
    $$
    \tilde{g}(x,y):=(g_r(x+iy))^2= (e^{2\pi r y})^2 - 2 e^{2\pi r y-3\pi r}\cos(2\pi r x) +(e^{-3\pi r})^2,
    $$
    so the derivative with respect to $y$ will be:
    $$
    \partial_y\tilde{g}(x,y)=4\pi r\, e^{2\pi r y}
\left(
e^{2\pi r y}
- e^{-3\pi r}\cos(2\pi r x)
\right),
    $$
    which is clearly always positive on $y\geq -1$, and furthermore is uniquely maximized for each $y$ at $x=0$. Now to translate this to the $G_{r,s}(w)$ function we note that because this function has a symmetry across the line $x=s/2r$, if $x\geq s/2r$ we have that the derivative of $\mathrm{Im}G_{r,s}(x+iy)$ with respect to $y$ is smaller than the derivative of the function $\log g_r(s/2r+iy-1/2) + \log g_r(x+iy-1/2-s/2r)$. But the derivative of this function is clearly smaller than the derivative of the function $\log g_r(s/2r+iy-1/2) + \log g_r(iy-1/2)$ which in turn is within a constant of $G_{r,s}(iy)$. Because the difference will now be the integral of these derivatives, this now gives the result.
\par
    \textbf{Points (iii), and (iv)} actually follow directly from:
    $$
    Y_{r,s}(w) \simeq G_{r,s}(w) = \frac{1}{2}(Y_r(w)+Y_r(w-s/r) )+C(r,s),
    $$
    and the fact that Cheraghi proved this lemma for $Y_{r}$ \cite{cheraghi2022arithmetic}. To give a quick argument as to why Cheraghi's case was true, once it is a matter of checking the derivatives with respect to $y$ of $g_r$.
\end{proof}
\section{Trichotomy in the Bi-Critical Model}
\label{trichinbi}
In this section we introduce some special ``height functions" which will determine topological and geometric features of the model. In particular this will lead to recovering a topological trichotomy for the model post-critical sets that turns out to be the same as was observed in the uni-critical case by \cite{cheraghi2022arithmetic}. The notation for this section is as referred to in Section \ref{definemodel}.
\subsection{Height Functions}
Define $b^j_n(x) = \min\{y | x+iy \in I^j_n$\}. Since $\mathds{Y}_n$ sends vertical lines to vertical lines, note that:
$$
I^j_n = \{w \in \mathds{C} \mid 0\leq \mathrm{Re}(w) \leq 1/\alpha_n,\mathrm{Im}(w)\geq b^j_n(\mathrm{Re}(w))\}.
$$
This function is continuous, and will satisfy $b_n^{j+1}\geq b^j_n$. Define $b_n : [0,1/\alpha_n] \rightarrow [-1,+\infty]$ as:
$$
b_n(x) = \lim_{j\rightarrow +\infty}b^j_n(x) = \sup_{j\geq1}b_n^j(x).
$$
In particular:
$$
I_n=\{w\in\mathds{C} \mid 0\leq \mathrm{Re}(w)\leq 1/\alpha_n,\mathrm{Im}(w)\geq b_n(\mathrm{Re}(w))\}.
$$
\begin{lemma}[Accumulation of the hairs]
    For all $n\geq -1$, we have:
    \begin{enumerate}
        \item for all $x \in [0,1/\alpha_n)$, $\liminf_{t\rightarrow x^+}b_n(s)=b_n(x)$
        \item for all $x \in [0,1/\alpha_n)$, $\liminf_{t\rightarrow x^-}b_n(s)=b_n(x)$.
    \end{enumerate}
\end{lemma}
\begin{proof}
    Since $I_n$ is the closed set above the graph of $b_n$, this implies that
\[
\forall x \in [0,1/\alpha_n), \liminf_{s \to x^+} b_n(s) \ge b_n(x),
\quad \text{and} \quad
\forall x \in (0,1/\alpha_n], \liminf_{s \to x^-} b_n(s) \ge b_n(x).
\]
If $b_n(x)=+\infty$, then we automatically have the equalities. So assume that $b_n(x)\neq\infty$. Fix $x_n \in [0,1/\alpha_n]$, and let $e_n \in \{+1,-1\}$ be arbitrary, except that $e_n=+1$ if $x_n=0$, and $e_n=-1$ if $x_n=1/\alpha_n$. Define the sequence $\{e_m\}_{m\ge n}$ according to $e_{m+1}=-\epsilon_{m+1}e_m$. To prove the equalities it is enough to show that for every $\delta>0$ there is $x_n'$ strictly between $x_n$ and $x_n+\delta e_n$ such that
\[
b_n(x_n') \le b_n(x_n)+\delta.
\]
The idea of the proof comes from fact that we can map $x_n+i b_n(x_n)\in I_n$ to $z_m\in I_m$ using the maps $\mathds{Y}_j^{-1}$, for large $m\geq n$. Then, because the functional relations, induce invariance of the heights of the $b_n$, we find
\[
z_m' = z_m + e_m \in I_m,
\]
which maps to $z_n' \in I_n$ using $\mathds{Y}_j$. It follows that $\mathrm{sign}(\mathrm{Re} z_n' - x_n)=e_n$, and by the uniform contraction of $\mathds{Y}_j$, $z_n'$ is close to $x_n+i b_n(x_n)$. However, there are additional technical difficulties due to $\epsilon_j=\pm1$ and $z_m'=z_m+e_m\notin I_m$. To see these steps see Cheraghi's original proof in Section 8.3 of \cite{cheraghi2022arithmetic}.
\end{proof}
\begin{proposition}
\label{supbn}
    For all $n\geq -1$, there exists a uniform constant $C$ for all $\alpha \in \mathds{R}\setminus{\mathds{Q}}$ and $\iota\in\mathds{R}$, and all $n\geq-1$ we have:
    $$
    |\mathcal{B}(\alpha_{n+1},\iota_{n+1})-2\pi\sup_{x \in [0,1/\alpha_n]}b_n(x)|\leq C_{\ref{supbn}}.
    $$

\end{proposition}
\begin{proof}

For $n\geq -1$ and $j\geq 0$ define
    $$
    D^j_n = \max\{b^j_n(x) | \; x  \in [0,1/\alpha_n]\}.
    $$
    First we show that the following recursive relations hold:
    $$
    2\pi\alpha_n D^{j-1}_n + \mathcal{M}(\alpha_n,\iota_n) - C_{\ref{height}} \leq 2\pi D^j_{n-1}\leq 2\pi\alpha_n D^{j-1}_n + \mathcal{M}(\alpha_n,\iota_n) + C_{\ref{height}},
    $$
    where recall the functions $x_{\max}(r,s):=(1+s)/2r-1/2$ and $\mathcal{M}(r,s)=\frac{1}{2}\log\frac{1}{r}+\frac{1}{2}\log\frac{1}{r+s}$. Since $b^j_{n-1}$ and $b^{j-1}_n$ are periodic of period +1, we may choose $x_{n-1} \in [x_{\max}(\alpha_{n-1},\iota_{n-1})-1/2,x_{\max}(\alpha_{n-1},\iota_{n-1})+1/2]$ and $x_{n} \in [x_{\max}(\alpha_{n},\iota_{n})-1/2,x_{\max}(\alpha_{n},\iota_{n})+1/2]$ such that $b^j_{n-1}(x_{n-1}) = D^j_{n-1}$ and $b^{j-1}_n(x_n) = D^{j-1}_n$. Choose $x'_n \in [0,1/\alpha_n]$ such that $-\epsilon_n\alpha_nx'_n \in x_{n-1} + \mathds{Z}$, so $x'_n \in [x_{\max}(\alpha_{n},\iota_{n})-1/2,x_{\max}(\alpha_{n},\iota_{n})+1/2]$. Apply \ref{height} with $y = b^{j-1}_n(x'_n)$ and $x=x'_n$ to obtain:
    \begin{equation*}
        \begin{aligned}
            2\pi\alpha_n D^{j-1}_n + \mathcal{M}(\alpha_n,\iota_n) = 2\pi\alpha_n b^{j-1}_n(x_n) + \mathcal{M}(\alpha_n,\iota_n) \geq &  2\pi\alpha_nb^{j-1}_n(x'_n)+\mathcal{M}(\alpha_n,\iota_n) \\
              \geq& 2\pi \mathds{Y}_n(x'_n + ib^{j-1}_n(x'_n))-C_{\ref{height}} \\
              =&2\pi b^j_{n-1}(x_{n-1})-C_{\ref{height}} \\
              =&2\pi D^j_{n-1}-C_{\ref{height}} ,
        \end{aligned}
    \end{equation*}
    here using the upper bound. Similarly for the lower bound we have:
    \begin{equation*}
        \begin{aligned}
            2\pi\alpha_n D^{j-1}_n + \mathcal{M}(\alpha_n,\iota_n) = 2\pi\alpha_n b^{j-1}_n(x_n) + \log 1/\alpha_n \leq &  2\pi\mathrm{Im}\mathds{Y}_n(x_n+b^{j-1}_n(x_n))+C_{\ref{height}} \\
              =&2\pi b^j_{n-1}(x_{n-1})+C_{\ref{height}} \\
              =&2\pi D^j_{n-1}+C_{\ref{height}} .
        \end{aligned}
    \end{equation*}
Now let $A_{n+i}(\alpha_{n+1}) = \prod_{l=1}^i\alpha_{n+l}$ for $i\geq1$, and $A_n(\alpha_{n+1}) = 1$. Then let:
    $$
    X_k = 2\pi A_{n+k}(\alpha_{n+1})D^{j-k}_{n+k} + \sum_{i=1}^kA_{n+i-1}(\alpha_{n+1}) \mathcal{M}(\alpha_{n+i},\iota_{n+i}).
    $$
    Thus:
    $2\pi D^j_n = \sum_{k=0}^{j-1}(X_k - X_{k+1}) + X_j$. Now using the bound, we have:
    \begin{equation*}
        \begin{aligned}
            |X_k - X_{k+1} |&= |A_{n+k}(\alpha_{n+1})(2\pi D^{j-k}_{n+k}-2\pi\alpha_{n+k+1}D^{j-k-1}_{n+k+1}-\mathcal{M}(\alpha_{n+k+1},\iota_{n+1+1})|\\
            &\leq A_{n+k}(\alpha_{n+1})C_{\ref{height}}.
        \end{aligned}
    \end{equation*}

    Additionally we have:
    $$
    |X_j - \sum_{i=1}^jA_{n+i-1}(\alpha_{n+1})\mathcal{M}(\alpha_{n+i},\iota_{n+i})| = |2\pi A_{n+j}(\alpha_{n+1})D^0_{n+j}| \leq 2\pi.
    $$
    Combining this together we get:
    \begin{equation*}
        \begin{aligned}
            |2\pi D^j_n -\sum_{i=1}^j A_{n+i-1}(\alpha_{n+1}) \mathcal{M}(\alpha_{n+i},\iota_{n+i})| \leq &  \sum_{k=0}^{j-1} A_{n+i-1}(\alpha_{n+1}) 4+2\pi \\
              \leq&\sum_{k=0}^{j-1}2^{-k}C+ 2\pi \leq 8C_{\ref{height}} +2\pi.
        \end{aligned}
    \end{equation*}
    By the containment of sets, $b^j_n\geq b^{j-1}_n$, which implies $D^j_n\geq D^{j-1}_n$. Therefore, for each fixed $n$ $D^j_n$ forms an increasing sequence. Hence:
    $$
    |2\pi\lim_{j\rightarrow +\infty}D^j_n - \mathcal{B}(\alpha_{n+1},\iota_{n+1})| \leq 4C_{\ref{height}} + 2\pi.
    $$
    Note that because $b^j_n\leq b_n$ and $b^j_n \rightarrow b_n$ pointwise, we must have $\sup_{x\in[0,1/\alpha_n]}b_n = \lim_{j\rightarrow+\infty}D^j_n$. 
\end{proof}

This in particular proves that the maximal heights of the $b_n$ will go to infinity if and only if $\mathcal{B}(\alpha,\iota)=\infty$. Overall this has the consequence that just as in the uni-critical case, the supremum of the base functions is infinite if and only if $\mathcal{B}(\alpha_{n+1})=\infty$, since we have the inequality:
    $$
    \frac{1}{2}\mathcal{B}(\alpha_{n+1})\leq\mathcal{B}(\alpha_{n+1},\iota_{n+1})\leq \mathcal{B}(\alpha_{n+1}).\qedhere
    $$

\subsection{When is $\mathds{A}_{[\alpha],[\iota]}$ a Jordan Curve?}
This question is equivalent to determining when the set of points defined by $b_n(x)$ is a Jordan curve, which follows if $b_n$ is continuous. To answer this question we introduce the peak functions, $p^j_n(x)$. The point of these peak functions is that they will describe the infinum of the interiors of $I_n$:
$$
p^0_n(x) := \sup_{x\in[0,1/\alpha_n]}b_n(x)+\tilde{C} \equiv \frac{1}{2\pi}\mathcal{B}(\alpha_{n+1},\iota_{n+1})+2C_{\ref{height}},
$$
where $\tilde{C}$ is just chosen such that the second constant $2C_{\ref{height}}$ is attained. Now inductively define $p^{j+1}_n$ as:
$$
p^{j+1}_n(x):= \mathrm{Im} \mathds{Y}_{n+1}(x_{n+1}+ip^{j}_{n+1}(x_{n+1})),
$$
where we choose $x_n$ and $l_n$ such that $\epsilon_{n+1}\alpha_{n+1}x_{n+1}=x_n-l_n$. In other words, the graph of $p^{j+1}_n$ is generated from applying $\mathds{Y}_n$ to the graph of $p^j_{n+1}$, and then applying translations of integers. 
\par
For all $n\geq-1$ and $j\geq0$, we have $p^{j+1}_n(x+1)=p^{j+1}_n(x)$ for $x\in[0,1/\alpha_n-1]$. Furthermore, it is clear that each $p^j_n$ will be continuous, and $p^j_n(0)=p^j_n(1/\alpha_n)$. Now by \ref{height}:
\begin{equation*}
    \begin{aligned}
        p^1_n\leq &\max_{x\in[0,1/\alpha_{n+1}]}\mathrm{Im}\mathds{Y}_{n+1}(x+i(\mathcal{B}(\alpha_{n+2},\iota_{n+2})+2C_{\ref{height}})/2\pi) \\
        \leq  &\alpha_{n+1}\mathcal{B}(\alpha_{n+2},\iota_{n+2})/2\pi+\mathcal{M}(\iota_{n+1},\alpha_{n+1})/2\pi+2\alpha_{n+1}C_{\ref{height}}/2\pi+C_{\ref{height}}
        \\\leq &\mathcal{B}(\alpha_{n+1},\iota_{n+1})/2\pi +2C_{\ref{height}} = p^0_n.
    \end{aligned}
\end{equation*}
Then by induction it is easy to see that $p^{j+1}_n(x)\leq p^j_n(x)$.
Therefore we may define:
$$
p_n(x) = \lim_{j\rightarrow\infty}p^j_n(x).
$$
On the other hand it is relatively easy to see that $p^j_n(x)\geq b^j_n(x)$ which implies $p_n(x)\geq b_n(x)$. Now it is possible to see the following statement which remarkably links the $\mathcal{H}_2$ class and the peak functions:
\begin{proposition}
\label{pb0}
    For all $(\alpha,\iota)\in(\mathds{R}\setminus\mathds{Q})\times\mathds{R}$ we have $(\alpha,\iota) \in \mathcal{H}_2$ if and only if $p_n(0) = b_n(0)=0$ for all $n\geq -1$.
\end{proposition}
\begin{proof}

    This proof follows directly in spirit from Cheraghi's result in Sections 8.5 and 8.6 of \cite{cheraghi2022arithmetic}. Note that that the $h_{\alpha_n,\iota_n}^{-1}$ functions are uniformly close to $\mathds{Y}_n(iy)$ on particular domains around the line $\mathrm{Re}(w)=0$, for $\mathrm{Im}(w)\geq 1$, and furthermore even arbitrarily long compositions where both are defined will also be close by \ref{telescopeherm}. 
    \par
    Note that the value $p_n(0)$ depends on the limit $\mathds{Y}_{n}\circ ...\circ \mathds{Y}_{n+m}(i\mathcal{B}(\alpha_m,\iota_m)/2\pi)$ as $m\rightarrow\infty$. 
    But if $(\alpha,\iota) \in \mathcal{H}_2$, it is the case that for some integer $M(n)$ we will have that $\mathds{Y}_n\circ...\circ\mathds{Y}_{M(n)}(i\mathcal{B}(\alpha_{M(n)},\iota_{M(n)})/2\pi)\leq 1$. Because this also holds for $n'\geq n$ we can pass to an arbitrarily higher level $M(n')>M(n)$ which will essentially induce arbitrarily more compositions of $\mathds{Y}_n$, but then due to the contracivity property we can conclude that $p_n(0)=0$. 
    \par
    For the other direction, note that if $p_n=0$ for all $n$, then because compositions of the $h_{\alpha_n,\iota_n}^{-1}$ depend on the compositions of the $\mathds{Y}_n$, we will get that for each fixed $n$ there is some $m(n')\geq n$ for all $n'\geq n$ such that $h_{\alpha_n',\iota_n'}^{-1}\circ...\circ h_{\alpha_{m(n')},\iota_{m(n')}}^{-1}(\mathcal{B}(\alpha_{m(n')},\iota_{m(n')})\leq 1$. But then it is clear that $h_{\alpha_{n'-1},\iota_{n'-1}}^{-1}\circ...\circ h_{\alpha_{m(n')},\iota_{m(n')}}^{-1}(\mathcal{B}(\alpha_{m(n')},\iota_{m(n')})\leq 0$, which gives the result. \qedhere
    
\end{proof}

\begin{lemma}
    For all $n\geq -1$ the function $p_n(x): [0,1/\alpha_n] \rightarrow [1,\infty)$ is continuous. Moreover if $\alpha\in\mathcal{H}_2$ then $p_n=b_n$ on $[0,1/\alpha_n]$, for all $n\geq-1$.
\end{lemma}
\begin{proof}
    To see this, we note that from the geometric inequalities in \ref{Herman2} we have that the action of the $\mathds{Y}_n$ functions on vertical lines can always be compared to the action on the line $x=0$, which will bound how $\mathds{Y}_n$ behaves along all the other lines, thus forcing $p_n=b_n$. In particular note that \ref{Herman2} implies that we can say that there is a constant $C>0$ such that:
    $$
    p_n^j(x)-p_n(x) \leq p_n^j(0)-p_n(0)+ C.
    $$
    But then by the contractivity property this implies we will be able to obtain:
    $$
    p_n^j(x)-p_n(x) \leq |p_{n+m}^{t}(0)-p_{n+m}(0)|0.9^m+ C0.9^m,
    $$
    for all $t$ and $m$ which gives uniform convergence. To see that the resultant function must satisfy $p_n=b_n$ now note \ref{Herman2} simultaneously implies that there is a $C>0$ such that
    $$
    p_n^j(x)-b_n(x) \leq p_n^j(0)-b_n(0)+ C,
    $$
    which implies in the limit that we obtain $|p_n(x)-b_n(x)|\leq C$ for all $n\geq -1$, but due to the contraction property this in fact yields $p_n(x)=b_n(x)$.
\end{proof}
\begin{lemma}
\label{bpdense}
    For all $(\alpha,\iota)\in(\mathds{R}\setminus\mathds{Q})\times\mathds{R}$, $p_n(x) = b_n(x)$ on a dense set of points for all $n\geq -1$.
\end{lemma}
\begin{proof}
    It is enough to prove that on each $I_n$ there exists an $x$ such that $p_n(x)=b_n(x)$. To find such a point, we select $x_n$ such that $x_n\in [x_{max}-1/2,x_{max}+1/2]$ for all $n$. Then by the same argument in the Brjuno height lemma, we obtain that there is some constant $C>0$ such that:
    $$
    |2\pi\mathcal{B}(\alpha_n,\iota_n) - b_n(x_n)| \leq C.
    $$
    This implies that by definition we have that $|p^{n+k}_{n+k}(x_{n+k})-b_{n+k}(x_{n+k})| \leq C $, hence:
    $$
    p^k_n(x_n) - b_n(x_n) \leq 0.9^k (C),
    $$
    which implies that $p_n(x_n)=\lim_{k\rightarrow\infty}p^k_n(x_n)=b_n(x_n)$.
\end{proof}
Recall that $\mathds{A}_{[\alpha],[\iota]}:=\partial \mathds{M}_{[\alpha],[\iota]}$:
\begin{proposition}
    For all $(\alpha,\iota)\in(\mathds{R}\setminus\mathds{Q})\times\mathds{R}$ the following are satisfied:
    \begin{itemize}
        \item The set $\mathds{A}_{[\alpha],[\iota]}$ is a Jordan curve if and only if $(\alpha,\iota)\in\mathcal{H}_2$.
        \item The boundary of the interior of $\mathds{M}_{[\alpha],[\iota]}$ is always a Jordan curve.
    \end{itemize}  
\end{proposition}
\begin{proof}
    For the first portion of the statement notice that unless $(\alpha,\iota)$ satisfy the $\mathcal{H}_2$ condition, we have that $b_n<p_n$ on a dense set of points, while due to the above statement we also have that $b_n=p_n$ on a dense set of points. Therefore the set defined by $b_n$ cannot be a Jordan curve if $(\alpha,\iota)\notin\mathcal{H}_2$, while on the other hand when $(\alpha,\iota)\in\mathcal{H}_2$ $b_n=p_n$ so it is a Jordan curve.
    \par
     Note that the Jordan curve defined by $p_n$ is the boundary of the \textbf{interior} of $I_n$. This follows because $p_n\equiv b_n$ on a dense set of points, and $p_n\geq b_n$, while the set of points strictly above the graph $x+ip_n(x)$ is an open set with boundary equal to the curve defined by $p_n$ that is clearly a subset of the interior of $I_n$. On the other hand, any open ball in $I_n$ cannot include any points on $p_n$ as since  $p_n\equiv b_n$ on a dense set of points such an open ball would necessarily intersect with $b_n$, and therefore not be contained in $I_n$. This proves the latter part of the result.
\end{proof}

\subsection{Cantor Bouquets and Hairy Jordan curves}

A \textbf{Cantor bouquet} is any subset of the plane which is ambiently homeomorphic to a set of the form:
$$
\{ re^{2\pi i\theta} \in \mathds{C} \mid 0\leq\theta\leq1,0\leq r\leq R(\theta)\},
$$
where $R: \mathds{R}/\mathds{Z}\rightarrow [0,1]$ satisfies:
\begin{itemize}
    \item $R \equiv0$ on a dense subset of $\mathds{R}/\mathds{Z}$, and $R>0$ on a dense subset of $\mathds{R}/\mathds{Z}$,
    \item for each $\theta_0\in\mathds{R}/\mathds{Z}$ we have:
    $$
    \limsup_{\theta\rightarrow\theta_o^+} R(\theta) = R(\theta_0) = \limsup_{\theta\rightarrow\theta_0^-}R(\theta).
    $$
\end{itemize}
A \textbf{one-sided hairy Jordan curve} is any subset of the plane which is ambiently homeomorphic to a set of the form 
$$
\{re^{2\pi i\theta}\in\mathds{C}\mid 0\leq\theta\leq 1,1\leq r\leq 1+R(\theta)\},
$$
where $R: \mathds{R}/\mathds{Z} \rightarrow [0,1]$ satisfies the same properties. Such sets enjoy similar features to the standard Cantor set, and under a mild additional condition (topological smoothness) they are uniquely characterised by some topological axioms \cite{AO}.
\subsection{The Trichotomy}

\begin{theorem}[Trichotomy of the maximal invariant set]
\label{trich}
    For every $(\alpha,\iota) \in \mathds{R} \setminus \mathds{Q}\times\mathds{R}$ one of the following statements hold:
\begin{itemize}
    \item[(i)] If $\mathcal{B}(\alpha,\iota)<\infty$ and $(\alpha,\iota)$ is bi-critical Herman then $\mathds{A}_{\alpha,\iota}$ is a Jordan curve,
    \item[(ii)] If $\mathcal{B}(\alpha,\iota)<\infty$ but $(\alpha,\iota)$ is not bi-critical Herman, and $\mathds{A}_{[\alpha],[\iota]}$ is a one-sided hairy Jordan curve,
    \item[(iii)] If $\mathcal{B}(\alpha,\iota)=\infty$, and $\mathds{A}_{[\alpha],[\iota]}$ is a Cantor bouquet.
\end{itemize}

\end{theorem}

\begin{proof}

We note that $\mathds{M}_{[\alpha],[\iota]}$ is defined as the image under the exponential map of $I_{-1}$, where the latter is defined as the points above $b_{-1}$. As proved before, the boundary of the interior of $I_{-1}$ when it exists is $p_{-1}$. With these two facts in mind it is clear that $\mathds{A}_{[\alpha],[\iota]}$ is a Jordan curve if and only if $(\alpha,\iota)$ is bi-critical Herman. On the other hand, when $\alpha$ is not Brjuno we have that $\mathcal{B}(\alpha,\iota)=\infty$ which implies that the $b_{-1}$ function descends to a Cantor bouquet. Finally when $\alpha$ is Brjuno but $(\alpha,\iota)$ is not bi-critical Herman we get that $\mathds{A}_{[\alpha],[\iota]}$ is stuck between the base and peak functions, and therefore it will be a one-sided hairy Jordan curve.
\end{proof}

But then because the bi-critical Brjuno and Herman conditions were found to be equivalent to the original Brjuno and Herman conditions by Theorem \ref{Hermanstar}, this immediately proves that the topological trichotomy is in fact exactly the same as the original one provided by Cheraghi for the uni-critical case \cite{cheraghi2022arithmetic}. Therefore this, combined with the basic fact seen in Section \ref{bicriticalbrjun} that $\frac{1}{2}\mathcal{B}(\alpha)\leq\mathcal{B}(\alpha,\iota)\leq\mathcal{B}(\alpha)$, proves Theorem \ref{trichA}.
\begin{corollary}
\label{thsize}
    There exists a uniform constant $C_{\ref{thsize}}>0$ such that for all $[\iota]\in\mathds{R}/\mathds{Z}$ and all $\alpha\in\mathcal{B}$:
    \begin{itemize}
        \item $\mathds{M}_{[\alpha],[\iota]} \supset B(0,C_{\ref{thsize}}^{-1}e^{- \mathcal{B}(\alpha,\iota)})$,
        \item whilst $\mathds{M}_{[\alpha],[\iota]}\supsetneq B(0,C_{\ref{thsize}}e^{-\mathcal{B}(\alpha,\iota)})$.
    \end{itemize}
\end{corollary}
\begin{proof}
    Now that we have computed that the $b_n$ functions depend on the bi-critical Brjuno function, it is a quick argument to describe the size of the Siegel disks when $B(\alpha,\iota)<\infty$. This follows from the fact the peak functions $p_{-1}^j$ that describe the boundary of the Siegel disk are strictly decreasing from the constant function $\frac{1}{2\pi}B(\alpha,\iota)+2C_{\ref{height}}$ as $j\rightarrow\infty$. On the other hand, there exist $x\in\mathds{R}$ such that $b_{-1}(x)$ attains it maximum possible value above $1/2\pi \mathcal{B}(\alpha,\iota)-C_{\ref{supbn}}$. This yields the result by taking $C_{\ref{thsize}}$ to be the larger of the exponentiation under $e^{2\pi w}$ for either constant.
\end{proof}
Should conjecture \ref{rempp} hold, this corollary points at how the connection between conformal radius and the Brjuno function studied by Yocozz and Buff-Cheritat in \cite{MR1924912,MR2233849,MR2047656} might change in multi-critical cases.

\section{Dynamical Properties of the Model}
\label{dynamprop}
In this section we collect some facts about the map $\mathds{T}_{[\alpha],[\iota]}$. In particular we see that the map on $\mathds{A}_{[\alpha],[\iota]}=\partial\mathds{M}_{[\alpha],[\iota]}$ enjoys some dynamical properties which are exactly analogous to the uni-critical model \cite{cheraghi2022arithmetic}, such as topological recurrence (and indeed the proofs are also very close). However a new insight is that the bi-critical model will also be genuinely different because there will be a very large set of conjugacy classes for each fixed $\alpha\in\mathds{R}\setminus\mathds{Q}$ which is not Herman.
\subsection{Topological Recurrence and Closed Invariant Sets}
\begin{proposition}
For every $(\alpha,\iota)\in(\mathds{R}\setminus\mathds{Q}) \times \mathds{R}$ the map $\mathds{T}_{[\alpha],[\iota]}:\mathds{A}_{[\alpha],[\iota]}\rightarrow\mathds{A}_{[\alpha],[\iota]}$ is topologically recurrent.
\end{proposition}
\begin{proof}
    We show that $\tilde T_{\alpha,\iota} : I_{-1}/\mathds{Z} \to I_{-1}/\mathds{Z}$ is topologically recurrent. To that end, fix an arbitrary $\omega_{-1} \in I_{-1}$, and assume that $(\omega_i)_{i\ge -1}$ denotes the trajectory of $\omega_{-1}$. 
\par
    The basic idea is that, dependent on whether each $\omega_i$ is in the $K_i$ or $I_i\setminus K_i$, there will be another element $\omega'_i =\omega_i+1\in I_i$ or an element $\omega_i' \in I_i$ equivalent to some composition of the $\mathds{Y}_{i+k}$ acting on $\omega_{i+k}$ such that the real part is $\mathrm{Re}(w_i+1-1/\alpha_i)$. The key point is that by the definition of the map, both of these elements will correspond to some high iterate of the $\tilde T_{\alpha,\iota}$ in $I_{-1}$. On the other hand, in $I_{i}/\mathds{Z}$, the point $\omega_i$ will lie within a uniform constant of $\omega_i'$. By contractivity, this means that we can find iterates of $\omega$ arbitrarily close to itself in $I_{-1}$. This concludes the proof. For further details see Section 9 of \cite{cheraghi2022arithmetic}.
\end{proof}
Now fix an arbitrary $y \geq 0$, and inductively define $y_n \geq 0$, for $n \geq -1$, according to
\begin{equation*}
    y_{-1} = y, \quad y_{n+1} = \mathrm{Im} \, \mathds{Y}_{n+1}^{-1}(i y_n).
\end{equation*}

For $n \geq 0$, let
\begin{equation*}
    {}^y I_n^0 = \{ w \in \mathds{C} \mid \mathrm{Re} \, w \in [0, 1/\alpha_n], \mathrm{Im} \, w \geq y_n - 1 \},
\end{equation*}
\[
{}^y J_n^0 = \{ w \in {}^y I_n^0 \mid \mathrm{Re} \, w \in [1/\alpha_n - 1, 1/\alpha_n] \}, \quad {}^y K_n^0 = \{ w \in {}^y I_n^0 \mid \mathrm{Re} \, w \in [0, 1/\alpha_n - 1] \}.
\]

In a similar manner to before, inductively obtain the sets ${}^y I_n^j$, ${}^y J_n^j$, and ${}^y K_n^j$ by successively applying the $\mathds{Y}_n$ along with appropriate translations. It is clear that ${}^y I_n^j+1\subset {}^y I_n^j$ for all $j$ and $n$, allowing the definition:
$$
{}^y I_n := \cap_{j\geq0}{}^y I_n^j.
$$
This will satisfy natural properties such as ${}^{y'} I_n \subsetneq {}^y I_n$ for $y'>y>0$, and furthermore by the uniform contraction of $\mathds{Y}_n$ which induces arbritary contraction under high iterates:
$iy \notin {}^{y'} I_n$ when $y<y'$. Recall that for all $\alpha \in \mathds{R} \setminus \mathds{Q}$, $\max(\mathds{A}_{[\alpha],[\iota]} \cap \mathds{R}) = +1$. We define $r_{\alpha,\iota} \geq 0$ according to
\[
[r_{\alpha,\iota}, 1] = \mathds{A}_{[\alpha],[\iota]} \cap [0, +\infty).
\]

If $\alpha \notin \mathcal{B}, r_{\alpha,\iota} = 0$, and if $\alpha \in \mathcal{B}, r_{\alpha,\iota} = e^{-2\pi p_{-1} (0)}$. When $\alpha \notin \mathcal{B}$ and $t \in (0, 1)$, choose $y \geq 0$ so that $t = e^{-2\pi y}$ and define
\[
{}^t \mathds{A}_{[\alpha],[\iota]} = \{ s(e^{2\pi i w}) \mid w \in {}^y I_{-1} \} \cup \{0\}.
\]

We extend this notation by setting ${}^0 \mathds{A}_{[\alpha],[\iota]} = \{0\}$. When $\alpha \in \mathcal{B}$ and $t = e^{-2\pi y}\in [r_\alpha, 1]$, we also have
\[
{}^t \mathds{A}_{[\alpha],[\iota]} = \{ s(e^{2\pi i w}) \mid w \in {}^y I_{-1}, \mathrm{Im} \, w \leq p_{-1} (\mathrm{Re} \, w) \}.
\]

For all $\alpha \in \mathds{R} \setminus \mathds{Q}$: ${}^1\mathds{A}_{[\alpha],[\iota]} = \mathds{A}_{[\alpha],[\iota]}$. Then every $r_{\alpha,\iota} \leq s < t \leq 1$, we have
\[
{}^s \mathds{A}_{[\alpha],[\iota]} \subsetneq {}^t \mathds{A}_{[\alpha],[\iota]} \quad \text{and} \quad t \notin {}^s \mathds{A}_{[\alpha],[\iota]}.
\]
\begin{proposition}
    The following holds:
    \begin{itemize}
        \item For every $(\alpha,\iota) \in (\mathds{R}\setminus\mathds{Q}) \times \mathds{R}$ and any $t\in[r_{\alpha,\iota},1]$, ${}^t\mathds{A}_{[\alpha],[\iota]}$ is fully invariant under $\mathds{T}_{[\alpha],[\iota]}$.
        \item if $\alpha\notin \mathcal{B}$ then for all $y\geq 0$, the orbit of $iy$ under $\tilde{T}_{\alpha,\iota}$ is dense in ${}^yI_{-1}$
        \item if $\alpha \in \mathcal{B}$ then for all $y$ with $0\leq y\leq p_{-1}(0)$, the orbit of $iy$ under $\tilde{T}_{\alpha,\iota}$ is dense in:
        $$
        \{ w\in {}^yI_{-1}\mid \mathrm{Im(w)}\leq p_{-1}(\mathrm{Re(w))}\}.
        $$
        
    \end{itemize}
\end{proposition}
\begin{proof}
    For the first point notice that each of the ${}^yI_{-1}$ spaces will be invariant under the functional relations as was the case with $I_{-1}$. Because the map corresponds to translating in a lower layer, then composing the change of coordinates, this implies that the map must remain in the ${}^yI_{-1}$ set, therefore they are invariant.\par
    The other two points can be solved at once by noting that as long as $z\in I_{-1}$ with $\mathrm{Im}z\leq p_{-1}(\mathrm{Re}(z))$ there will be some element in the orbit of $y$ in a neighborhood of $z$. This follows because the orbit of $y$ will follow the curves that constitute the definition of ${}^yI_{-1}$. For further details see Section 9 of \cite{cheraghi2022arithmetic}.
\end{proof}
This then leads to the following result:
\begin{proposition}
\label{closedinvariantsetss}
    For any $(\alpha,\iota)\in\mathds{R}\setminus \mathds{Q}\times\mathds{R}$ and any $t\in[r_{\alpha,\iota},1]$, $\omega(t)={}^t\mathds{A}_{[\alpha],[\iota]}$ we have that:
    \begin{itemize}
        \item if $s>t$, $s\notin \omega(t)$,
        \item the orbit of $+1$ is dense in $\mathds{A}_{[\alpha],[\iota]}$.
    \end{itemize}
    On the other hand, this classifies the closed invariant sets fully, as if $X$ is a closed invariant set of $\mathds{T}_{[\alpha],[\iota]}$ on $\mathds{A}_{[\alpha],[\iota]}$ then there is a $t\in[r_{\alpha,\iota},1]$ such that $X={}^t\mathds{A}_{[\alpha,\iota]}$.
\end{proposition}
\begin{proof}
    The first part is an immediate consequence of the prior corollary. For the latter part we must find an element in $X$ that will be the $t$ such that $X={}^t\mathds{A}_\alpha$. To do this, we take pre-images of the set $X$ under the $\mathds{Y}_n$ maps. Each pre-image will obtain a new set, $X_n\subset I_n$. Each of these sets will have a minimum height, say $\min \mathrm{Im}(X_n)$. Now take forward images under a combination of  $\mathds{Y}_n$ and translations acting on the function defined by $h^0_n(x)= \min \mathrm{Im}(X_n)-5$. By the contraction property, the sets that are obtained under this process will approximate $X_n$. There will be a limit function, $h_{-1}$ that gives the exact boundary of $X_{-1}$, and now because the set is invariant under the map, and because of the way the $h^j_n$ were defined using the changes of coordinates, we obtain that $X_{-1}$ is equal the limit of the orbit of $h_{-1}(0)$. For further details see Section 9 of \cite{cheraghi2022arithmetic}.
\end{proof}
\begin{proposition}
    There is a classification of the topologies for each ${}^t\mathds{A}_{[\alpha],[\iota]}$:
    \begin{itemize}
        \item if $\alpha\notin \mathcal{B}$, for every $t\in(r_{\alpha,\iota},1]$, ${}^t\mathds{A}_{[\alpha],[\iota]}$ is a Cantor bouquet;
        \item if $\alpha \in \mathcal{B}\setminus\mathcal{H}$, for every $t\in (r_{\alpha,\iota},1]$, ${}^t\mathds{A}_{[\alpha],[\iota]}$ is a one-sided hairy Jordan curve.
    \end{itemize}
    And in addition the map $t\mapsto {}^t\mathds{A}_{[\alpha],[\iota]}$ on $t\in [r_{\alpha,\iota},1]$ is continuous with respect to the hausdorff metric on compact subsets of $\mathds{C}$.
\end{proposition}
This proposition follows essentially word for word from the same result proven in \ref{trich} but for the overall topology, now using the ${}^y I_n$ sets instead of the $I_n$ sets, so we do not repeat it here. Interestingly, this proves that one cannot find disjoint invariant ``hedgehog" sets in the dynamical system.
\begin{corollary}
\label{linearisation}
    The subset of $I_n$ defined by $\{ x+iy\in I_n: y\geq p_n(x)$, which is equivalent under exponentiation to the set defined by $\{\omega(t): 0\leq t\leq r_{\alpha,\iota}\} $, will define a domain of linearisability for the map $(\mathds{T}_{[\alpha],[\iota]}:\mathds{M}_{[\alpha],[\iota]}\rightarrow\mathds{M}_{[\alpha],[\iota]})$.
\end{corollary}
\begin{proof}
    We note that by the above arguments $\omega(t)$ for each $0\leq t\leq r_{\alpha,\iota}$ is a Jordan curve, which hence implies at least topological linearisation along each $\omega(t)$. But then because $\mathds{T}_{[\alpha],[\iota]}$ is continuous, and because the $\omega(t)$ partition the interior of $\mathds{M}_{[\alpha],[\iota]}$,  this implies topological linearisation exists on the set itself.
\end{proof}
\subsection{Dynamical Variance and Conjugacy Classes}
\begin{proposition}
\label{swapbeta}
    For all $(\alpha,\iota)\in(\mathds{R}\setminus\mathds{Q})\times\mathds{R}$ there is a geometric symmetry between the two critical points insofar as there exists a constant $C_{\ref{swapbeta}}>0$, such that for all $k\geq 0$, we have:

        $$
        C_{\ref{swapbeta}}^{-1}|\mathds{T}_{[\alpha],[\iota]}^{\circ {k}}(+1)|< \mathds{T}_{[\alpha],[-\iota]}^{\circ k}(e^{2\pi i -\iota})<C_{\ref{swapbeta}}|\mathds{T}_{[\alpha],[\iota]}^{\circ {k}}(+1)|.
        $$

\end{proposition}
\begin{proof}
    Note that $\mathds{M}_{[\alpha],[-\iota]}$ will include the line segment $[0,1]$, and $\mathds{M}_{[\alpha],[-\iota]}$ will also include the segment $[0,1]e^{2\pi i\iota}$ as in Proposition \ref{+1beta}. Let $R_{\alpha,\iota}(x)$ for $x\in[0,1]$ denote the maximal ray lying in $\mathds{M}_{[\alpha],[\iota]}$ at the internal angle $x$. There exists a uniform constant $C>0$ such that:
    $$
    C^{-1}R_{[\alpha],[\iota]}\subset e^{2\pi i\iota}R_{[\alpha],-[\iota]}\subset C R_{[\alpha],[\iota]}.
    $$ 
    To see this, note that because of the third functional relation, (F3), each of the change of coordinates for the $([\alpha],-[\iota])$ case is the equal to the change of coordinates for $([\alpha],-[\iota])$ but translated over by $\iota$. Therefore, the only difference between the two sets, up to translation, can arise in the way the $I_n^j$ are defined. Let $p_n^{\pm}$ and $b_n^{\pm}$ denote the base and peak functions for the $[\alpha],[\pm\iota]$ case respectively. In particular, since the height and peak functions depend on the arithmetic in question, it is clear that there exists a constant such that for all $n$:
    \begin{equation*}
        \begin{aligned}
            |p_n^+(x-\iota_n/\alpha_n)-p_n^-(x)| \leq C_{\ref{swapbeta}}=e^{20\pi} &\text{ and }  |b_n^+(x-\iota_n/\alpha_n)-b_n^-(x)|\leq C_{\ref{swapbeta}}=e^{20\pi}.
        \end{aligned}
    \end{equation*}
    Note that the differences arise because $J_n^j$ in each case will lie in different places, leading to a slightly modified curve for the $I_n^j$. However, this difference will be at most $1$. Because of the contractivity property, we can then conclude that the heights of the $b_n$ and $p_n$ in each case will be bounded with a constant of $10$ of each other. Thus the equation above is true with $C_{\ref{swapbeta}}=e^{20\pi}$ by taking the exponential of a geometric sum that includes the contractivity factor of $0.9$ for the changes of coordinates.
    \par
    This extends to indefinite iterates of the map by noting that the map is generated from many compsoitions of the change of coordinates, but then rotating the $\mathds{M}_{[\alpha],[\iota]}$ over by $e^{-2\pi i\iota}$ essentially moves every point in $\mathds{M}_{[\alpha],[\iota]}$ to a point which is, up to multiplication by $C_{\ref{swapbeta}}$, arising from an analagous composition of change of coordinates.
\end{proof}
In a similar spirit:
\begin{proposition}
    Recall the notation from the previous section. For all $[\alpha]\in\mathds{R}\setminus\mathds{Q}/\mathds{Z}$ and $[\iota]\in\mathds{R}/\mathds{Z}$ there exists a number $0\leq \tilde{r}_{\alpha,\iota} \leq 1$ such that if $X$ is a closed invariant set of $\mathds{T}_{[\alpha],[\iota]}$ in $\mathds{A}_{[\alpha],[\iota]}$, then there is some $t\in [\tilde{r}_{\alpha,\iota},1]$ such that $\omega(te^{2\pi i\iota})=X$, and furthermore $\omega(e^{2\pi i\iota})=\mathds{A}_{[\alpha],[\iota]}$.
\end{proposition}
This result clarifies that the two critical points truly are interchangeable in the model. In particular it shows that the extra critical point does not introduce any new closed invariant sets, as by \ref{closedinvariantsetss} we have that for each $t\in [\tilde{r}_{\alpha,\iota},1]$ there exists some $t'\in[r_{\alpha,\iota},1]$ such that $\omega(te^{2\pi i\iota})=\omega(t')$.
\begin{proof}

    We note that the argument given in the previous section for Proposition \ref{closedinvariantsetss} can again be applied to the vertical line that passes through $x=\iota$ on $I_{-1}$. The change of coordinates acts on this vertical line in exactly the same manner as it acts upon the line at $x=0$, due to the third functional relation ($F_3$). It immediately clear that all the same arguments pass through, but it is not obvious that the constant $\tilde{r}_{\alpha,\iota}$, should be the same.
    \end{proof}
We now recall that by Theorem \ref{trich}, if $\alpha$ is fixed but $\iota$ varies, we end up with topologically identical sets, $\mathds{A}_{[\alpha],[\iota]}$ and $\mathds{A}_{[\alpha],[\iota']}$, which therefore must be homeomorphic. However, it turns out that for $\alpha\notin\mathcal{H}$, there are choices of different $\iota$ and $\iota'$ for which any homeomorphism that will preserve the order of angles (i.e. it sends the critical point at $+1$ to itself) cannot respect the dynamics:

\begin{theorem}[Non-equivalence of different $\iota$]
\label{conjclasspf}
Fix $[\iota]\in\mathds{R}/\mathds{Z}$. Then for every irrational $\alpha\notin\mathcal{H}$, there exists a dense set of $\iota'$ such that any ambient homeomorphism taking the marked point at $+1$ in $(\mathds{M}_{[\alpha],[\iota]},\mathds{T}_{[\alpha],[\iota]})$ to $+1$ in $(\mathds{M}_{[\alpha],[\iota']},\mathds{T}_{[\alpha],[\iota']})$ does not define a conjugacy. On the other hand, when $\alpha\in\mathcal{H}$, $(\mathds{M}_{[\alpha],[\iota]},\mathds{T}_{[\alpha],[\iota]})$ and $(\mathds{M}_{[\alpha],[\iota']},\mathds{T}_{[\alpha],[\iota']})$ are dynamically conjugate for all $\iota\in\mathds{R}$.
\end{theorem}
\begin{proof}
Let $b_n$, $p_n$ denote the base and peak functions that arises in the case of our fixed $\iota$, while $b_n'$ and $p'_n$ will denote the equivalent functions that arise for a specifically defined choice of $\iota'$.
\\ \textbf{Step 1}: When $\alpha\notin\mathcal{B}$.
\\
We first set $\alpha\notin \mathcal{B}$. Recall the sequence of numbers $x_{\max}(\iota_n,\alpha_n)\in [0,1/\alpha_n]$. We pick an $x'\in[0,1]$ such that $x_n'\in [x_{\max}(\alpha_n,\iota_n)-1/2,x_{\max}(\alpha_n,\iota_n)+1/2]$ for all $n$. By following the argument given in \ref{bpdense}, in particular we must have that $b_n(x'_n)=\infty$ for all $n\geq -1$. On other hand, choose $\iota'=x'$ such that in particular $\iota_n'\in[x_{\max}(\alpha_n,\iota_n)-1/2,x_{\max}(\alpha_n,\iota_n)+1/2]$ for all $n$. Now we note that due to the third functional relation, we must get that $b_n'(\iota_n'/\alpha_n)<\infty$ as the change of coordinates will roughly keep the height near $0$ here. So therefore we have found a $\iota'$ and an $x'$ such that $b_n'(x'_n)<b_n(x'_n)=\infty$. 
\par
If there was an ambient dynamical conjugacy defined sending $+1$ to itself, between $\mathds{T}_{[\alpha],[\iota]}$ and $\mathds{T}_{[\alpha],[\iota']}$, then this homeomorphism must send radial lines to radial lines with the same angle. But then this is clearly a contradiction as along the line defined by $\arg(z)=2\pi i x'$ the homeomorphism must send a line to a point. 
\\ \textbf{Step 2}: When $\alpha\in\mathcal{B}\setminus\mathcal{H}$.
\\
The argument is analogous. Similarly consider the $x'$ defined as above. It is no longer true that $b_n(x') =\infty$ (in fact the base function is now never infinity), but as explained in Lemma \ref{bpdense} it is true that $b_n(x')=p_n(x')$. However for $\iota'=x'$ then in particular it is the case that $b'_n(x')<p'_n(x')$. This follows as a consequence of the fifth relation in \ref{dsY}, which implies that the values that the function $\mathds{Y}_n$ takes at the point $x'$ is the same as at the line $x=0$, but as established in \ref{pb0}, since $\alpha\notin\mathcal{H}=\mathcal{H}_2$, then $b'_n(0)<p'_n(0)$. The reason that this establishes a dynamically distinct map, is because if there was a topological conjugacy sending $+1$ to itself between $\mathds{T}_{[\alpha],[\iota]}$ and $\mathds{T}_{[\alpha],[\iota']}$, then it would conjugate the maps along the Jordan curves defined by $p_n(x)$ and $p'_n(x)$, since it is linearisable here as per corollary \ref{linearisation}. But it would also conjugate $\omega(+1)$ in each case, sending the line $x=x'$ to $x=x'$, introducing a dynamical relationship for $\mathds{T}_{[\alpha],[\iota]}$ that didn't previously exist, thus leading to a contradiction.
\\ \textbf{Step 3}: There is a dense choice of $\iota'$.
\\
In both cases the proof easily extends to a dense set of $\iota\in[0,1/2]$ by simply allowing the $\iota_n$ to lie in any size $\alpha_n$ interval up until some $N>1$, and then imposing same above restriction by setting $\iota_n'\in [x_{\max}(\alpha_n,\iota_n)-1/2,x_{\max}(\alpha_n,\iota_n)+1/2]$ for all $n\geq N$.
\\ \textbf{Step 4}: When $\alpha\in\mathcal{H}$.
\\
Finally, to see that when $\alpha\in\mathcal{H}$ all models are topologically conjugate, we note that in this case the domain of linearisability for all cases is $\mathds{M}_{[\alpha],[\iota]}$ itself. This implies that even as $\iota'$ varies, all $(\mathds{T}_{[\alpha],[\iota']}:\mathds{M}_{[\alpha],[\iota']}\rightarrow\mathds{M}_{[\alpha],[\iota']})$ are conjugate to rotation by $\alpha$ on $\mathds{D}$, thus concluding the result.

\end{proof}
The consequence of this is that the map $\mathds{T}_{[\alpha],[\iota]}$ will have genuinely different dynamics when $\alpha$ is not Herman as $\iota$ varies. This result also immediately implies the following which in particular proves that although the higher degree uni-critical case behaves like the quadratic uni-critical case, in fact in general the bi-critical case is very different:

\begin{corollary}
For every $\alpha\notin\mathcal{H}$, there exists a $\iota$ such that $(\mathds{M}_{[\alpha],[\iota]},\mathds{T}_{[\alpha],[\iota]})$ is not conjugated to $(\mathds{M}_{[\alpha],[0]},\mathds{T}_{[\alpha],[0]})$ via any map sending $+1$ to itself.
\end{corollary}
For this corollary, just apply Theorem \ref{conjclasspf} with $[\iota]=[0]$.

\begin{remark}
    This result also indirectly relates uni-critical irrationally indifferent attractors to bi-critical irrationally indifferent attractors, as in fact it is the case that $(\mathds{M}_{[\alpha],[0]},\mathds{T}_{[\alpha],[0]})=(\mathds{M}_{[\alpha]},\mathds{T}_{[\alpha]})$, where the latter side of the equality represents the model developed for the unicritical case by Cheraghi \cite{cheraghi2022arithmetic}.
\end{remark}
To see the equivalence to Cheraghi's model for Theorem \ref{conjclass} simply note that when $[\iota]=0$ we have that all $\iota_n$ are also equal to $0$, therefore the change of coordinates at each stage is in fact equal to $Y_{r,0}=Y_r$ where $Y_r$ is the change of coordinates defined by Cheraghi. So it is not just conjugate, but in fact \textbf{equal} in the model. From this study it is clear that there is a rich complexity that arises from studying the conjugacy classes when $\alpha\notin\mathcal{H}$. 
\begin{remark}
    One might wonder if this can be optimized. However in order to answer this question adequately, it is also vital to categorize the centralisers of the $(\mathds{T}_{[\alpha,\iota]},\mathds{M}_{[\alpha,\iota]})$ (non-trivial homeomorphisms of $\mathds{M}_{[\alpha,\iota]}$ to itself which conjugate $\mathds{T}_{[\alpha,\iota]}$ to itself). Even in the uni-critical case, this is a difficult and rich question. We make a study of this in a forthcoming work \cite{jossydavoud}.
\end{remark}

\section{Geometry of Bi-Critical Irrationally Indifferent Attractors}
\label{geometryofbicritjust}
In this section we give the reasoning behind the definition of our model, and we furthermore dedicate a section to potential consequences of this model. For the sake of ease, we assume throughout that $\alpha$ and $\iota$ have been chosen such that $\alpha\in(0,1/2)\setminus\mathds{Q}$ and $\iota\in[0,1/2]$, but the arguments will also apply to the cases when $\alpha$ and $\iota$ are not necessarily in these subsets. 
\subsection{A Connection to Covering Maps}

We first give a lemma which essentially describes the geometry of uni-critical irrationally indifferent fixed points.

\begin{lemma}
\label{unicritkomp}
    Let $K_\alpha$ denote a compact class of rational maps with an irrationally indifferent fixed point at $0$ of rotation angle $\alpha$, such that in particular the second derivative of $f$ at $0$ is not equal to $0$. Then for all $f\in K_\alpha$, the only degenerate fixed point will be roughly of size $\alpha$.
\end{lemma}

\begin{proof}
    The fixed points of $f$ will be solutions to the following equation for $A\neq0$:
    $$
    f(z)-z = z((e^{2\pi i\alpha}-1) + Az + o(z^2)) = 0,
    $$
    which will have solutions at $z=0$ and $z= (e^{2\pi i\alpha}-1)/(A-o(z))$. Although the latter does not strictly count as a solution, since we are interested in the degenerating fixed points (fixed points that can be arbitrarily close to zero), we can suppose that $\alpha$ is small enough in comparison to $|A|$ that the $o(z)$ factor cannot have a large impact, therefore allowing us to conclude that the fixed point near zero will be asymptotic to $\alpha$. The function may have other fixed points, but they will not be degenerate as they will have to lie away from zero. This is because the product of all the fixed points apart from $0$ will have be asmpytotic to $(e^{2\pi i\alpha}-1)$, but since we already have one fixed point of that rough size, all the others must have modulus within a constant of $1$.
\end{proof}

For maps, $f$, with irrationally indifferent fixed points such that renormalization can be applied, there will typically exist special sectors, say $\mathcal{S}_f$, which have dynamical meaning. They are defined by a curve coming out of $0$, say $l_f$, it's image, $f(l_f)$, and another curve linking the end points of the previous two (see Figure \ref{coveringcoordinates}). In some cases such as the Inou-Shishikura class in \cite{inoushish}, it has been shown that these sectors can be chosen such that they include the critical point.
\begin{figure}[H]
\captionsetup{justification=centering}
    \centering
\includegraphics[width=0.4\textwidth]{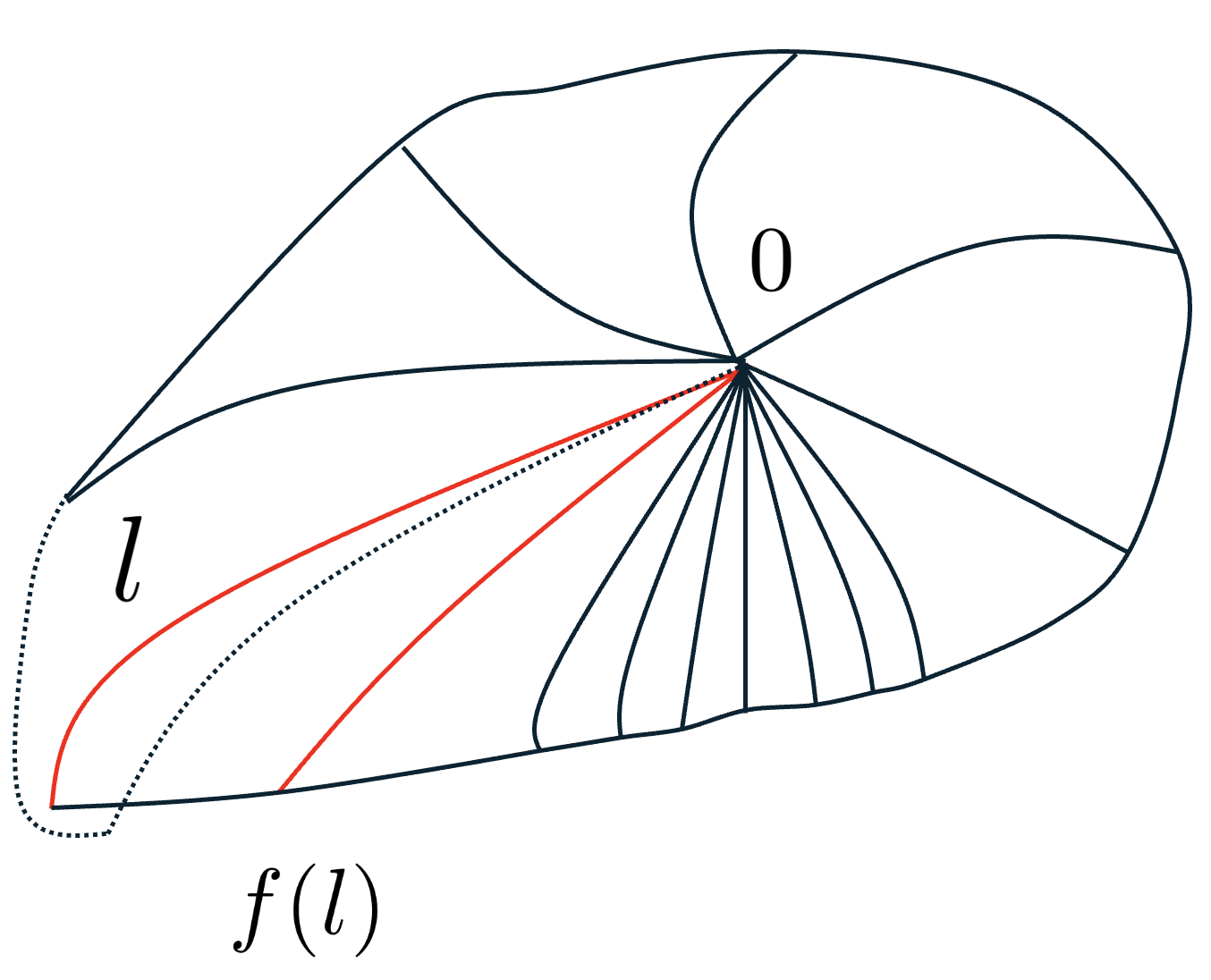}

    \caption{\label{coveringcoordinates}The set bounded by $l$ and $f(l)$ in the image represents $\mathcal{S}$. The dotted lines depict the ``return map". }
\end{figure}
In a natural fashion, these sectors can be extended by iteration of $f$ up until we reach the return map, which will happen for each point $z\in\mathcal{S}$ under $N_z\simeq 1/|\alpha|$ iterates. Let $U_f$ denote this collection of sectors. In particular, $U_f$ is special because it can be equipped with a \textbf{Fatou Coordinate}, $\phi_f:U_f\rightarrow \mathds{C}$, which will satisfy:
$$
\phi_f(f(z))=\phi_f(z)+1,
$$
for all $z\in U_f$ such that $f(z)\in U_f$ as well. Under these coordinates, the curve connecting the end points is clear as it can be defined to simply be the pre-image of a horizontal line (see Figure \ref{fatoucoordinates}).
\begin{figure}[H]
\captionsetup{justification=centering}
    \centering
\includegraphics[width=0.5\textwidth]{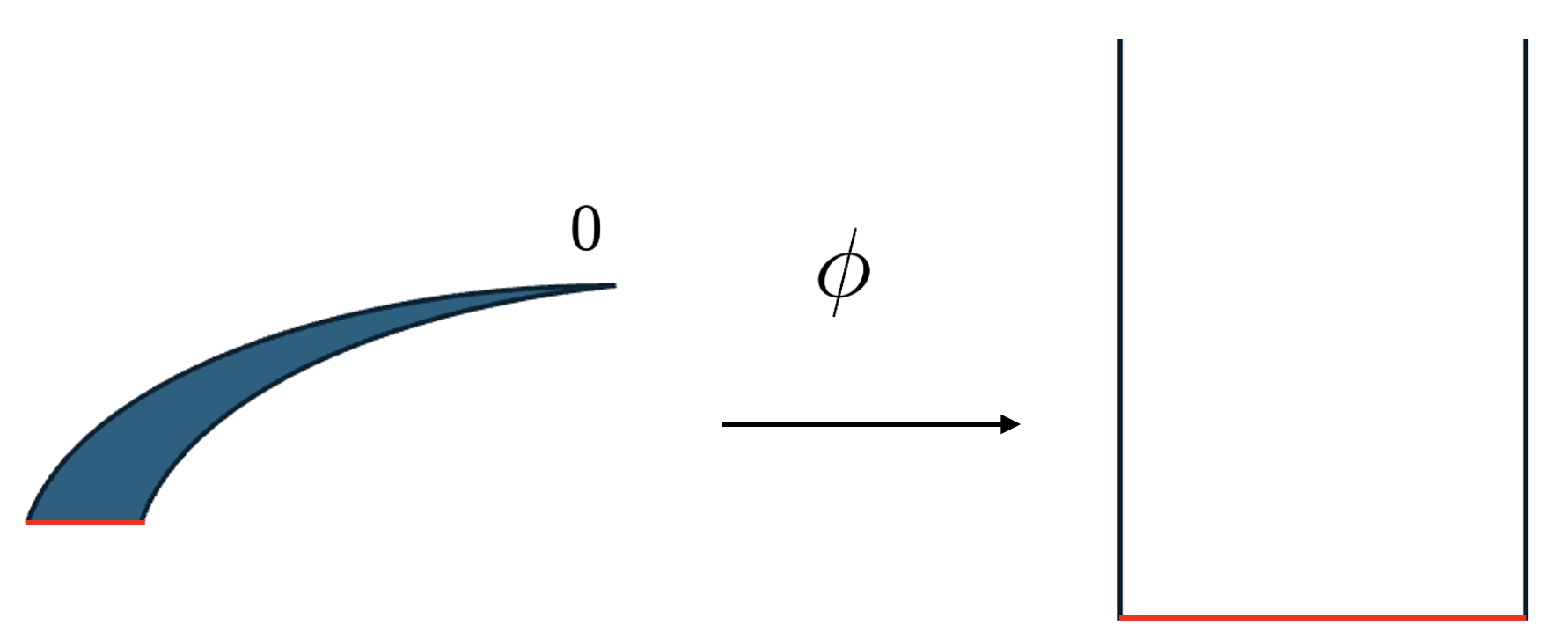}

    \caption{\label{fatoucoordinates} Fatou Coordinates}
\end{figure}
Studying $\phi_f$ will give information about $f$ itself. In particular, the functional equation means that the first roughly $1/|\alpha|$ iterates of $f$ will be equivalent to the set of numbers defined by $\phi_f^{-1}(w+n)$. In \cite{bifurmitsu}, Shishikura makes an extensive study of such coordinates. The method for studying such coordinates that prevailed in works such as \cite{MR3940907} and \cite{MR3079339} by Cheraghi was to employ quasi-conformal estimates to uniformly compare the $\phi_f^{-1}$ to a simpler covering map. This comparison rests upon the fact that these uni-critical Fatou coordinates will avoid the repelling fixed point at $\sigma_f\sim\alpha$. Define the following covering map:
$$
\tau_f(w):= \sigma_f/(1-e^{-2\pi i\alpha w}): \mathds{C} \rightarrow \hat{\mathds{C}}\setminus\{0,\sigma_f \}.
$$
There is a special class called the Inou-Shishikura class, denoted by $\mathcal{IS}$, for which the kind of sector renormalisation including the critical point, needed for the model, exists \cite{inoushish}. In particular Cheraghi showed in \cite{cheraghi2022topology} that when $f\in\mathcal{IS}$, for all $z\in\mathds{C}$ for which an inverse Fatou coordinate, $\phi_f$, normalized by $\phi_f(cp_f)=0$, is defined, we have that
    $\phi_f^{-1}(z)$ and $\tau_f(z)$ are roughly the same (see Section 5 of \cite{cheraghi2022topology} for details).
    \par
    
This remarkable insight forms the engine behind the toy model for the uni-critical case, meaning that up to $1/|\alpha|$ iterates, the modulus of orbits under $f$ will be within a constant of an orbit along the curve $x\mapsto \tau_f(x+iy)$. For higher iterates it then is vital to pass to the renormalisation, but the change of coordinates that this involves will depend exactly on the $\phi_f$ maps, and therefore be approximable by the covering maps.
\par
Finally, because the ``argument" of points does not have a dynamical meaning for $f$, we can ``straighten" this mapping in a non-conformal manner (see Figure \ref{straight}) to get a change of coordinates defined by:
$$
A_\alpha(x+iy)=\alpha x+i|\tilde \exp^{-1}\circ\tau_f(x+iy)|,
$$

which by the result given by Cheraghi is essentially the unicritical change of coordinates, $Y_\alpha$, up to a constant. The $\tau_f$ covering map depended only on the extra fixed points that $f$ has near $0$, and the rotation number of $f$. In the cubic case, there will now be two degenerate fixed points close to $0$. So to study our examples, we need to understand conformal mappings from $\mathds{C}$ onto $\mathds{C}\setminus\{0,\sigma_+,\sigma_-\}$. We can then hope to assume that the renormalisation will somehow preserve the class of maps with these ``geometric qualities", and from that the model can arise.

\begin{figure}[H]
\captionsetup{justification=centering}
    \centering
\includegraphics[width=0.2\textwidth]{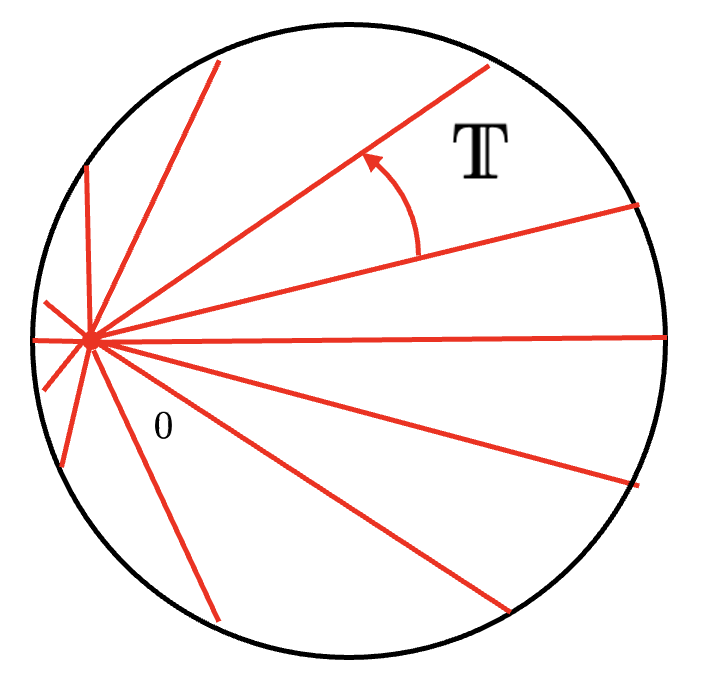}

    \caption{\label{straight}``Straightened" dynamics for the uni-critical model.}
\end{figure}

\begin{remark}
    Another method for producing a ``model" for uni-critical irrationally indifferent attractors, which was conformal in nature rather than ``straightened", was shared with myself by Buff and Cheritat in \cite{arithmetichedge}.
\end{remark}

\subsection{Studying the Bi-Critical Covering Maps}
In this section we outline why the bi-critical geometry that the model exhibits is natural. We note that since an adequate renormalisation theory does not yet exist for our examples, this section is by nature slightly imprecise. Nonetheless, there turn out to be important reasons why the model is defined the way it is.
\par
The intuition is that the iterates of the critical points for maps with a bi-critical irrationally indifferent attractor will follow a different dynamical curve (see Figure \ref{bicritfig}), arising from special covering maps, that depends on $\alpha$ and $\iota$ up until the $\lfloor1/|\alpha| \rfloor$ iterate. For further iterates we pass to the renormalisation. As we saw in the uni-critical case, the argument will depend on understanding the nature of the additional fixed points that arise for bi-critical maps, and how their covering maps degenerate.
\begin{figure}[H]

\captionsetup{justification=centering}
    \centering
\includegraphics[width=0.7\textwidth]{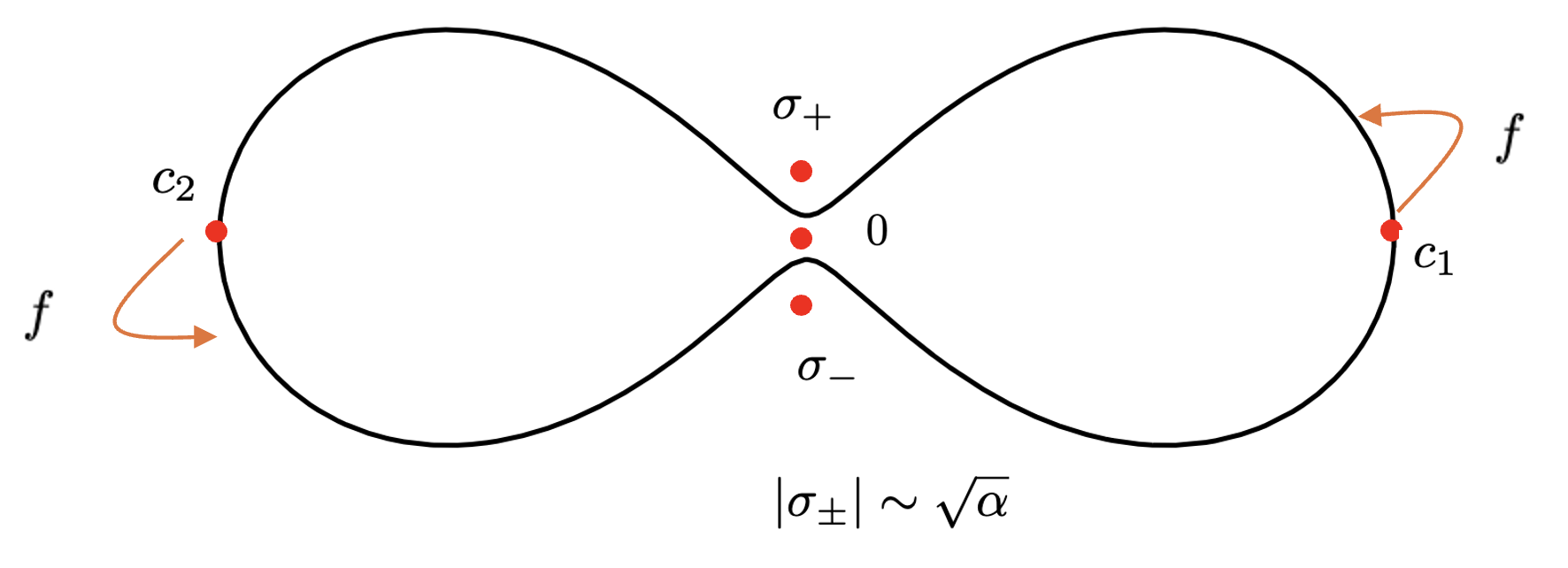}

    \caption{ \label{bicritfig} The curve describing a bi-critical orbit at $\iota=1/2$.}
 \end{figure}

The key difference here is that instead of just having one degenerate fixed point, we now must keep track of two fixed points that may get very close to $0$ depending on the parameters chosen. Some of these ``bi-critical examples" will be essentially the same as the uni-critical case as the second derivative will be away from zero, for instance the cubic uni-critical case. For these examples, Lemma \ref{unicritkomp} will apply, and the geometry of the Fatou Coordinates will still be comparable to the $\tau_f$ covering maps. Some extensions of this idea are seen in \cite{chéritat2020near} and \cite{yang2021parabolic}, in which the kind of renormalisation that works for the unicritical quadratic case was extended to arbitrary finite degree uni-critical examples.
\par
But to justify the model, we need to understand the situation where the second derivative may no longer be bounded away from $0$. The first natural place to look is the cubics, in particular let us choose the parameter space described by Zakeri in \cite{Zakeri_1999}:
$$
P_c : z \rightarrow \lambda_\alpha z (1 -\frac{1}{2}(1+\frac{1}{c})z + \frac{1}{3c}z^2).
$$
These cubics have an irrationally indifferent fixed point at $0$, and critical points at $+1$ and $c$. Every cubic with an irrationally indifferent fixed point is conjugate to at least one map of the above form. In fact each cubic will be conjugate to $2$ maps of this form in general, as there is an affine symmetry swapping the critical points. In the model this correlates with Proposition \ref{swapbeta}, which states that one can ``swap" the marking of the critical points without affecting the geometry and dynamics of the model.
\par
The first question is whether maps satisfying definition \ref{bicritdef} even exist in this parameter space. The answer is known for $\alpha$ of bounded type, as proven by Zakeri in \cite{Zakeri_1999}.
\begin{theorem}[Zakeri]
    Fix $\alpha\in\mathds{R}\setminus\mathds{Q}$ of bounded type, and let $P_c$ be as above. Then the space of $c$ values for which $\ref{bicritdef}$ is satisfied is a Jordan curve, denoted by $\mathcal{Z}_{\alpha}$.
\end{theorem}
This space was later called the Zakeri curve, $\mathcal{Z}_{\alpha}$, and it can be parameterized by the $\iota$ parameter. The properties of $\mathcal{Z}_{\alpha}$ were further explained in \cite{MR3793658}, and some results from \cite{Zakeri_1999} were later extended to entire maps in \cite{MR2654221}. For $\alpha$ of unbounded type it is unknown if $\mathcal{Z}_{\alpha}$ is a Jordan curve, however should the model be conjugate to the dynamics of $P_c$ in a sufficiently geometric manner, it could imply that the Zakeri curve is indeed a Jordan curve for all $\alpha\in\mathds{R}\setminus\mathds{Q}$. In this general setting the Zakeri curve is not well understood, but some statements can be made about the geometry:
\begin{lemma}
    Let $\alpha\in\mathds{R}\setminus\mathds{Q}$, then $\{+1,-1\}\subset\mathcal{Z}_{\alpha}\subset \mathds{A}_{1/30,30}$ (the annulus between $1/30$ and $30$). The points $+1$,$-1$ correspond to $\iota=0$ and $\iota=1/2$ respectively.
\end{lemma}
The second inclusion in this lemma essentially gives a hint of \textit{a-priori bounds} on the ratios of the two critical points (meaning one critical point cannot be either too close to $\infty$ or $0$ with respect to other).
\begin{proof}
    Firstly $c=+1$ is the uni-critical cubic, which clearly has a bi-critical irrationally indifferent attractor with $\iota=0$. For $c=-1$ we note that there is an affine symmetry of the parameter space that sends $c$ to $-1/c$ via dynamical conjugacy, in effect swapping the marking of the two critical points. This implies that $\omega(-1)$ and $\omega(+1)$ are symmetric, and since at least one of these must be the irrationally indifferent attractor, they both must be. This furthermore implies that both are recurrent to each other. The symmetry between the two points also immediately gives that $\iota=1/2$.
    \par
    For the final $\mathcal{Z}_{\alpha}\subset \mathds{A}_{1/30,30}$ statement, note that in \cite{Zakeri_1999} a simple proof that the locus of connectivity lies inside $\mathds{A}_{1/30,30}$ is given. Clearly $\mathcal{Z}_\alpha$ must lie in the locus of connectivity so we can conclude.
\end{proof}
In order to understand the geometry of orbits, we once again have to understand where degenerate fixed points will emerge. Note that each $P_c$ will have fixed points at:
$$
\sigma_{\pm}(c,\alpha) = \frac{3}{4}\left((c+1) \pm \sqrt{(c+1)^2 - \frac{8}{3}c (1-\lambda_{-\alpha})}\right),
$$
where we choose $\sqrt{z}$ to be the principle branch. 
\begin{lemma}
\label{approxfixcub}
    For any $\alpha\in\mathds{R}\setminus\mathds{Q}$, and any $c\in\mathcal{Z}_{\alpha}$ the two fixed points will satisfy the following asymptotic relations:
    \begin{equation*}
        \begin{aligned}
            |\sigma_+\sigma_-|\asymp \alpha, && \alpha \lesssim \sigma_- \lesssim \sqrt{\alpha}\lesssim\sigma_+\lesssim 1,
        \end{aligned}
    \end{equation*}
    while furthermore, whenever the second derivative of $P(z)$ is close to zero, both fixed points will lie at radial dynamical angles that are in-between the critical points.
\end{lemma}
\begin{proof}
    The first result follows from standard algebraic relations on the roots of the polynomial $P(z)-z$. For the second note that clearly $|\sigma_+|\geq|\sigma_-|$ by the choice of branch given, but then note that $\sigma_+$ is also bounded away from infinity since $c$ is, So the only way the inequality can fail is as $c$ approaches $c=-1$. But at $c=-1$ $\sigma_+=-\sigma_-$ which implies that $\sigma_+\asymp\sqrt{\alpha}\asymp\sigma_-$. Then because of the first relation we can also immediately conclude the inequality for $\sigma_-$.
    \par
    For the final portion note that when the second derivative of $P(z)$ is close to zero, $c$ will be close to $c=-1$ which furthermore implies that $\arg(\sigma_-)$ will be close to $-\arg(\sigma_+)$, and both will lie close to the line at angle $3\pi/2$ from the explicit form of both fixed points. Because the critical points will be at $+1$ and near $c=-1$ this implies that they will be at dynamical angles in between the fixed points.
\end{proof}
In order to justify the geometry of the model, an important step is to relate the geometry of the $\sigma_\pm$ to the $\iota$ parameter. We note that in fact Oudkerk developed near-parabolic Fatou coordinates for some multi-critical maps in \cite{Oudkerk}. In doing so, he was able to give an inequality on the ``amount of iterates" to pass through a close approach which will be related with the $\iota$ parameter. In our context, this gives us the following bound:
\begin{lemma}[Oudkerk]
    For each bi-critical cubic with the $\iota$ parameter sufficiently far away from $0$, we will have that:
    $$
    \sigma_{+}\lesssim \sqrt{\alpha/\iota}.
    $$
\end{lemma}
\begin{proof}
    The $\iota$ parameter is linked with the number of iterates it takes to pass from one critical point to another. Using the inequality in \cite{Oudkerk} that describes the amount of iterates to pass through a gate, the result follows immediately.
\end{proof}

For the cubic examples, this lemma forces the parameter $c$ to lie in a neighbourhood of $c=-1$, as away from this neighbourhood one of the fixed points will be bounded away from $0$. One method for studying covering maps from $\mathds{C}$ to $\mathds{C}\setminus\{0,\sigma_-,\sigma_+\}$, which Oudkerk also extensively employed, is to study holomorphic vector fields. Holomorphic vector fields are defined with the equation:
$$
\dot z(t) = Q(z),
$$
for $Q(z)$ some polynomial. They have their own version of ``Fatou Coordinates" which arise from inverting the function defined by 
$$
w\mapsto \int_{w_0}^{w} 1/Q(z),
$$
for some contour that avoids the zeroes of $Q$. This formal inverse,
$\Psi_Q: \mathds{C}\rightarrow \mathds{C}\setminus \{\text{zero set of Q}\}$ will satisfy $\Psi_Q(t)=z(t)$. A very common method to study iterative dynamics to is to compare iteration of a map, $P$, to the vector field defined by $\dot z=Q(z)=P(z)-z$, which has zeroes at all the fixed points.
\par
In order to use this strategy to study our example, we need to introduce a model situation. We already know that when $\iota$ is sufficiently close to $0$, we essentially get uni-critical dynamics, so we assume that $\iota$ is away from $0$, in which case the contents of \ref{approxfixcub} apply so both fixed points, $\sigma_-$ and $\sigma_+$ will be at approximately ``opposite" angles to one another (see Figure \ref{bicritfig}). For this reason, we now just set $\sigma_\pm=|\sigma_\pm|>0$, and assume that both fixed points lie on either side of the real line (at $\sigma_+$ and $-\sigma_-$). To represent the ``rotation by $\alpha$" portion of the covering map, we multiply by $2\pi i\alpha$ and hence define:
$$
Q(z):=2\pi i \alpha z(z-\sigma_+)(z+\sigma_-),
$$
to represent the equation for our vector field. We now need to understand how ``fast" the vector field will pass through a gate, hence we compute the integral $I=\int_{\Gamma}\frac{dz}{Q(z)}$,
where $\Gamma$ is the contour defined by the vertical line segment $z(t)=\frac{\sigma_+}{2}+i(1-t)$ for $t\in[-1,1]$. Noting, \(\sigma_+\sigma_-=\alpha\), a partial fraction decomposition yields an antiderivative of
\[
2\pi i \alpha F(z)
=
-\frac{1}{\sigma_+\sigma_-}\log z
+\frac{1}{\sigma_+(\sigma_-+\sigma_+)}\log(z-\sigma_+)
+\frac{1}{\sigma_-(\sigma_++\sigma_-)}\log(z+\sigma_-).
\]
for suitably chosen branches of $\log$. Therefore, by simplifying we obtain

$$
2\pi i \alpha I=2i\frac{2\sigma_-\arctan(1/\sigma_-)+\sigma_+(\arctan(1/\sigma_-)-\arctan(1/(\sigma_++\sigma_-))}{\sigma_+\sigma_-(\sigma_++\sigma_-)}.
$$
Since in our case $\sigma_-$ and $\sigma_+$ should be of comparable size and very small, therefore $\arctan(1/\sigma_-)-\arctan(1/(\sigma_++\sigma_-))\simeq 0$, and $\arctan(1/\sigma_-)\asymp 1$. This leaves an asymptotic expression of $I\asymp \frac{\sigma_-}{\alpha\sigma_+\sigma_-(\sigma_++\sigma_-)}$. Then we note that the residue at $0$ of our initial function, $Q(z)$, is in fact equal to $-1/(\alpha\sigma_+\sigma_-)$. Because this will be the result of integrating around the circle, we should divide by this number to obtain the ratio of total ``time" to pass through the gate. Because this ratio should be equivalent to $\iota$, this implies:
$$
\iota \asymp \sigma_-/(\sigma_++\sigma_-).
$$
Now since $\sigma_+\sigma_-\asymp \alpha$, we know that there exists a $\iota'\in(0,1/2)$ such that $\sigma_- \asymp_{\sqrt{2}} \sqrt{\alpha\iota'}$ and $\sigma_- \asymp_{\sqrt{2}}\sqrt {\alpha/\iota'}$. So then:
    \begin{equation*}
        \begin{aligned}
            \iota& \asymp \sqrt{\alpha\iota'}/(\sqrt{\alpha\iota'}+\sqrt{\alpha/\iota'} ) \\
            &= \sqrt{\iota'^2}/(1+\iota'^2)\asymp\iota',
        \end{aligned}
    \end{equation*}
    which now gives a relationship between the $\iota$ parameter, and the $\sigma_+,\sigma_-$. So the following inequalities are attained away from $\iota=0$:
\begin{equation*}
    \begin{aligned}
        |\sigma_{+}|\asymp |\alpha/\iota|^{1/2}, && |\sigma_{-}| \asymp |\alpha\iota|^{1/2}.
    \end{aligned}
\end{equation*}
Finally, in order to deal with the case at $\iota=0$, we ``paste over" by noting that the above statement also implies the following away from $\iota=0$:
\begin{equation*}
    \begin{aligned}
        |\sigma_{+}|\asymp (|\alpha|/(|\iota|+|\alpha|)^{1/2}, && |\sigma_{-}| \asymp |\alpha|(|\iota|+|\alpha|)^{1/2},
    \end{aligned}
\end{equation*}
which will extend naturally to $\iota=0$ and exhibit the uni-critical behavior we expect at $\iota=0$ with one degenerate fixed point of rough size $\alpha$. Note that these relations on the smaller $\sigma_{-}$ fixed point correspond to the $\mathcal{M}(\alpha,\iota)=\frac{1}{2}\log\left(\frac{1}{\alpha(\alpha+\iota)}\right)$ parameter described in \ref{maxminy} that arises in the bi-critical Brjuno sum once we take a logarithm.

\begin{remark}
    There may be more efficient and general methods to make such an ``angle calculation", however for the sole purpose of justifying this bi-critical model the above ``brute force" integration is sufficient.
\end{remark}

\begin{remark}
    There is an equivalent notion of Lemma \ref{unicritkomp} for the ``generalised bi-critical case", allowing some of the above arguments to be extended beyond just the cubic setting all bi-critical cases. In particular the bi-critical degenerate fixed point set up we observe can actually arise in all cases for which the \textbf{third} derivative is non-zero, but the second derivative may vary near $0$, and also away from $0$.
\end{remark}
Another aspect of justifying the model, was to study what the bi-critical Herman condition should be. For this we note that in the uni-critical examples the standard Herman condition arises in the sense that the Herman function itself, $h_r$ described in Section \ref{arithmeticconditions}, appears in the change of coordinates from the renormalisation sectors \cite{cheraghi2022arithmetic,MR4956555}. On the other hand, we note that when $\iota=1/2$, the explicit covering map that arises from the vector field defined by $\dot z=P(z)-z$ has a formal expression as:
$$
\Psi_{P(z)-z} (w) = \sqrt{\frac{\lambda_{\alpha}-1}{1-e^{2(\lambda_\alpha - 1)w}}},
$$
also studied in \cite{tanxav}. In this, we see that we should expect a ``$\sqrt{}$" type behavior here along the radial directions, thus leading to the $1/2\log(y)$ portion of the definition for $h_{r,s}(y)$ that arises at $s=1/2$. As $s$ varies toward $s=0$,  there should be a natural interpolation to the $\log(y)$ behavior exhibited in $h_r(y)$. This is exhibited in the model via the $h_{r,s}$ functions.

\subsection{Conjectures and Consequences}

\label{conjectconseq}
We now state a series of conjectures, which will be the fundamental engine for any conclusion to draw from this work:

\begin{conjecture}
    \label{rempp}
    Let $f$ be a rational map with a bi-critical irrationally indifferent attractor at $0$ which satisfies $f'(0)=e^{2\pi i \alpha}$ for $\alpha\in(\mathds{R}\setminus\mathds{Q})$, and let $c_1$ and $c_2$ be the two critical points that satisfy definition \ref{bicritdef}. Then for all $k\geq 0$ there exists a uniform constant $C>0$ such that for $c\in\{c_1,c_2\}$:
    $$
    C^{-1} |\mathds{T}^{\circ k}_{[\alpha],[\iota]}(+1)| \leq |f^{\circ k}(c)| \leq C|\mathds{T}^{\circ k}_{[\alpha],[\iota]}(+1)|,
    $$
    where $[\iota]\in\mathds{R}/\mathds{Z}$ denotes the internal angle from $c$ to the other critical point.
\end{conjecture}
Furthermore, it is believed that this geometric connection will also extend to a genuine topological conjugacy:
\begin{conjecture}
     Let $f$ be a rational map with a bi-critical irrationally indifferent attractor at $0$ which satisfies $f'(0)=e^{2\pi i \alpha}$ for $\alpha\in(\mathds{R}\setminus\mathds{Q})$, and let $c_1$ and $c_2$ be the two critical points that satisfy definition \ref{bicritdef}, and let $[\iota]\in\mathds{R}/\mathds{Z}$ denote either the internal angle from $c_1$ to $c_2$, or vice versa. Then there exists a homeomorphism $h:\omega(c_1)=\omega(c_2)\mapsto \mathds{A}_{[\alpha],[\iota]}$ which will send points along a fixed internal angle to points with argument equal to this internal angle such that $h$ will conjugate $f|_{\omega(c_1)=\omega(c_2)}$ to $\mathds{T}_{[\alpha],[\iota]}$.
\end{conjecture}
Should such conjectures hold, a consequence is:
\begin{conjecture}
    For all holomorphic maps, $f$, which have a bi-critical irrationally indifferent attractor with rotation number $\alpha\in\mathds{R}\setminus\mathds{Q}$, the set generated by the closure of the orbit of the two specified critical points will satisfy a trichotomy dependent on the arithmetic of $\alpha$:
    \begin{itemize}
        \item If $\alpha\in\mathcal{H}$, the orbits will define a Jordan curve,
        \item if $\alpha\in\mathcal{B}\setminus\mathcal{H}$, the orbits will define a one-sided hairy Jordan curve,
        \item and if $\alpha\notin\mathcal{B}$, then the orbits will define a Cantor bouquet.
    \end{itemize}
\end{conjecture}
This essentially extends a conjecture of a similar nature by Cheraghi in \cite{cheraghi2022topology} and \cite{cheraghi2022arithmetic} to include the orbits of two critical points. On the other hand, the $\iota$ parameter comes into play in the following senses:
\begin{conjecture}
    Suppose $f$ has a bi-critical irrationally indifferent attractor with rotation number $\alpha\in\mathcal{B}$, and critical points, $c_1,c_2$, for which the internal angle is equal to $\iota$. Then we have that there exists a uniform constant $C>0$ across all such maps so that:
    $$
    C^{-1}e^{- \mathcal{B}(\alpha,\iota)}<\min_{z\in\omega(c_1,c_2)}\{|z|\}<Ce^{-\mathcal{B}(\alpha,\iota)}.
    $$
\end{conjecture}
This essentially means that the ``closest approach" possible for the critical orbit is always of size roughly equal to $e^{-\mathcal{B}(\alpha,\iota)}$ for these examples.
\begin{conjecture}
    For every irrational rotation number $\alpha\notin\mathcal{H}$, there exists two distinct holomorphic maps with bi-critical irrationally indifferent attractors of rotation number $\alpha$ but distinct dynamical angles $\iota_1$ and $\iota_2$, such that the irrationally indifferent attractors are not topologically conjugate.
\end{conjecture}

\begin{conjecture}
    For every irrational rotation number $\alpha$, the equivalent uni-critical irrationally indifferent attractor studied in \cite{cheraghi2022topology,MR4956555} is conjugate to the bi-critical irrationally indifferent attractor with $\iota=0$. On the other hand, when $\alpha\notin\mathcal{H}$, there exists a choice of $\iota$ such that the corresponding bi-critical irrationally indifferent attractor is not topologically conjugate to the uni-critical.
\end{conjecture}
The question of conjugacy for these irrationally indifferent attractors raises the topic of \textit{centralisers}; maps which analytically conjugate a map to itself. In order to fully settle the problem of determining these conjugacy classes, it would first be prudent to understand exactly what topological centralisers can arise in the toy model, even for the uni-critical examples. This is studied extensively in a forthcoming work \cite{jossydavoud}.

\section{Acknowledgments}

This work was supported by the Additional Funding Programme for Mathematical Sciences, delivered by EPSRC (EP/V521917/1) and the Heilbronn Institute for Mathematical Research. I would like to thank the Heilbronn Institute for Mathematical Research for generally being very supportive for this project, and Imperial College London for hosting me throughout. I would like to also give special thanks to Davoud Cheraghi for his supervision, guidance, and much discussion. I am grateful to Xavier Buff and Arnaud Cheritat for sharing their alternative approach for developing a uni-critical model \cite{arithmetichedge} and for discussions during a visit to Toulouse. Furthermore I am thankful to Yang Fei and Igor Kravosky, my PhD examiners, who gave extensive useful comments on the version of this work that formed part of my Thesis.

The author states that there is no conflict of interest, and that there is no data set associated with this work.

\bibliographystyle{abbrv}
\bibliography{refs2}


\appendix
\blfootnote{\textit{Address:} Department of Mathematics, Imperial College London, London SW7 2AZ, UK\\
\textit{Email Address:} \href{mailto:j.russell22@imperial.ac.uk}{j.russell22@imperial.ac.uk}
}

\end{document}